\documentclass{amsart}[12pt]
\usepackage[utf8]{inputenc}
\usepackage[maxbibnames=99,
style=alphabetic,
doi=false,
giveninits=true,
]{biblatex}
\usepackage{microtype}
\usepackage[greek,english]{babel}

\usepackage{tikz,tikz-cd}

\usepackage[font=scriptsize,labelfont=bf]{subcaption}
\usepackage[font=scriptsize,labelfont=bf,margin=3cm]{caption}
\usepackage{pdflscape}
\PassOptionsToPackage{hyphens}{url}\usepackage[colorlinks, linkcolor=blue, citecolor=green,breaklinks]{hyperref}

\usepackage{graphicx}
\graphicspath{ {images/} }

\usepackage{calc}
\usepackage{cobordism}
\usetikzlibrary{fit}
\usepackage{knots}
\usepackage{amssymb}
\usepackage{amsmath, amsfonts, amsthm, bbm}
\usepackage{cancel}
\usepackage{bmpsize}
\usepackage[all]{xy}

\usepackage{bm}
\usepackage{float}
\usepackage{ stmaryrd }
\usepackage{xcolor}
\usepackage{mathtools}
\usepackage{enumerate}
\usepackage{enumitem}
\usepackage{etoolbox}
\usepackage[percent]{overpic}
\usepackage{cleveref}

\usepackage[mathscr]{euscript}

\usepackage{multirow}

\usepackage{quiver}

\DeclareSymbolFont{rsfs}{U}{rsfs}{m}{n}
\DeclareSymbolFontAlphabet{\mathscrsfs}{rsfs}

\numberwithin{equation}{section}
\theoremstyle{definition}
\newtheorem{theorem}{Theorem}[section]
\newtheorem{definition}[theorem]{Definition}
\newtheorem{lemma}[theorem]{Lemma}
\newtheorem{prop}[theorem]{Proposition}
\newtheorem{exmp}[theorem]{Example}
\newtheorem{corol}[theorem]{Corollary}
\newtheorem{remark}[theorem]{Remark}
\newtheorem{conj}[theorem]{Conjecture}

\DeclareMathOperator{\id}{id}
\DeclareMathOperator{\Hom}{Hom}

\def\normalbordisms{\scalecobordisms{1}\setlength\obscurewidth{4pt}\def\toff{0.2cm}\def\boff{0.3cm}}

\makeatletter
\def\calign@preamble{%
   &\hfil\strut@
    \setboxz@h{\@lign$\m@th\displaystyle{##}$}%
    \ifmeasuring@\savefieldlength@\fi
    \set@field
    \hfil
    \tabskip\alignsep@
}
\let\cmeasure@\measure@
\patchcmd\cmeasure@{\divide\@tempcntb\tw@}{}{}{}
\patchcmd\cmeasure@{\divide\@tempcntb\tw@}{}{}{}
\patchcmd\cmeasure@{\ifodd\maxfields@
  \global\advance\maxfields@\@ne
  \fi}{}{}{}    
\newenvironment{calign}
{%
  \let\align@preamble\calign@preamble
  \let\measure@\cmeasure@
  \align
}
{%
  \endalign
}  
\makeatother

\def\inv{{\hspace{0.5pt}\text{-} \hspace{-0.5pt}1}} 

\newcommand{\Rarrow}[1]{
    \begin{aligned}
	    \begin{tikzpicture}
		    \node [scale=0.9] at (0.5,0.75) {\ensuremath{#1}};
		    \draw [double equal sign distance, -Implies] (0,0.5) -- (1,0.5);
	    \end{tikzpicture}
	\end{aligned}
		                }

\newcommand{\RLarrow}[2]{
    \begin{aligned}
	    \begin{tikzpicture}
		    \node [scale=0.9] at (0.5,0.75) {\ensuremath{#1}};
		    \draw [double equal sign distance, -Implies] (0,0.5) -- (1,0.5);
		    \draw [double equal sign distance, -Implies] (1,0) -- (0,0);
		    \node [scale=0.9] at (0.5,-0.35) {\ensuremath{#2}};
	    \end{tikzpicture}
	\end{aligned}
		                }

\def\quotient#1#2{
    \raise1ex\hbox{$#1$}\Big/\lower1ex\hbox{$#2$}   
}                                                   

\usepackage{mathtools}                                                  
\makeatletter                                                           
\def\slashedarrowfill@#1#2#3#4#5{
  $\m@th\thickmuskip0mu\medmuskip\thickmuskip\thinmuskip\thickmuskip    
   \relax#5#1\mkern-7mu
   \cleaders\hbox{$#5\mkern-2mu#2\mkern-2mu$}\hfill                     
   \mathclap{#3}\mathclap{#2}
   \cleaders\hbox{$#5\mkern-2mu#2\mkern-2mu$}\hfill                     
   \mkern-7mu#4$
}                                                                       
\def\rightslashedarrowfill@{
  \slashedarrowfill@\relbar\relbar\mapstochar\rightarrow}               
\newcommand\xslashedrightarrow[2][]{
  \ext@arrow 0055{\rightslashedarrowfill@}{#1}{#2}}                     
\makeatother                                                            

\tikzset{tiny label/.style={
        shape=circle,
        text=black,
        draw=black,
        line width=1,
        fill=white,
        inner sep=0.05cm,
        minimum width=0.25cm,
        font=\small}}

\tikzset{tiny red label/.style={
        shape=circle,
        text=red,
        draw=red,
        line width=0.5,
        fill=white,
        inner sep=0.05cm,
        minimum width=0.25cm,
        font=\small}}

\newenvironment{tz}{\begin{aligned} \begin{tikzpicture}}{\end{tikzpicture} \end{aligned}}

\newcommand{\C}{\mathcal{C}}
\newcommand{\cP}{\mathcal{P}}

\newcommand{\cM}{\mathcal{M}}
\newcommand{\coend}{\mathcal{F}}
\newcommand{\Znc}{\mathcal{Z}}

\newcommand{\interchangor}{\text{int}}
\newcommand{\phileft}{\phi_l}
\newcommand{\phiright}{\phi_r}

\tikzset{double arrow scope/.style={every path/.style={double equal sign distance, -Implies}}}

\newcommand\I{\ensuremath{\mathrm{I}}}
\newcommand\II{\ensuremath{\mathrm{II}}}
\newcommand\III{\ensuremath{\mathrm{III}}}

\colorlet{LightGray}{black!15}
\colorlet{LightRed}{red!15}
\colorlet{LightBlue}{blue!30}
\colorlet{LightOrange}{orange!15}
\colorlet{DarkGreen}{green!50!black!50!}
\colorlet{LightGreen}{green!45}
\colorlet{VLG}{black!5}
\colorlet{VLB}{black!10}

\def\proarrow{\relbar\joinrel\mapstochar\joinrel\rightarrow}

\newcommand{\Bordor}{\ensuremath{\text{Bord}^{or}_{3,2,1}}}
\newcommand{\Bord}{\ensuremath{\text{Bord}^{or,nc}_{3,2,1}}}
\newcommand{\Bordncp}{\ensuremath{\text{Bord}^{p_1,nc}_{3,2,1}}}
\newcommand{\Bordsig}{\ensuremath{\text{Bord}^{sig}_{3,2,1}}}
\newcommand{\Bordncsig}{\ensuremath{\text{Bord}^{sig,nc}_{3,2,1}}}

\newcommand{\bBord}{\ensuremath{\textbf{Bord}^{or,nc}_{3,2,1}}}
\newcommand{\bBordncp}{\ensuremath{\textbf{Bord}^{p_1,nc}_{3,2,1}}}

\newcommand{\bBordncsig}{\ensuremath{\textbf{Bord}^{sig,nc}_{3,2,1}}}

\newcommand{\Vect}{\ensuremath{\text{Vect}}}
\newcommand{\Vecfd}{\ensuremath{\text{Vec}_k}}
\newcommand{\unit}{\ensuremath{\mathbbm{1}}}


\def\gap{\hspace{5pt}}

\tikzset{bot=true}

\tikzset{
    partial ellipse/.style args={#1:#2:#3}{
        insert path={+ (#1:#3) arc (#1:#2:#3)}
    }
}

\def\gpdashlength{0.4*\pgflinewidth}                          
\tikzset{                                                     
         gp path/.style={dash pattern=on 7.5*\gpdashlength    
                                     off 7.5*\gpdashlength}}  

\newcommand{\tinypants}{\hspace{0.2mm}\raisebox{-0.75mm}{\ensuremath{  
\begin{tikzpicture}[scale=0.1]                                         
		\draw (2.5,3) ellipse (1cm and 0.25cm);                        
		\draw [gp path] (1,0) [partial ellipse=0:180:1cm and 0.25cm];  
		\draw (1,0) [partial ellipse=180:360:1cm and 0.25cm];          
		\draw [gp path] (4,0) [partial ellipse=0:180:1cm and 0.25cm];  
		\draw (4,0) [partial ellipse=180:360:1cm and 0.25cm];          
		\draw (0,0) to [out=90, in=270, looseness=1] (1.5,3);          
		\draw (5,0) to [out=90, in=270, looseness=1] (3.5,3);          
		\draw (2,0) to [out=90, in=90, looseness=2.5] (3,0);           
\end{tikzpicture}                                                      
}                                                                      
}\hspace{-1mm}                                                         
}                                                                      

\newcommand{\tinycopants}{\hspace{0.2mm}\raisebox{-0.75mm}{\ensuremath{
\begin{tikzpicture}[scale=0.1]                                         
		\draw (1,3) ellipse (1cm and 0.25cm);                          
		\draw (4,3) ellipse (1cm and 0.25cm);                          
		\draw [gp path] (2.5,0) [partial ellipse=0:180:1cm and 0.25cm];
		\draw (2.5,0) [partial ellipse=180:360:1cm and 0.25cm];        
		\draw (1.5,0) to [out=90, in=270, looseness=1] (0,3);          
		\draw (3.5,0) to [out=90, in=270, looseness=1] (5,3);          
		\draw (2,3) to [out=270, in=270, looseness=2.5] (3,3);         
\end{tikzpicture}                                                      
}                                                                      
}\hspace{-1mm}                                                         
}                                                                      

\newcommand{\tinycap}{\hspace{0.2mm}\raisebox{-0.75mm}{\ensuremath{    
\begin{tikzpicture}[scale=0.18]                                        
		\draw (0,0) to [out=90, in=90, looseness=3] (2,0);             
		\draw [gp path] (1,0) [partial ellipse=0:180:1cm and 0.25cm];  
		\draw (1,0) [partial ellipse=180:360:1cm and 0.25cm];          
\end{tikzpicture}                                                      
}                                                                      
}\hspace{-1mm}                                                         
}                                                                      

\newcommand{\tinycup}{\hspace{0.2mm}\raisebox{-1.9mm}{\ensuremath{     
\begin{tikzpicture}[scale=0.18]                                        
		\draw (0,3) to [out=270, in=270, looseness=3] (2,3);           
		\draw (1,3) ellipse (1cm and 0.25cm);                          
\end{tikzpicture}                                                      
}                                                                      
}\hspace{-1mm}                                                         
}                                                                      

\newcommand*\xcobpos[2]{#1*\the\cobwidth,#2*\the\cobheight}

\tikzset{knot/clip radius=\obscurewidth}
\tikzset{knot/clip width=0.1*\obscurewidth, knot/end tolerance=2pt}

\tikzset{xmorphismlabel/.style={draw=red, thin, circle, inner sep=-100pt, minimum width=5pt, fill=white, font=\tiny}}

\tikzset{morphlabel/.style={draw=black, thin, rectangle, minimum width=7pt, fill=white, font=\scriptsize}}

\tikzset{morphismlabel/.style={draw=red, thin, circle, inner sep=-1pt, minimum width=7pt, fill=white}}

\tikzset{above strand label/.style={yshift=0.4\cobwidth, anchor=mid, font=\footnotesize}}
\tikzset{below strand label/.style={yshift=-0.4\cobwidth, anchor=mid, font=\footnotesize}}

\def\myred{red}
\def\myblue{blue}
\def\mygreen{green}
\def\mypurple{purple}
\def\mybrown{brown}
\def\myblack{black}

\def\colourA{\myred}
\def\colourB{\myblue}
\def\colourC{\mygreen}
\def\colourD{\mypurple}
\def\colourE{\mybrown}
\def\colourF{\myblack}

\tikzset{style 1/.style={draw=white, double distance=\cobordismlinewidth, line width=0.5\obscurewidth, double=\colourA}}
\tikzset{style 2/.style={draw=white, double distance=\cobordismlinewidth, line width=0.5\obscurewidth, double=\colourB}}
\tikzset{style 3/.style={draw=white, double distance=\cobordismlinewidth, line width=0.5\obscurewidth, double=\colourC}}
\tikzset{style 4/.style={draw=white, double distance=\cobordismlinewidth, line width=0.5\obscurewidth, double=\colourD}}
\tikzset{style 5/.style={draw=white, double distance=\cobordismlinewidth, line width=0.5\obscurewidth, double=\colourE}}

\tikzset{red strand/.style={draw=white, double distance=\cobordismlinewidth, line width=0.5\obscurewidth, double=\colourA}}
\tikzset{red strand blue back/.style={draw=\outermorphismcolor, double distance=\cobordismlinewidth, line width=0.5\obscurewidth, double=\colourA}}
\tikzset{red strand no back/.style={draw=\colourA}}
\tikzset{blue strand/.style={draw=white, double distance=\cobordismlinewidth, line width=0.5\obscurewidth, double=\colourB}}
\tikzset{green strand/.style={draw=white, double distance=\cobordismlinewidth, line width=0.5\obscurewidth, double=\colourC}}
\tikzset{purple strand/.style={draw=white, double distance=\cobordismlinewidth, line width=0.5\obscurewidth, double=\colourD}}
\tikzset{brown strand/.style={draw=white, double distance=\cobordismlinewidth, line width=0.5\obscurewidth, double=\colourE}}
\tikzset{black strand/.style={draw=white, double distance=\cobordismlinewidth, line width=0.5\obscurewidth, double=\colourF}}
\tikzset{invisible strand/.style={draw=none}}

\newcommand\templabel[1]{}

\begin{document}

\title{Non-compact 3d TQFT and non-semisimplicity}
\author{Theodoros Lagiotis}
\date{2025}

\maketitle

\begin{abstract}
We define a once extended non-compact 3-dimensional TQFT $\Znc$ from the data of a (potentially) non-semisimple modular tensor category. This is in the framework of generators and relations of \cite{bartlett_modular_2015}, having disallowed generating 2-morphisms whose source is the empty. Moreover, we show that the projective mapping class group representations this TQFT gives rise to, are dual to those of \cite{lyubashenko_invariants_1995} and \cite{de_renzi_mapping_2023}. We develop a method to decompose a closed 3-manifold in terms of 2-morphism generators. We use this to compute the value of $\Znc$ on 3-manifolds, explaining why it should recover Lyubashenko's 3-manifold invariants \cite{lyubashenko_invariants_1995}. Finally, we explain that the value of the non-compact TQFT on the solid torus recovers the data of a modified trace \cite{geer_modified_2009}.

\end{abstract}

\tableofcontents

\newpage
\pagenumbering{arabic} 

\section{Introduction}

In this thesis, we define a once extended non-compact 3-dimensional TQFT $\Znc$ from the data of a (potentially) non-semisimple modular tensor category $\C$, such that $\Znc(S^1)=\C$.  This is done in a framework of generators and relations, akin to \cite{bartlett_modular_2015}.  In Section \ref{modpres} we define the (conjectural) presentation of the cobordism bicategory we use to define the TQFT. This complies with Lurie's approach in \cite[Def 4.2.10]{lurie_classification_2009}, where the 2-morphisms having connected components with empty intersection with the source are discarded, hence the term `non-compact'. This work is closely related to the constructions of Kerler-Lyubashenko \cite{kerler_non-semisimple_2001} and De Renzi et al. \cite{de_renzi_3-dimensional_2022,de_renzi_extended_2021}. It should be loosely thought of as translating the above in an undecorated, 2-category setting. This is the non-semisimple analog of what \cite{bartlett_modular_2015} is to the Reshetikhin-Turaev TQFT \cite{turaev_quantum_2020} in the semisimple case. 

In section \ref{construction} we prove the following:

\begin{theorem}[\ref{main_th}]\label{main_th_intro}
    The assignment $\Znc$, defined in section \ref{genass}, is a well defined symmetric monoidal 2-functor.
\end{theorem}

Assuming Conjecture \ref{conjecture}, we have:

\begin{corol}
    $\Znc$ is a non-compact 3d TQFT.
\end{corol}


We believe that a proof of Conjecture \ref{conjecture} will be straightforward once the presentation of \cite{bartlett_modular_2015} for the full 3d cobordism bicategory is verified. See \cite{sytilidis_presentations_2025} for the current progress - a proof for the presentation of the 3d cobordism bicategory, whose 2-morphisms are restricted to be invertible. One way to view Theorem \ref{main_th_intro} is as providing independent evidence for Conjecture \ref{conjecture}.

The perks of our approach are the following. First of all, the cobordism category is promoted to a 2-category, which takes away any need for decorations and more complicated definitions. This in turn, relates to the fact that it is easier to talk about `classification' of TQFTs. This is something we do not address in this thesis, but intend to consider in future work.

In addition, we tackle questions related to invariants associated to $\Znc$, providing further evidence for Conjecture \ref{conjecture}.
We compute the values of 2-morphisms meant to correspond to mapping class group (MCG) generators. From there, we relate the (would be) projective mapping class group representations obtained from this construction to those of \cite{lyubashenko_invariants_1995} and \cite{de_renzi_mapping_2023}. More specifically, if we call their (equivalent) representations $\bar{\rho}_X$, we have the following:

\begin{theorem}[\ref{mapping-class-th}]
    If $\bar{\rho}_{nc}$ is the projective mapping class group representation obtained from the TQFT $\mathcal{Z}$, then $\bar{\rho}_{nc}\cong\bar{\rho}_X^*$.
\end{theorem}

For a closed oriented 3-manifold $M$ we can compute the value of $\Znc$ on $M^*:=\left(M\setminus(D^3\sqcup D^3)\right)$. If $\mathcal{L}(M)$ is the Lyubashenko invariant of $M$ \cite{lyubashenko_invariants_1995}, using a decomposition of $M$ in terms of 2-morphism generators we show that:

\begin{theorem}[\ref{invariants-th}]
    $\Znc(M^*)=(1/\mathscrsfs{D})\cdot\mathcal{L}(M)\cdot\id_{\Znc^\prime(S^2)}$.
    
    In the case where $\C$ is semisimple, $\Znc(M^*)=\mathcal{RT}(M)\cdot\id_{\Znc^\prime(S^2)}$, the Reshetikhin-Turaev invariant of $M$ \cite{reshetikhin_invariants_1991}.
\end{theorem}

Lastly, again following ideas from \cite{lurie_classification_2009}, the solid torus should give rise to a trace map, part of the `Calabi-Yau' structure on $\mathcal{Z}(S^1)$. 

\begin{theorem}[\ref{iso_exist_canon},\ref{mod-tr-thm}]
    Denoting $\cP:=\text{Proj}(\C)$, the full subcategory of projectives in $\C$,
    \begin{itemize}
        \item There exists a canonical isomorphism $$c_t\colon\displaystyle\int^{P\in\cP} \hspace{-20pt} \Hom_\C(P,P)\xrightarrow{\sim}\Znc(T^2)$$
        \item The map $t^\prime:=\Znc(S^1\times D^2)\circ c_t$ is a modified trace.
    \end{itemize}

\end{theorem}

\subsection{Survey of previous literature}

A quantum field theory is called topological if its action functional has no dependence on the spacetime metric \cite{witten_topological_1988}. The mathematical formalization of QFT is a notoriously difficult problem. However, the class of topological QFTs proves to be simple enough to allow for a mathematical definition.
Atiyah in \cite{atiyah_topological_1988}, following Segal's work \cite{tillmann_definition_2004}, proposed a definition of TQFT as a symmetric monoidal functor from a `geometric' cobordism category, to an `algebraic' category of vector spaces. This definition has undergone several generalizations since its first appearance. 
Given the above, an interesting question arises naturally: Are there examples of physical TQFTs that can be adapted to the mathematical formalism?

One of the first and most important examples of this bridging of the two notions came with the foundation of the area of quantum topology. In his seminal work \cite{witten_quantum_1989}, Witten recovered the Jones polynomial of a link, through Chern-Simons theory. The input data for the 3d Chern-Simons TQFT is a connected, compact Lie group $G$, and a level (which is an integer for simple $G$). Soon after, Reshetikhin and Turaev were able to obtain a mathematically rigorous construction of the associated 3-manifold invariants \cite{reshetikhin_ribbon_1990,reshetikhin_invariants_1991}, and using those, to define a mathematical TQFT \cite{turaev_quantum_2020}. Importantly, their construction uses the same input data of $G$ and a level. This is in the form of the category of representations $\C(G,q)$ of the corresponding quantum group at a root of unity $q$, determined by the level. In fact, the necessary algebraic input data for the Reshetikhin-Turaev (RT) construction is that of a semisimple modular tensor category. Objects and morphisms of this category can be used to decorate the cobordism category used in the RT construction.


Extending this construction to include non-semisimple MTCs has been a problem of interest since its appearance. This is interesting from a mathematical perspective of generalizing the result, but also from a physical one. Rational vertex operator algebras (VOAs) are not the only types of VOAs that arise in conformal field theory (CFT). Another interesting class is that of logarithmic VOA's, whose categories of representations are not semisimple. Some references for the connection to logarithmic CFT are \cite{fuchs_hopf_2011,fuchs_non-semisimple_2013,fuchs_coends_2016,fuchs_consistent_2017,creutzig_qft_2024}. There are of course more classes of non-semisimple theories in physics, so one could argue that from a physical point of view, non-semisimplicity is quite `natural'. 

The first attempts of generalizing the RT construction started with work of Hennings \cite{hennings_invariants_1996}, who was the first to define 3-manifold invariants from finite-dimensional factorizable ribbon Hopf algebras. Since semisimplicity was not required, this constituted the first construction from non-semisimple data. Lyubashenko generalized this construction by giving the definition of a non-semisimple MTC. He also used it to define projective mapping class group representations \cite{lyubashenko_invariants_1995}. Then, Kerler-Lyubashenko in \cite{kerler_non-semisimple_2001}, define a partial version of an extended TQFT. Their functor is defined out of a double category of cobordisms, whose objects are restricted to be connected surfaces, and the monoidal product to be given by the connected sum.

The question of whether there is a 3d TQFT construction from a non-semisimple MTC was answered some years later by \cite{bartlett_modular_2015}. They proved that once extended 3d TQFTs are classified by semisimple modular tensor categories. In this extended setting, instead of a decorated category of cobordisms, they considered an undecorated 2-category with objects 1-dimensional closed manifolds. The classification result is achieved by considering a `generators and relations' presentation of the 3d cobordism 2-category. Importantly, this suggests that semisimplicity is inevitable, if one considers the full cobordism 2-category as the source of the TQFT.

The work of \cite{de_renzi_3-dimensional_2022}, constructs a RT-type TQFT from the data of a (potentially) non-semisimple MTC. The decorated cobordism categories considered are subject to an admissibility condition. This reflects the comment above, regarding whether there can be non-semisimple TQFTs defined on all cobordisms. In \cite{de_renzi_mapping_2023}, these TQFTs were further examined by computing the resulting projective mapping class group representations, and showing their equivalence to those of Lyubashenko's.

In fact, De Renzi defined a 2-categorical extension of the above in \cite{de_renzi_extended_2021}, in the form of a once extended TQFT, using the so-called `extended universal construction'. Importantly, we note that this is still in the framework of decorated cobordism categories. In fact, 2-dimensional cobordisms with empty incoming boundary and no decorations are disallowed in this construction. In some ways, our work, even though in a different framework, provides an answer to De Renzi's question in \cite{de_renzi_extended_2021}, as to whether the admissibility condition for surfaces can be lifted. This is a result of considering the appropriate target for the TQFT.
The discrepancy between our choice of the target and De Renzi's makes a straightforward comparison of the two constructions somewhat difficult. However, we can somewhat compare the assignments on some generating 1- and 2-morphisms, as the assignment to the circle, although different, is similar: $\C$ vs. $\cP$. Accounting for this difference, we see that the assignments for the pairs of pants are the same: the tensor product functor and (inevitably) its biadjoint, which involves the Lyubashenko coend $\coend$ (\ref{coend_def}). We also see that the assignment of De Renzi's on the 3-ball involves the modified trace, something that happens in our construction as well --- $\nu^\dagger$ contains the data of the modified trace if the appropriate assignment to the $\tinycap$ is considered (see Remark \ref{cap-string-diag}). Based on these observations, we believe that a De Renzi type construction with Bimod as a target, when restricted to undecorated 1-morphisms should recover our construction.


Very recently there have been some 1-categorical `non-compact' TQFT constructions from non-semisimple data. This means that there is a non-compact condition imposed on the morphisms of the (1-)category of cobordisms considered. The first such construction appears in \cite{costantino_non_2023}, where such a non-compact 3d TQFT is constructed from the data of non-semisimple spherical categories, and thus generalizing the Turaev-Viro TQFT \cite{turaev_state_1992,barrett_spherical_1999} to the non-semisimple case. Soon after, there was a construction of a 4d TQFT \cite{costantino_skein_2023} using the data of a non-semisimple ribbon category, generalizing the Crane-Yetter TQFT \cite{crane_categorical_1993} in the non-semisimple setting.

\subsection{Intuition and outline of results}

Having reviewed the literature thus far, and especially if we are interested in a classification result, we are led to a natural question: How can we allow for non-semisimplicity in the once extended 3d setting?
This is the main motivation of this thesis, and what we eventually achieve.



To do this however, there is a crucial observation to be made.

Starting with a once extended 3d TQFT $Z$, one can compactify it on the circle $Z(S^1\times -)$ and obtain a fully extended 2d TQFT. This implies that $Z(S^1)$ will be 2-dualizable. However, for all the `2-vector space' candidate targets for such a 3d ETQFT that were considered in \cite[Appendix A]{bartlett_modular_2015}, the 2-dualizable objects are semisimple.

In order to overcome this problem, there are only so many things we can alter: the source, and the target 2-categories of the TQFT we consider.

To appropriately alter the source, we turn to the idea of non-compact TQFTs introduced by Lurie in \cite{lurie_classification_2009}. In this framework, we disallow certain 3d cobordisms (2-morphisms) that witness adjunctions for $\text{ev}_{Z(S^1)}$ and $\text{coev}_{Z(S^1)}$, and therefore lift the strong requirement of 2-dualizability on $Z(S^1)$.
Having gotten rid of the 2-dualizability, we also need to choose a suitable target 2-category, whose 1-dualizable objects are not semisimple. This is the bicategory Lincat we define in Section \ref{target}.

The necessary admissibility conditions for the decorated cobordism categories encountered in \cite{de_renzi_3-dimensional_2022,de_renzi_extended_2021} can be thought of as a reflection of the above fact. The admissibility condition boils down to requiring the existence of projective strands. The non-compactness condition matches this, since it corresponds to disallowing the 0-handle, which in the decorated setting translates to not allowing the inclusion of the empty skein.

With this in mind, and given a modular tensor category $\C$ in the sense of Definition \ref{MTC_def}, we can define the assignment $\mathcal{Z}$ in section \ref{genass}. The source of $\Znc$ is $\Bordncsig$, the (conjectural) presentation of the 2-category $\bBordncsig$, the signature extension of the 2-category of non-compact 3d cobordisms. Comparing it to \cite{bartlett_modular_2015}, we have disallowed the only 2-morphism generator with empty incoming boundary, as well as any relations involving it. Below, we give a table (\ref{genass_table_data}) to describe the algebraic data used to define $\Znc$. Note that at this point, we only mention non-invertible 2-morphisms of $\Bordncsig$, as the assignments for the invertible 2-morphisms (\ref{MC}, \ref{RC}, \ref{anomaly_gen}) are rather straightforward.

Some key ingrediends that appear are:

\begin{itemize}
    \item The distinguished object $\coend\in \C$, defined in \ref{coend_def}. It carries the structure of a Hopf algebra in $\C$.
    \item The Hopf algebra $\coend$ has additional structure, namely that of an integral $\Lambda\colon \unit\to \coend$ and a cointegral $\Lambda^{co}\colon \coend\to \unit$ (Definition \ref{integral_def}). These morphisms are crucial for our construction.
    \item The projective cover $(P_\unit,\varepsilon_\unit)$ of the unit of $\C$, along with the uniquely determined (Lemma \ref{coint_eta_eps}) morphism $\eta_\unit\colon \unit\to P_\unit$, exhibiting it as the injective envelope as well.
\end{itemize}

\scalecobordisms{0.5}

\begin{table}[h!]
\centering
 \begin{tabular}{||c | c||} 
 \hline
 Generators of $\Bordncsig$& Data used to define $\Znc$ \\ [0.5ex] 
 \hline\hline
$\begin{tz}
\node[Cyl, top, height scale=0]  at (0,0) {};
\end{tz}  $ & $\C$    \\ [0.5ex] 
 $\tinypants$&   $\otimes\colon\C\boxtimes\C\to \C$ \\ [0.5ex] 
 $\tinycopants$ &$\displaystyle\int^{y\in\C}\hspace{-15pt}-\otimes y^\lor\boxtimes y\colon\C\to \C\boxtimes\C$\\ [0.5ex] 
 $\tinycup$    & $-\otimes\unit\colon\Vecfd\to\C$ \\ [0.5ex] 
 $\tinycap$&   $\Hom_\C(-,\unit)^*\colon\C\to \Vecfd$\\ [0.5ex] 
 $\begin{tz} 
        \node[Pants, top, bot] (A) at (0,0) {};
        \node[Copants, bot, anchor=leftleg] at (A.leftleg) {};
    \end{tz}
    \quad
    \xRightarrow{\epsilon}
    \quad
    \begin{tz} 
        \node[Cyl, top, bot, tall] (A) at (0,0) {};
    \end{tz}$ & Involves the evaluation\\ [1ex] 
$\begin{tz} 
        \node[Cyl, tall, top, bot] (A) at (0,0) {};
        \node[Cyl, tall, top, bot] (B) at (2*\cobwidth, 0) {};
    \end{tz}
    \quad
    \xRightarrow{\eta}
    \quad
    \begin{tz} 
        \node[Pants, bot] (A) at (0,0) {};
        \node[Copants, top, bot, anchor=belt] at (A.belt) {};
    \end{tz}$& Involves the coevaluation \\ [1ex] 
    $\begin{tz} 
        \node[Cyl, top, bot, tall] (A) at (0,0) {};
    \end{tz}
    \quad
    \xRightarrow{\epsilon^\dagger}
    \quad
    \begin{tz} 
        \node[Pants, top, bot] (A) at (0,0) {};
        \node[Copants, bot, anchor=leftleg] at (A.leftleg) {};
    \end{tz}$ & Involves the integral $\Lambda\colon \unit\to \coend$ \\ [1ex] 
    $\begin{tz} 
        \node[Pants, bot] (A) at (0,0) {};
        \node[Copants, top, bot, anchor=belt] at (A.belt) {};
    \end{tz}
    \quad
    \xRightarrow{\eta^\dagger}
    \quad
    \begin{tz} 
        \node[Cyl, tall, top, bot] (A) at (0,0) {};
        \node[Cyl, tall, top, bot] (B) at (2*\cobwidth, 0) {};
    \end{tz}$& Involves the cointegral $\Lambda^{co}\colon\coend\to \unit$ \\ [1ex] 
    $\begin{tz}
        \node[Cup, top] (A) at (0,0) {};
        \node[Cap, bot] (B) at (0,-2*\cobheight) {};
    \end{tz}
    \quad
    \xRightarrow{\mu}
    \quad
    \begin{tz}
        \node[Cyl, top, bot, tall] (A) at (0,0) {};
    \end{tz}$ & Involves $\varepsilon_\unit\colon P_\unit\to \unit$ and $\eta_\unit\colon \unit\to P_\unit$ \\ [1ex]
    $\begin{tz}
        \node[Cyl, top, bot, tall] (A) at (0,0) {};
    \end{tz}
    \quad
    \xRightarrow{\mu^\dagger}
    \quad\begin{tz}
        \node[Cup, top] (A) at (0,0) {};
        \node[Cap, bot] (B) at (0,-2*\cobheight) {};
    \end{tz}$ & Unit of $\Hom_\C(-,\unit)^* \dashv -\otimes\unit$  \\ [1ex]
    $\begin{tz}
        \node[Cap, bot] (A) at (0,0) {};
        \node[Cup] at (0,0) {};
    \end{tz}
    \quad
    \xRightarrow{\nu^\dagger}
    \quad
    \begin{tz}
        \draw[green] (0,0) rectangle (0.6, 0.6);  
    \end{tz}$ & Counit of $\Hom_\C(-,\unit)^* \dashv -\otimes\unit$ \\ [1ex]
 \hline
 \end{tabular}
  \caption{Algebraic data assigned by $\Znc$ to generators of $\Bordncsig$}
    \label{genass_table_data}
\end{table}

We then proceed to show that $\mathcal{Z}$ preserves all the relations from \ref{modpres}. This implies that $\mathcal{Z}$ is a well defined symmetric monoidal functor from the 2-category generated by the presentation \ref{modpres} to Lincat.

We will highlight a few important points of our construction: 
\scalecobordisms{0.3}
\begin{itemize}
    \item State spaces for surfaces are given by $\Znc(\Sigma_g)=\Hom_\C(\coend^{\otimes g},1)^*$. We already see that these are dual to the state spaces of \cite{de_renzi_3-dimensional_2022}.
    \item It is a classical result of Crane and Yetter \cite{crane_algebraic_1999} that the once punctured torus is a Hopf algebra object. It is therefore only natural that it is sent to the coend $\coend$ by $\Znc$.
    \item Under $\Znc$, the 2-isomorphism trivializing $\begin{tz}
 \node[Pants, bot] (A) at (0,0) {};
 \node[Cyl, bot, anchor=top] (B) at (A.leftleg) {};
 \node[Copants, anchor=leftleg] (C) at (A.rightleg) {};
 \node[Cyl, top, bot, anchor=bottom] (D) at (C.rightleg) {};
 \node[Cap] (E) at (A.belt) {};
 \node[Cup] (F) at (C.belt) {};
\end{tz}$ gives rise to a trivialization of the Nakayama functor (Remark \ref{snake-nakayama}).
\end{itemize}


Having constructed the 2-functor $\Znc\colon\Bordncsig\to \text{Lincat}$, we obtain representations of a central extension of the mapping class group of surfaces. Equivalently, we obtain projective representations of the mapping class group. We compute the action of generators of the MCG and find that we obtain representations dual to those of \cite{lyubashenko_invariants_1995,de_renzi_mapping_2023}. To do this, we use descriptions of the corresponding Dehn twists as composites of generators of \ref{modpres}.

The outline of the thesis is as follows:

\begin{enumerate}
    \item In the first part of chapter \ref{setup} we review the definitions of extended non-compact cobordism categories and define the presentation of $\Bordncsig$. Following that, we give the necessary algebraic background on modular tensor categories.
    \item In the first part of chapter \ref{construction} we give an explicit definition of the assignment $\Znc$ that we sketched in table \ref{genass_table_data}. The second part of this chapter is devoted to proving that $\Znc$ preserves the relations of the presentation \ref{modpres}. This constitutes the proof of Theorem \ref{main_th_intro}.
    \item In the first part of chapter \ref{mcg_actions} we investigate the actions of composites that correspond to Dehn twists that generate the mapping class group. We then proceed to identify the corresponding projective representations with those of Lyubashenko and De Renzi et al.
    \item There is a 2-functor $\text{Lincat}\to \text{Bimod}$. In chapter \ref{Lincat-Bimod} we use this to compute the values of $\Znc$ in Bimod, thus obtaining a diagrammatic way to represent the action of 2-morphisms. We do this to facilitate computations in the next chapters.
    \item In chapter \ref{3-mfld-inv} we explain a procedure to extract a number from the non-compact TQFT $\Znc$, given a 3-manifold $M$. We verify that this number `coincides' with the Lyubashenko invariant for $M$. A corollary of this in the semisimple case is that the value of $M$ under the TQFT of \cite{bartlett_modular_2015} is the Reshetikhin-Turaev invariant of $M$.
    \item The aim of chapter \ref{mod-trace} is to explain why, and in what sense the value of $\Znc(S^1\times D^2)$ is a modified trace. This involves explaining why there is a canonical isomorphism $\Znc(T^2)\cong \displaystyle\int^{P\in\cP} \hspace{-20pt} \Hom_\C(P,P)$. We do this via arguments relating to the theory of 1-dualizability in bicategories.
\end{enumerate}

\subsection{Future directions}\label{converse}

As has already been alluded to, an important motivation for this work is to obtain a full classification of 3d non-compact TQFTs. Having established Theorem \ref{main_th_intro}, this amounts to considering the converse question: `Starting with a non-compact 3d TQFT $Z$ with target Lincat/Bimod, what is the induced structure on $Z(S^1)$?'

Because of the proof of Theorem \ref{main_th_intro} being very explicit, a lot of the induced structure is already clear. We expect that $Z(S^1)$ turns out to be some form of `generalized' MTC. This is due to the fact that rigidity will only be imposed on projective objects.

\subsection{Conventions and notation}

Fix $k$ to be a perfect field. This is to ensure that the results of \cite{kerler_non-semisimple_2001} hold. Note that this is a milder assumption than the algebraic closedness in other works. All vector spaces, algebras and categories are considered to be over $k$. By $\text{Vec}_k$, we denote the category of finite dimensional vector spaces and linear maps. In a monoidal category $\C$, we denote the left dual of an object $X$ by $X^\lor$, and consequently, we denote the left evaluation and coevaluation by $\text{ev}_X\colon X^\lor\otimes X\to\unit$ and $\text{coev}_X\colon\unit \to X\otimes X^\lor$ respectively. The right dualizability data only appears in chapter \ref{mod-trace}. Since the categories we work with are pivotal ($\text{piv}\colon X\xrightarrow{\sim} \left(X^\lor\right)^\lor$), we define it by: 
$$\text{ev}^R_X:= \text{ev}_{X^\lor}\circ(\text{piv}\otimes \id_{X^\lor}) 
\quad \textrm{and} \quad 
\text{coev}^R_X:= (\id_{X^\lor}\otimes \:\text{piv}^\inv)\circ \text{coev}_{X^\lor}.$$
As an exception, we denote the dual vector space of $V$ as $V^*$.

The convention for reading cobordisms and diagrams is from bottom to top. Internal string diagrams encountered in chapters \ref{Lincat-Bimod} and \ref{3-mfld-inv} are read from top to bottom. 

\subsection{Acknowledgements}

I would like to thank my advisor Pavel Safronov for his guidance throughout this project. I am grateful to Benjamin Ha\"ioun and Iordanis Romaidis for helpful discussions. I also want to extend my thanks to the authors of \cite{bartlett_modular_2015} for making their tikz code available. This research is part of my PhD thesis, undertaken at the university of Edinburgh.

\newpage

\section{Setup}\label{setup}

\subsection{The non-compact 3d cobordism 2-category}\label{cobordisms}

We give a sketch of a definition of the bicategory of non-compact 3d oriented cobordisms $\bBord$. For a detailed definition of the full bicategory of oriented cobordisms, we refer the reader to \cite{schommer-pries_classification_2014}.

\scalecobordisms{1}

\begin{itemize}
\item An object is a closed oriented $1$-manifold, i.e. a disjoint unions of a finite (possibly zero) number of circles.
\begin{equation}
\scalecobordisms{1}
\begin{tz}
\node[Cyl, top, height scale=0]  at (0,0) {};
\node[Cyl, top, height scale=0]  at (1.5,0) {};
\node[Cyl, top, height scale=0]  at (3,0) {};
\node at (4.5,0) {$\cdots$};
\node[Cyl, top, height scale=0]  at (6,0) {};
\end{tz}
\end{equation}
\item A $1$-morphism is a compact oriented $2$-dimensional cobordism between the objects. For instance, the cobordism depicted below is a $1$-morphism $S^\unit\to S^1\sqcup S^1$.
\begin{equation}
\begin{tz} 
\node[Copants, top, anchor=belt] at (0,0) {};
\end{tz}
\end{equation}
\noindent Composition is given by gluing manifolds along their common boundary.
\item A $2$-morphism $\Sigma\to \Sigma^\prime$ is a $3$-manifold $M$, potentially with corners, which is a cobordism between $\Sigma$ and $\Sigma^\prime$, with the property that every connected component of $M$ has nonempty intersection with $\Sigma$.


For example, a $3$-manifold realizing the $2$-morphism
\begin{equation}
\begin{tz} 
\node[Cyl, tall, top, bot] (A) at (0,0) {};
\node[Cyl, tall, top, bot] (B) at (2*\cobwidth, 0) {};
\end{tz}
\quad
\Longrightarrow
\quad
\begin{tz} 
\node[Pants, bot] (A) at (0,0) {};
\node[Copants, top, bot, anchor=belt] at (A.belt) {};
\end{tz}
\end{equation}
\noindent can be visualized as having appropriately 'drilled' out two solid cylinders from the solid `X' shape that is the target of the 2-morphism. This can be depicted as follows:


$$\begin{tz}
        \node[Pants, bot, belt scale=1.5] (A) at (0,0) {};
        \node[Copants, bot, anchor=belt, belt scale=1.5, top] (B) at (A.belt) {}; 
        \begin{scope}
                \node (i) at (B.leftleg) [above] {};
                \node (j) at (B.rightleg) [above] {};
                \node (i2) at (A.leftleg) [below] {};
                \node (j2) at (A.rightleg) [below] {};
                \draw (i.south) to[out=down, in=up] (B-belt.in-leftthird) to[out=down, in=up] (i2.north);
                \draw (j.south) to[out=down, in=up] (B-belt.in-rightthird) to[out=down, in=up] (j2.north);
        \end{scope}
\end{tz}$$

where we have removed the tubular neighborhoods of the two lines depicted above.

To give some intuition for the non-compactness condition, the 3-ball $D^3\colon S^2\to\empty$ is a 2-morphism in this category, but $D^3\colon \emptyset\to S^2$ is not.

\item The symmetric monoidal structure is given by disjoint union.
\end{itemize}

As mentioned in the introduction, the input data for our construction is that of a modular tensor category. In the case where the ratio $p_+/p_-\neq 1$ (as defined in Definition \ref{twistnondegen}), the resulting TQFT is anomalous. We will therefore work with a presentation that should be equivalent to a 'central extension' of $\bBord$. This is the 2-category of cobordisms equipped with signature, $\bBordncsig$. The definition is the same as for $\bBord$, except for the fact that 1- and 2- morphisms are equipped with some extra data:

    \begin{itemize}
        \item Objects are closed 1-dimensional manifolds $\Gamma$, equipped with a bounding 2-manifold $\tilde{\Gamma}$, which we choose to be a disjoint union of disks.
        \item 1-morphisms $(\Sigma,\tilde{\Sigma})\colon(\Gamma,\tilde{\Gamma})\to(\Gamma^\prime,\tilde{\Gamma}^\prime)$ are again 2-dimensional cobordisms $\Sigma$, equipped with a choice of a 3-manifold $\tilde{\Sigma}$ that bounds $\overline{\tilde{\Gamma}} \cup_{\Gamma} \Sigma \cup_{\Gamma^\prime} \tilde{\Gamma}^\prime$.
        \item 2-morphisms $(M,W)\colon(\Sigma,\tilde{\Sigma})\to(\Sigma^\prime,\tilde{\Sigma}^\prime)$ are 3-dimensional cobordisms $M\colon\Sigma\to\Sigma^\prime$ with corners, with the property that every connected component of $M$ has nonempty intersection with $\Sigma$, and equipped with a 4-manifold $W$ with  $\partial W= \overline{\tilde{\Sigma}} \cup_{\Sigma} M \cup_{\Sigma^\prime} \tilde{\Sigma}^\prime$.
    \end{itemize}

Note that we could have also defined $\bBordncsig$ similar to \cite{de_renzi_extended_2021}, equipping 1-morphisms with a Lagrangian in their first cohomology, and 2-morphisms with an integer (the signature of the corresponding bounding 4-manifold). 

\subsection{Presentation of the non-compact 3d cobordism 2-category}\label{presentations}

In this section we give a description of our source 2-category for the TQFT in terms of generators and relations. To be more precise, this is a conjectured equivalence between the 2-category of non-compact 3d cobordisms and the presentation we will provide later in this section.

In more detail, a presentation for a symmetric monoidal 2-category consists of
the following finite collection of data:

\begin{itemize}
    \item generating objects;
    \item generating 1-morphisms, whose sources and targets are composites of
generating objects;
    \item generating 2-morphisms, whose sources and targets are composites of
generating 1-morphisms;
    \item relations, which are equations between composites of generating 2-morphisms.
\end{itemize}

To such a presentation, we can associate a (fully-weak) symmetric monoidal
2-category that is generated by it, in the sense of \cite[page 161]{schommer-pries_classification_2014}.

Strict symmetric monoidal functors out of this generated symmetric monoidal 2-category are uniquely specified by the images of the generating objects, 1-morphisms and 2-morphisms, such that the relations of the presentation are satisfied.

\begin{conj}\label{conjecture}
    The non-compact 3d cobordism 2-category $\bBordncsig$, whose definition was sketched in Section \ref{cobordisms}, is equivalent to the 2-category obtained from the presentation we describe below.
\end{conj}


 The following presentation is identical to the signature presentation in \cite{bartlett_modular_2015}, apart from the fact that we have removed `$\nu$' from the set of generators, as well as any relations involving it. 
 Note that if we discard the generator $\xi$ (i.e $\xi=\id$), then we expect that the corresponding presentation is equivalent to $\bBord$.




\scalecobordisms{0.5}

\begin{definition} \label{modpres}
    
We define the \textbf{non-compact} signature presentation $\Bordncsig$ as follows:
\begin{itemize}
\item Generating object:
\[
\begin{tz}
        \node[Cyl, top, height scale=0]  at (0,0) {};
\end{tz}
\]
 
 \item Generating 1-morphisms:
    \[
    \begin{tz}
        \node[Pants, top, bot] (A) at (0,0) {};
        \node[Copants, top, bot] (B) at (2,0) {};
        \node[Cup, top] (C) at (4,0.1) {};
        \node[Cap, bot] (D) at (6,-0.1) {};
    \end{tz}
    \]
 \item Invertible generating 2-morphisms:
    \begin{equation}
    \begin{tz}
        \node[Pants, top, bot, wide] (A) at (0,0) {};
        \node[Pants,  bot, anchor=belt] (B) at (A.leftleg) {};    
        \node[Cyl, bot, anchor=top] at (A.rightleg) {}; 
    \end{tz}
    \quad
    \RLarrow{\alpha}{\alpha^\inv}
    \quad
    \begin{tz}
        \node[Pants, top, bot, wide] (A) at (0,0) {};
        \node[Pants,  bot, anchor=belt] (B) at (A.rightleg) {};    
        \node[Cyl, bot, anchor=top] at (A.leftleg) {}; 
    \end{tz} 
    \qquad
    \qquad
    \begin{tz}
        \node[Pants, top, bot] (A) at (0,0) {};
        \node[Cyl, bot, anchor=top] at (A.leftleg) {};
        \node[Cup] at (A.rightleg) {};  
    \end{tz}
    \quad
    \RLarrow{\rho}{\rho^\inv}
    \quad
    \begin{tz}
        \node[Cyl, bot, top, tall] at (0,0) {};
    \end{tz}
    \quad
    \RLarrow{\lambda^\inv}{\lambda}
    \quad
    \begin{tz}
        \node[Pants, top, bot] (A) at (0,0) {};
        \node[Cyl, bot, anchor=top] at (A.rightleg) {};
        \node[Cup] at (A.leftleg) {};   
    \end{tz} \label{MC}
    \end{equation}

    \begin{equation}
    \begin{tz}
        \node[Pants, top, bot] (A) at (0,0) {};
    \end{tz}
    \quad
    \RLarrow{\beta}{\beta^\inv}
    \quad
    \begin{tz}
        \node[Pants, top, bot] (A) at (0,0) {};
        \node[BraidB, anchor=topleft, bot] at (A.leftleg) {};
    \end{tz}
    \qquad
    \qquad
    \begin{tz}
        \node[Cyl, top, bot, tall] (A) at (0,0) {};
    \end{tz}
    \quad
    \RLarrow{\theta}{\theta^\inv}
    \quad
    \begin{tz}
        \node[Cyl, top, bot, tall] (A) at (0,0) {};
    \end{tz} \label{RC}
    \end{equation}

        \begin{equation}
        \begin{tikzpicture}
 \node[Cap] (A) at (0,0) {};
 \node[Cup] (B) at (0,0) {};
\end{tikzpicture}
    \quad
    \RLarrow{\xi}{\xi^\inv}
    \quad
            \begin{tikzpicture}
 \node[Cap] (A) at (0,0) {};
 \node[Cup] (B) at (0,0) {};
\end{tikzpicture} \label{anomaly_gen}
    \end{equation}

 \item Non-invertible generating 2-morphisms:
    \begin{equation}
    \begin{tz} 
        \node[Cyl, tall, top, bot] (A) at (0,0) {};
        \node[Cyl, tall, top, bot] (B) at (2*\cobwidth, 0) {};
    \end{tz}
    \quad
    \RLarrow{\eta}{\eta^\dagger}
    \quad
    \begin{tz} 
        \node[Pants, bot] (A) at (0,0) {};
        \node[Copants, top, bot, anchor=belt] at (A.belt) {};
    \end{tz}
    \qquad
    \qquad
    \begin{tz} 
        \node[Pants, top, bot] (A) at (0,0) {};
        \node[Copants, bot, anchor=leftleg] at (A.leftleg) {};
    \end{tz}
    \quad
    \RLarrow{\epsilon}{\epsilon^\dagger}
    \quad
    \begin{tz} 
        \node[Cyl, top, bot, tall] (A) at (0,0) {};
    \end{tz} \label{A1}
    \end{equation}
    
    \begin{equation} 
    \begin{tz}
        \draw[green] (0,0) rectangle (0.6, 0.6);  
    \end{tz}
    \quad
    \xLeftarrow{\nu^\dagger}
    \quad
    \begin{tz}
        \node[Cap, bot] (A) at (0,0) {};
        \node[Cup] at (0,0) {};
    \end{tz}
    \qquad
    \qquad
    \begin{tz}
        \node[Cup, top] (A) at (0,0) {};
        \node[Cap, bot] (B) at (0,-2*\cobheight) {};
    \end{tz}
    \quad
    \RLarrow{\mu}{\mu^\dagger}
    \quad
    \begin{tz}
        \node[Cyl, top, bot, tall] (A) at (0,0) {};
    \end{tz}
    \label{A2}
    \end{equation}

 \end{itemize}
 
The relations are as follows:
\begin{itemize}
\item (Inverses) Each of the invertible generating 2-morphisms $\omega$ satisfies \mbox{$\omega \circ \omega^\inv = \id$} and $\omega^\inv \circ \omega = \id$. 
\item (Monoidal) The generating 2-morphisms in \eqref{MC} obey the pentagon and unit equations:
\begin{equation} \label{eq:pentagon}
\begin{tz}[xscale=2.1, yscale=1.2]
\node (1) at (0,0)
{
$\begin{tikzpicture}
        \node [Pants, wider, top] (A) at (0,0) {};
        \node [Pants, wide, anchor=belt] (B) at (A.leftleg) {};
        \node [Cyl, tall, anchor=top] (C) at (A.rightleg) {};
        \node[Pants, anchor=belt] (D) at (B.leftleg) {};
        \node[Cyl, anchor=top] (E) at (B.rightleg) {};
        \selectpart[green] {(A-belt) (B-leftleg) (A-rightleg)};
        \selectpart[red] {(A-leftleg) (D-leftleg) (B-rightleg)};
\end{tikzpicture}$
};
\node (2) at (1,1)
{
$\begin{tikzpicture}
        \node [Pants, verywide, top] (A) at (0,0) {};
        \node [Pants, anchor=belt] (B) at (A.rightleg) {};
        \node [Cyl, anchor=top] (C) at (A.leftleg) {};
        \node[Pants, anchor=belt] (D) at (C.bottom) {};
        \node[Cyl, anchor=top] (E) at (B.leftleg) {};
        \node[Cyl, anchor=top] (F) at (B.rightleg) {};
        \selectpart[green] {(A-leftleg) (A-rightleg) (D-leftleg) (F-bottom)};
\end{tikzpicture}$
};
\node (3) at (2,1)
{
$\begin{tikzpicture}
        \node [Pants, verywide, top] (A) at (0,0) {};
        \node [Pants, anchor=belt] (B) at (A.leftleg) {};
        \node[Cyl, anchor=top] (C) at (A.rightleg) {};
        \node[Cyl, anchor=top] (D) at (B.leftleg) {};
        \node[Cyl, anchor=top] (E) at (B.rightleg) {};
        \node[Pants, anchor=belt] (F) at (C.bottom) {};
        \selectpart[green] {(A-belt) (B-leftleg) (A-rightleg)};
\end{tikzpicture}$
};
\node (4) at (3,0)
{
$\begin{tikzpicture}
        \node [Pants, wider, top] (A) at (0,0) {};
        \node [Cyl, tall, anchor=top] (B) at (A.leftleg) {};
        \node[Pants, anchor=belt, wide] (C) at (A.rightleg) {};
        \node[Cyl, anchor=top] (D) at (C.leftleg) {};
        \node[Pants, anchor=belt] (E) at (C.rightleg) {};
\end{tikzpicture}$
};
\node (5) at (1,-1)
{
$\begin{tikzpicture}
        \node [Pants, veryverywide, top] (A) at (0,0) {};
        \node [Pants, wide, anchor=belt] (B) at (A.leftleg) {};
        \node [Cyl, tall, anchor=top] (C) at (A.rightleg) {};
        \node[Pants, anchor=belt] (D) at (B.rightleg) {};
        \node[Cyl, anchor=top] (E) at (B.leftleg) {};
                \selectpart[green] {(A-belt) (B-leftleg) (A-rightleg)};
\end{tikzpicture}$
};
\node (6) at (2,-1)
{
$\begin{tikzpicture}
        \node [Pants, veryverywide, top] (A) at (0,0) {};
        \node [Pants, wide, anchor=belt] (B) at (A.rightleg) {};
        \node [Cyl, tall, anchor=top] (C) at (A.leftleg) {};
        \node[Pants, anchor=belt] (D) at (B.leftleg) {};
        \node[Cyl, anchor=top] (E) at (B.rightleg) {};
                \selectpart[green] {(A-rightleg) (D-leftleg) (B-rightleg)};
\end{tikzpicture}$
};
\begin{scope}[double arrow scope]
    \draw (1) --  node[above left, green]{$\alpha$} (2);
    \draw (2) --  node[above]{$\interchangor$} (3);
    \draw (3) --  node[above]{$\alpha$} (4);
    \draw (1) --  node[below left, red]{$\alpha$} (5);
    \draw (5) --  node[below]{$\alpha$} (6);
    \draw (6) --  node[below right]{$\alpha$} (4);
\end{scope}
\end{tz}
\end{equation}
\begin{equation}
\label{eq:triangle}
\begin{tz}[xscale=1, yscale=2, every to/.style={out=down,in=up}]
\node (1) at (-1,1)
{
$\begin{tikzpicture}
    \node (A) [Pants] at (0,0) {};
    \node (B) [Pants, wide, top, anchor=leftleg] at (A.belt) {};
    \node (C) [Cup] at (A.rightleg) {};
    \node [Cyl, anchor=top] (D) at (A.leftleg) {};
    \node [Cyl, tall, anchor=top] (E) at (B.rightleg) {};
        \selectpart[green] {(B-belt) (A-leftleg) (B-rightleg)};
    \selectpart[red] {(B-leftleg) (D-bottom) (A-rightleg)};
\end{tikzpicture}$
};
\node (2) at (1,1)
{
$\begin{tikzpicture}
    \node (A) [Pants] at (0,0) {};
    \node (B) [Pants, wide, top, anchor=rightleg] at (A.belt) {};
    \node (C) [Cup] at (A.leftleg) {};
    \node (D) [Cyl, anchor=top] at (A.rightleg) {};
    \node [Cyl, tall, anchor=top] at (B.leftleg) {};
    \selectpart[green] {(B-rightleg) (A-leftleg) (D-bottom)};
\end{tikzpicture}$
};
\node (3) at (0,0)
{
$\begin{tikzpicture}
    \node (A) [Pants, top] at (0,0) {};
\end{tikzpicture}$
};
\begin{scope}[double arrow scope]
    \draw (1) --  node[above, green]{$\alpha$} (2);
    \draw (2) --  node[below right]{$\lambda$} (3);
    \draw (1) --  node[below left, red]{$\rho$} (3);
\end{scope}
\end{tz}
\end{equation}
The 2-morphism $\interchangor$ at \eqref{eq:pentagon} is an interchanger, part of a canonical family of 2-morphisms in any monoidal 2-category that switches the 1-morphism composition
order of two tensored 1-morphisms  \cite{bartlett_quasistrict_2014}.
 \item (Balanced) The data \eqref{MC} and \eqref{RC} forms a braided monoidal object equipped with a compatible twist:
\begin{equation}
\label{eq:hexagon}
\begin{tz}[xscale=2.2, yscale=1.2]
\node (1) at (0,0)
{
$\begin{tikzpicture}
    \node [Pants, wide, top] (A) at (0,0) {};
    \node [Pants, anchor=belt] (B) at (A.leftleg) {};
    \node [Cyl, anchor=top] (C) at (A.rightleg) {};
        \selectpart[red]{(A-leftleg) (B-leftleg) (B-rightleg)};
\end{tikzpicture}$
};

\node (2) at (0.75,1)
{
$\begin{tikzpicture}
    \node [Pants, wide, top] (A) at (0,0) {};
    \node [Pants, anchor=belt] (B) at (A.rightleg) {};
    \node [Cyl, anchor=top] (C) at (A.leftleg) {};
        \selectpart[green]{(A-belt) (A-leftleg) (A-rightleg)};
\end{tikzpicture}$
};

\node (3) at (1.5,1)
{
$\begin{tikzpicture}
    \node [Pants, wide, top] (A) at (0,0) {};
    \node[BraidB, wide, anchor=topleft] (B) at (A.leftleg) {};
    \node [Pants, anchor=belt] (C) at (B.bottomright) {};
    \node [Cyl, anchor=top] (D) at (B.bottomleft) {};
\end{tikzpicture}$
};

\node (4) at (2.25,1)
{
$\begin{tikzpicture}
    \node [Pants, wide, top] (A) at (0,0) {};
    \node [Pants, anchor=belt] (C) at (A.leftleg) {};
    \node [Cyl, anchor=top] (D) at (A.rightleg) {};
    \node [BraidB, anchor=topleft] (E) at (C.rightleg) {};
    \node[Cyl, anchor=top] (F) at (C.leftleg) {};
    \node[BraidB, anchor=topleft] (G) at (F.bottom) {};
    \node[Cyl, anchor=top] (H) at (G.bottomright) {};
    \node[Cyl, anchor=top] (I) at (G.bottomleft) {};
    \node[Cyl, anchor=top, tall] (J) at (E.bottomright) {};
    \selectpart[green]{(A-belt) (C-leftleg) (D-bottom)};
\end{tikzpicture}$
};

\node (5) at (3,0)
{
$\begin{tikzpicture}
    \node [Pants, wide, top] (A) at (0,0) {};
    \node [Pants, anchor=belt] (C) at (A.rightleg) {};
    \node [Cyl, anchor=top] (D) at (A.leftleg) {};
    \node [BraidB, anchor=topleft] (E) at (C.leftleg) {};
    \node[Cyl, tall, anchor=top] (F) at (A.leftleg) {};
    \node[BraidB, anchor=topleft] (G) at (F.bottom) {};
    \node[Cyl, anchor=top] (H) at (E.bottomright) {};
\end{tikzpicture}$
};

\node (6) at (1,-1)
{
$\begin{tikzpicture}
    \node [Pants, wide, top] (A) at (0,0) {};
    \node [Pants, anchor=belt] (B) at (A.leftleg) {};
    \node [Cyl, anchor=top, tall] (C) at (A.rightleg) {};
    \node [BraidB, anchor=topleft] (D) at (B.leftleg) {};
        \selectpart[green]{(A-belt) (B-leftleg) (A-rightleg)};
\end{tikzpicture}$
};

\node (7) at (2.0,-1)
{
$\begin{tikzpicture}
    \node [Pants, wide, top] (A) at (0,0) {};
    \node [Pants, anchor=belt] (B) at (A.rightleg) {};
    \node [Cyl, anchor=top] (C) at (A.leftleg) {};
    \node[BraidB, anchor=topleft] (D) at (C.bottom) {};
    \node[Cyl, anchor=top] (E) at (B.rightleg) {};
    \selectpart[green]{(A-rightleg) (B-leftleg) (B-rightleg)};
\end{tikzpicture}$
};

\begin{scope}[double arrow scope]
    \draw (1) --  node[above left]{$\alpha$} (2);
    \draw (2) --  node[above]{$\beta$} (3);
    \draw[-=] (3) --  (4);
    \draw (4) --  node[above right]{$\alpha$} (5);
    \draw (1) --  node[below left, red]{$\beta$} (6);
    \draw (6) --  node[below]{$\alpha$} (7);
    \draw (7) --  node[below right]{$\beta$} (5);
\end{scope}
\end{tz}
\end{equation}
\begin{equation} \label{balanced1}
\begin{tz}[xscale=1.4, yscale=1.5]

\node (1) at (0,0)
{
$\begin{tikzpicture}
        \node[Pants, top] (A) at (0,0) {};
        \selectpart[green, inner sep=1pt]{(A-belt)};
        \selectpart[red] {(A-leftleg) (A-rightleg) (A-belt)};
\end{tikzpicture}$
};
\node (2) at (1,0)
{
$\begin{tikzpicture}
        \node[Pants, top] (A) at (0,0) {};
\end{tikzpicture}$
};

\node (3) at (0,-1)
{
$\begin{tikzpicture}
        \node[Pants, top] (A) at (0,0) {};
        \selectpart[green, inner sep=1pt]{(A-leftleg)};
\end{tikzpicture}$
};

\node (4) at (1,-1)
{
$\begin{tikzpicture}
        \node[Pants, top] (A) at (0,0) {};
        \selectpart[green, inner sep=1pt]{(A-rightleg)};
\end{tikzpicture}$
};

\begin{scope}[double arrow scope]
    \draw (1) -- node[above, green] {$\theta$} (2);
    \draw (1) -- node[left, red] {$\beta^2$} (3);
    \draw (3) -- node[below] {$\theta$} (4);
    \draw (4) -- node[right] {$\theta$} (2);
\end{scope}
\end{tz}
\end{equation}
\begin{equation}
\label{balanced2}
\begin{tz}[xscale=1.4, yscale=2]
\node (1) at (0,0)
{
$\begin{tikzpicture}
    \node[Cup, top] (C) at (0,0) {};
        \selectpart[green, inner sep=1pt]{(C-center)};
\end{tikzpicture}$
};
\node (2) at (1,0)
{
$\begin{tikzpicture}
    \node[Cup, top] (C) at (0,0) {};
    \selectpart[green, inner sep=1pt]{(C-center)};

  \end{tikzpicture}$
};
\begin{scope}[double arrow scope]
    \draw (1) --  node[above]{$\theta$} (2);
\end{scope}d
\end{tz}
\quad = \quad
\begin{tz}[xscale=1.4, yscale=2]
\node (1) at (0,0)
{
$\begin{tikzpicture}
    \node[Cup, top] (C) at (0,0) {};
\end{tikzpicture}$
};
\node (2) at (1,0)
{
$\begin{tikzpicture}
    \node[Cup, top] (C) at (0,0) {};
  \end{tikzpicture}$
};
\begin{scope}[double arrow scope]
    \draw (1) --  node[above]{$\id$} (2);
\end{scope}
\end{tz}
\end{equation}

 \item (Rigidity) Write $\phileft$ for the following composite:
 \begin{align} \label{defn_of_phileft}
 \phileft \quad&:=\quad
\begin{tz}
 \node[Pants, top, bot] (A) at (0,0) {};
 \node[Cyl, bot, anchor=top] (B) at (A.leftleg) {};
 \node[Copants, bot, anchor=leftleg] (C) at (A.rightleg) {};
 \node[Cyl, top, bot, anchor=bottom] (D) at (C.rightleg) {}; 
 \selectpart[green, inner sep=1pt] {(A-belt) (D-top)};
\end{tz}
\,\, \Rarrow{\eta} \,\,
\begin{tz}
 \node[Copants, top, wide, bot] (F) at (0,0) {};
 \node[Pants, bot, wide, anchor=belt] (G) at (F.belt) {};
 \node[Pants, bot, anchor=belt] (A) at (G.leftleg) {};
 \node[Cyl, bot, anchor=top] (B) at (A.leftleg) {};
 \node[Copants, bot, anchor=leftleg] (C) at (A.rightleg) {};
 \node[Cyl, bot, anchor=bottom] (X) at (C.rightleg) {}; 
 \selectpart[green] {(F-belt) (A-leftleg) (X-bottom)};
\end{tz}
\,\, \Rarrow{\alpha} \,\,
\begin{tz}
 \node[Copants, top, wide, bot] (F) at (0,0) {};
 \node[Pants, bot, wide, anchor=belt] (G) at (F.belt) {};
 \node[Pants, bot, anchor=belt] (A) at (G.rightleg) {};
 \node[Cyl, tall, bot, anchor=top] (B) at (G.leftleg) {};
 \node[Copants, bot, anchor=leftleg] (C) at (A.leftleg) {};
 \selectpart[green] {(G-rightleg) (A-leftleg) (A-rightleg) (C-belt)};
\end{tz}
\,\, \Rarrow{\epsilon} \,\,
\begin{tz}
        \node[Copants, top, bot] (A) at (0,0) {};
        \node[Pants, bot, anchor=belt] (B) at (A.belt) {};
\end{tz}
\intertext{The left rigidity relation says that $\phileft$ is invertible, with the following explicit inverse:}
\label{explicit_phileft_inverse}
\phileft^\inv \quad&=\quad  \begin{tz}
                \node[Copants, top] (A) at (0,0) {};
                \node[Pants, anchor=belt] (B) at (A.belt) {};
                \selectpart[green, inner sep=1pt] {(B-rightleg)};
\end{tz} 
\,\, \Rarrow{\epsilon^\dagger} \,\,
\begin{tz}
 \node[Copants, top, wide, bot] (F) at (0,0) {};
 \node[Pants, bot, wide, anchor=belt] (G) at (F.belt) {};
 \node[Pants, bot, anchor=belt] (A) at (G.rightleg) {};
 \node[Cyl, tall, bot, anchor=top] (B) at (G.leftleg) {};
 \node[Copants, bot, anchor=leftleg] (C) at (A.leftleg) {};
 \selectpart[green]{(F-belt) (G-leftleg) (A-rightleg)};
\end{tz}
\,\, \Rarrow{\alpha^\inv} \,\,
\begin{tz}
 \node[Copants, top, wide, bot] (F) at (0,0) {};
 \node[Pants, bot, wide, anchor=belt] (G) at (F.belt) {};
 \node[Pants, bot, anchor=belt] (A) at (G.leftleg) {};
 \node[Cyl, bot, anchor=top] (B) at (A.leftleg) {};
 \node[Copants, bot, anchor=leftleg] (C) at (A.rightleg) {};
 \node[Cyl, bot, anchor=bottom] (X) at (C.rightleg) {}; 
 \selectpart[green]{(F-leftleg) (F-rightleg) (G-leftleg) (G-rightleg)};
\end{tz}
\,\, \Rarrow{\eta^\dagger} \,\,
\begin{tz}
 \node[Pants, top, bot] (A) at (0,0) {};
 \node[Cyl, bot, anchor=top] (B) at (A.leftleg) {};
 \node[Copants, bot, anchor=leftleg] (C) at (A.rightleg) {};
 \node[Cyl, top, bot, anchor=bottom] (D) at (C.rightleg) {}; 
\end{tz}
\intertext{Similarly, write $\phiright$ for $\phileft$ rotated about the $z$-axis (See \cite[Appendix B]{bartlett_extended_2014} for what this means):}
 \label{defn_of_phiright}
 \phiright \quad&:=\quad
\begin{tz}
 \node[Pants, top, bot] (A) at (0,0) {};
 \node[Cyl, bot, anchor=top] (B) at (A.rightleg) {};
 \node[Copants, bot, anchor=rightleg] (C) at (A.leftleg) {};
 \node[Cyl, top, bot, anchor=bottom] (D) at (C.leftleg) {}; 
 \selectpart[inner sep=1pt, green] {(D-top) (A-belt)};
 \end{tz}
\,\, \Rarrow{\eta} \,\,
\begin{tz}
 \node[Copants, top, wide, bot] (F) at (0,0) {};
 \node[Pants, bot, wide, anchor=belt] (G) at (F.belt) {};
 \node[Pants, bot, anchor=belt] (A) at (G.rightleg) {};
 \node[Cyl, bot, anchor=top] (B) at (A.rightleg) {};
 \node[Copants, bot, anchor=rightleg] (C) at (A.leftleg) {};
 \node[Cyl, bot, anchor=bottom] (X) at (C.leftleg) {}; 
 \selectpart[green] {(F-belt) (A-rightleg) (X-bottom)};
\end{tz}
\,\, \Rarrow{\alpha^\inv} \,\,
\begin{tz}
 \node[Copants, top, wide, bot] (F) at (0,0) {};
 \node[Pants, bot, wide, anchor=belt] (G) at (F.belt) {};
 \node[Pants, bot, anchor=belt] (A) at (G.leftleg) {};
 \node[Cyl, tall, bot, anchor=top] (B) at (G.rightleg) {};
 \node[Copants, bot, anchor=rightleg] (C) at (A.rightleg) {};
 \selectpart[green] {(G-leftleg) (A-leftleg) (A-rightleg) (C-belt)};
\end{tz}
\,\, \Rarrow{\epsilon} \,\,
\begin{tz}
        \node[Copants, top, bot] (A) at (0,0) {};
        \node[Pants, bot, anchor=belt] (B) at (A.belt) {};
\end{tz}
\intertext{The right rigidity relation says that $\phiright$ is invertible, with the following explicit inverse :}
\label{explicit_phiright_inverse}
\phiright^\inv \quad&=\quad
\begin{tz}
                \node[Copants, top] (A) at (0,0) {};
                \node[Pants, anchor=belt] (B) at (A.belt) {};
                \selectpart[green, inner sep=1pt] {(B-leftleg)};
\end{tz} 
\,\, \Rarrow{\epsilon^\dagger} \,\,
\begin{tz}
 \node[Copants, top, wide, bot] (F) at (0,0) {};
 \node[Pants, bot, wide, anchor=belt] (G) at (F.belt) {};
 \node[Pants, bot, anchor=belt] (A) at (G.leftleg) {};
 \node[Cyl, tall, bot, anchor=top] (B) at (G.rightleg) {};
 \node[Copants, bot, anchor=rightleg] (C) at (A.rightleg) {};
 \selectpart[green]{(F-belt) (A-leftleg) (G-rightleg)};
\end{tz}
\,\, \Rarrow{\alpha} \,\,
\begin{tz}
 \node[Copants, top, wide, bot] (F) at (0,0) {};
 \node[Pants, bot, wide, anchor=belt] (G) at (F.belt) {};
 \node[Pants, bot, anchor=belt] (A) at (G.rightleg) {};
 \node[Cyl, bot, anchor=top] (B) at (A.rightleg) {};
 \node[Copants, bot, anchor=rightleg] (C) at (A.leftleg) {};
 \node[Cyl, top, bot, anchor=bottom] (X) at (C.leftleg) {}; 
 \selectpart[green]{(F-leftleg) (F-rightleg) (G-leftleg) (G-rightleg)};
\end{tz}
\,\, \Rarrow{\eta^\dagger} \,\,
\begin{tz}
 \node[Pants, top, bot] (A) at (0,0) {};
 \node[Cyl, bot, anchor=top] (B) at (A.rightleg) {};
 \node[Copants, bot, anchor=rightleg] (C) at (A.leftleg) {};
 \node[Cyl, top, bot, anchor=bottom] (D) at (C.leftleg) {}; 
\end{tz}
\end{align}

\item (Ribbon) The twist satisfies the following equation:
\begin{equation}
\label{tortile}
\begin{tz}[xscale=1.4, yscale=2]
\node (1) at (0,0)
{
$\begin{tikzpicture}
        \node[Pants] (A) at (0,0) {};
        \node[Cap] at (A.belt) {};
        \selectpart[green, inner sep=1pt]{(A-leftleg)};
\end{tikzpicture}$
};
\node (2) at (1,0)
{
$\begin{tikzpicture}
        \node[Pants] (A) at (0,0) {};
        \node[Cap] at (A.belt) {};
\end{tikzpicture}$
};
\begin{scope}[double arrow scope]
    \draw (1) -- node[above] {$\theta$} (2);
\end{scope}
\end{tz}
\quad = \quad
\begin{tz}[xscale=1.4, yscale=2]
\node (1) at (0,0)
{
$\begin{tikzpicture}
        \node[Pants] (A) at (0,0) {};
        \node[Cap] at (A.belt) {};
        \selectpart[green, inner sep=1pt]{(A-rightleg)};
\end{tikzpicture}$
};
\node (2) at (1,0)
{
$\begin{tikzpicture}
        \node[Pants] (A) at (0,0) {};
        \node[Cap] at (A.belt) {};
\end{tikzpicture}$
};
\begin{scope}[double arrow scope]
    \draw (1) -- node[above] {$\theta$} (2);
\end{scope}
\end{tz} 
\end{equation}

\item (Adjoints) The data \eqref{A1} expresses $\tikztinypants$ as the biadjoint of $\tikztinycopants$, while (part of the) \eqref{A2} data, expresses $\tikztinycap$ as the left adjoint of $\tikztinycup$. That is, the following equations hold: 
\begin{align}
\label{adj_eta_epsilon1}
\begin{aligned}
\begin{tikzpicture}[xscale=1.6, yscale=4]
\node (1) at (0,0)
{
$\begin{tikzpicture}
    \node [Pants, top] (A) at (0,0) {};
    \node[Cyl, anchor=top] (B) at (A.leftleg) {};
    \node[Cyl, anchor=top] (C) at (A.rightleg) {};
    \fixboundingbox
    \selectpart[green]{(A-leftleg) (A-rightleg) (B-bottom) (C-bottom)}
\end{tikzpicture}$
};
\node (2) at (1,0)
{
$\begin{tikzpicture}
    \node [Pants, bot, top] (A) at (0,0) {};
    \node [Copants, anchor=leftleg] (B) at (A.leftleg) {};
    \node [Pants, anchor=belt, bot] at (B.belt) {};
    \selectpart [green] {(A-leftleg) (A-belt) (A-rightleg) (B-belt)}; 
\end{tikzpicture}$
};
\node (3) at (2,0)
{
$\begin{tikzpicture}
    \node [Pants] (A) at (0,0) {};
    \node [Cyl, tall, bot, anchor=bot, top] at (A.belt) {};
\end{tikzpicture}$
};
\begin{scope}[double arrow scope]
    \draw (1) --  node[above]{$\eta$} (2);
    \draw (2) --  node[above]{$\epsilon$} (3);
\end{scope}
\end{tikzpicture}
\end{aligned}
\quad &= \quad
\begin{aligned}
\begin{tikzpicture}[xscale=1.6, yscale=4]
\node (1) at (0,0)
{
$\begin{tikzpicture}
    \node [Pants, top] (A) at (0,0) {};
\end{tikzpicture}$
};
\node (2) at (1,0)
{
$\begin{tikzpicture}
    \node [Pants, bot, top] (A) at (0,0) {};
\end{tikzpicture}$
};
\begin{scope}[double arrow scope]
    \draw (1) --  node[above]{$\id$} (2);
\end{scope}
\end{tikzpicture}
\end{aligned}
\\
\label{adj_eta_epsilon2}
\begin{aligned}
\begin{tikzpicture}[xscale=1.6, yscale=4]
\node (1) at (0,0)
{
$\begin{tikzpicture}
    \node [Copants] (A) at (0,0) {};
    \node[Cyl, bot, anchor=bot, top] (B) at (A.leftleg) {};
    \node[Cyl, bot, anchor=bot, top] (C) at (A.rightleg) {};
    \fixboundingbox
    \selectpart[green]{(B-bottom) (C-bottom) (B-top) (C-top)}
\end{tikzpicture}$
};
\node (2) at (1,0)
{
$\begin{tikzpicture}
    \node [Copants, bot, top] (A) at (0,0) {};
    \node [Pants, anchor=belt, bot] (B) at (A.belt) {};
    \node [Copants, anchor=leftleg] (C) at (B.leftleg) {};
    \selectpart [green] {(A-belt) (B-leftleg) (B-rightleg) (C-belt)}; 
\end{tikzpicture}$
};
\node (3) at (2,0)
{
$\begin{tikzpicture}
    \node [Copants, top] (A) at (0,0) {};
    \node [Cyl, tall, anchor=top] at (A.belt) {};
\end{tikzpicture}$
};
\begin{scope}[double arrow scope]
    \draw (1) --  node[above]{$\eta$} (2);
    \draw (2) --  node[above]{$\epsilon$} (3);
\end{scope}
\end{tikzpicture}
\end{aligned}
\quad &= \quad
\begin{aligned}
\begin{tikzpicture}[xscale=1.6, yscale=4]
\node (1) at (0,0)
{
$\begin{tikzpicture}
    \node [Copants, top] (A) at (0,0) {};
\end{tikzpicture}$
};
\node (2) at (1,0)
{
$\begin{tikzpicture}
    \node [Copants, bot, top] (A) at (0,0) {};
\end{tikzpicture}$
};
\begin{scope}[double arrow scope]
    \draw (1) --  node[above]{$\id$} (2);
\end{scope}
\end{tikzpicture}
\end{aligned}
\\
\label{adj_eta_epsilon_dag1}
\begin{aligned}
\begin{tikzpicture}[xscale=1.6, yscale=4]
\node (1) at (0,0)
{
$\begin{tikzpicture}
    \node [Pants] (A) at (0,0) {};
    \node [Cyl, tall, bot, anchor=bot, top] (B) at (A.belt) {};
    \selectpart [green] {(B-bottom) (B-top)};
\end{tikzpicture}$
};
\node (2) at (1,0)
{
$\begin{tikzpicture}
    \node [Pants, bot, top] (A) at (0,0) {};
    \node [Copants, anchor=leftleg] (B) at (A.leftleg) {};
    \node [Pants, anchor=belt, bot] (C) at (B.belt) {};
    \selectpart [green] {(C-leftleg) (C-rightleg) (A-leftleg) (A-rightleg)}; 
\end{tikzpicture}$
};
\node (3) at (2,0)
{
$\begin{tikzpicture}
    \node [Pants, top] (A) at (0,0) {};
    \node[Cyl, anchor=top] (B) at (A.leftleg) {};
    \node[Cyl, anchor=top] (C) at (A.rightleg) {};
    \fixboundingbox
\end{tikzpicture}$
};
\begin{scope}[double arrow scope]
    \draw (1) --  node[above]{$\epsilon^\dagger$} (2);
    \draw (2) --  node[above]{$\eta^\dagger$} (3);
\end{scope}
\end{tikzpicture}
\end{aligned}
\quad &= \quad
\begin{aligned}
\begin{tikzpicture}[xscale=1.6, yscale=4]
\node (1) at (0,0)
{
$\begin{tikzpicture}
    \node [Pants, top] (A) at (0,0) {};
\end{tikzpicture}$
};
\node (2) at (1,0)
{
$\begin{tikzpicture}
    \node [Pants, bot, top] (A) at (0,0) {};
\end{tikzpicture}$
};
\begin{scope}[double arrow scope]
    \draw (1) --  node[above]{$\id$} (2);
\end{scope}
\end{tikzpicture}
\end{aligned}
\\
\label{adj_eta_epsilon_dag2}
\begin{aligned}
\begin{tikzpicture}[xscale=1.6, yscale=4]
\node (1) at (0,0)
{
$\begin{tikzpicture}
    \node [Copants, top] (A) at (0,0) {};
    \node [Cyl, tall, anchor=top] (B) at (A.belt) {};
    \selectpart [green] {(B-bottom) (A-belt)};
\end{tikzpicture}$
};
\node (2) at (1,0)
{
$\begin{tikzpicture}
    \node [Copants, bot, top] (A) at (0,0) {};
    \node [Pants, anchor=belt, bot] (B) at (A.belt) {};
    \node [Copants, anchor=leftleg] (C) at (B.leftleg) {};
\selectpart [green] {(B-leftleg) (B-rightleg) (A-leftleg) (A-rightleg)}; 
\end{tikzpicture}$
};
\node (3) at (2,0)
{
$\begin{tikzpicture}
    \node [Copants] (A) at (0,0) {};
    \node[Cyl, bot, anchor=bot, top] (B) at (A.leftleg) {};
    \node[Cyl, bot, anchor=bot, top] (C) at (A.rightleg) {};
    \fixboundingbox
\end{tikzpicture}$
};
\begin{scope}[double arrow scope]
    \draw (1) --  node[above]{$\epsilon^\dagger$} (2);
    \draw (2) --  node[above]{$\eta^\dagger$} (3);
\end{scope}
\end{tikzpicture}
\end{aligned}
\quad &= \quad
\begin{aligned}
\begin{tikzpicture}[xscale=1.6, yscale=4]
\node (1) at (0,0)
{
$\begin{tikzpicture}
    \node [Copants, top] (A) at (0,0) {};
\end{tikzpicture}$
};
\node (2) at (1,0)
{
$\begin{tikzpicture}
    \node [Copants, bot, top] (A) at (0,0) {};
\end{tikzpicture}$
};
\begin{scope}[double arrow scope]
    \draw (1) --  node[above]{$\id$} (2);
\end{scope}
\end{tikzpicture}
\end{aligned}
\\
\label{adj_nu_mu_dag1}
\begin{aligned}
\begin{tikzpicture}[xscale=1.6, yscale=4, every to/.style={out=down,in=up}]
\node (A) at (0,0)
{
$\begin{tikzpicture}
    \node[Cyl, top] (X) at (0,0) {};
    \node[Cup] at (X.bottom) {};
    \selectpart[green]{(X-top) (X-bottom)}
\end{tikzpicture}$
};
\node (2) at (1,0)
{
$\begin{tikzpicture}
    \node (1) [Cup, top] at (0,0) {};
    \node (2) [Cap] at (0,-1.5\cobheight) {};
    \node (3) [Cup] at (0,-1.5\cobheight) {};
    \fixboundingbox
    \selectpart[green]{(3) (2)}
\end{tikzpicture}$
};
\node (3) at (2,0)
{
$\begin{tikzpicture}
    \node (1) [Cup, top] at (0,0) {};
    \node (3) [Cap, invisible] at (0,-1.5\cobheight) {};
    \fixboundingbox
\end{tikzpicture}$
};
\begin{scope}[double arrow scope]
    \draw (A) --  node[above]{$\mu^\dagger$} (2);
    \draw (2) --  node[above]{$\nu^\dagger$} (3);
\end{scope}
\end{tikzpicture}
\end{aligned}
\quad &= \quad
\begin{aligned}
\begin{tikzpicture}[xscale=1.6, yscale=4, every to/.style={out=down,in=up}]
\node (A) at (0,0)
{
$\begin{tikzpicture}
    \node (1) [Cup, top] at (0,0) {};
\end{tikzpicture}$
};
\node (2) at (1,0)
{
$\begin{tikzpicture}
    \node[Cup, top] at (0,0) {};
\end{tikzpicture}$
};
\begin{scope}[double arrow scope]
    \draw (A) --  node[above]{$\id$} (2);
\end{scope}
\end{tikzpicture}
\end{aligned}
\\
\label{adj_nu_mu_dag2}
\begin{aligned}
\begin{tikzpicture}[xscale=1.6, yscale=4, every to/.style={out=down,in=up}]
\node (A) at (0,0)
{
$\begin{tikzpicture}
    \node[Cyl, bot] (X) at (0,0) {};
    \node[Cap, top] at (X.top) {};
    \selectpart[green]{(X-top) (X-bottom)}
\end{tikzpicture}$
};
\node (2) at (1,0)
{
$\begin{tikzpicture}
    \node (1) [Cap, top] at (0,+1.5\cobheight) {};
    \node (2) [Cup] at (0,+1.5\cobheight) {};
    \node (3) [Cap] at (0,0) {};
    \fixboundingbox
    \selectpart[green]{(2) (1)}
\end{tikzpicture}$
};
\node (3) at (2,0)
{
$\begin{tikzpicture}
    \node (1) [Cap] at (0,0) {};
    \node (3) [Cap, invisible] at (0,+1.5\cobheight) {};
    \fixboundingbox
\end{tikzpicture}$
};
\begin{scope}[double arrow scope]
    \draw (A) --  node[above]{$\mu^\dagger$} (2);
    \draw (2) --  node[above]{$\nu^\dagger$} (3);
\end{scope}
\end{tikzpicture}
\end{aligned}
\quad &= \quad
\begin{aligned}
\begin{tikzpicture}[xscale=1.6, yscale=4, every to/.style={out=down,in=up}]
\node (A) at (0,0)
{
$\begin{tikzpicture}
    \node (1) [Cap, top] at (0,0) {};
\end{tikzpicture}$
};
\node (2) at (1,0)
{
$\begin{tikzpicture}
    \node[Cap, top] at (0,0) {};
\end{tikzpicture}$
};
\begin{scope}[double arrow scope]
    \draw (A) --  node[above]{$\id$} (2);
\end{scope}
\end{tikzpicture}
\end{aligned}
\end{align}

 \item (Pivotality) The following equation holds, together with its rotation about the $z$-axis $\eqref{piv_on_sphere} {}^z$:  \begin{equation} \label{piv_on_sphere}
\begin{tz}[xscale=1.4, yscale=2]
\node (1) at (0,0)
{
$\begin{tikzpicture}
 \node[Cap] (A) at (0,0) {};
 \node[Cup, bot=false] (B) at (0,0) {};
 \node[Cobordism Bottom End 3D] (AA) at (0,0) {};              
 \selectpart[green, inner sep=1pt] {(AA)};
\end{tikzpicture}$
};

\node (2) at (1,0)
{
$\begin{tikzpicture}
    \node [Pants] (A) at (0,0) {};
        \node[Cap] at (A.belt) {};
    \node[Copants, anchor=leftleg] (B) at (A.leftleg) {};
    \node[Cup] at (B.belt) {};
    \selectpart[green, inner sep=1pt] {(A-rightleg)};
\end{tikzpicture}$
};

\node (3) at (2,0)
{
$\begin{tikzpicture}
    \node [Pants] (A) at (0,0) {};
        \node[Cap] at (A.belt) {};
    \node[Cyl, tall, anchor=top] (B) at (A.leftleg) {};
    \node[Cup] (C) at (A.rightleg) {};
    \node[Copants, anchor=leftleg] (D) at (B.bottom) {};
    \node[Cap, bot=false] (E) at (D.rightleg) {};
    \node [Cobordism Bottom End 3D] (B) at (D.rightleg) {};
    \node[Cup] at (D.belt) {};
    \selectpart[green] {(A-rightleg) (B)};
\end{tikzpicture}$
};

\node (4) at (3,0)
{
$\begin{tikzpicture}
    \node [Pants] (A) at (0,0) {};
        \node[Cap] (X) at (A.belt) {};
    \node[Copants, anchor=leftleg] (B) at (A.leftleg) {};
    \node[Cup] at (B.belt) {};
    \selectpart[green] {(X-center) (A-leftleg) (A-rightleg) (B-belt)};
\end{tikzpicture}$
};

\node (5) at (4,0)
{
$\begin{tikzpicture}
    \node [Cap] at (0,0) {};
    \node[Cup] at (0,0) {};
    \end{tikzpicture}$
};

\begin{scope}[double arrow scope]
    \draw (1) -- node[above] {$\epsilon^\dagger$} (2);
    \draw (2) -- node[above] {$\mu^\dagger$} (3);
    \draw (3) -- node[above] {$\mu$} (4);
    \draw (4) -- node[above] {$\epsilon$} (5);
\end{scope}
\end{tz}
\quad = \quad 
\begin{tz}[xscale=1.4, yscale=2]
\node (1) at (0,0)
{
$\begin{tikzpicture}
 \node[Cap] (A) at (0,0) {};
 \node[Cup] (B) at (0,0) {};
\end{tikzpicture}$
};

\node (2) at (1,0)
{
$\begin{tikzpicture}
 \node[Cap] (A) at (0,0) {};
 \node[Cup] (B) at (0,0) {};
\end{tikzpicture}$
};

\begin{scope}[double arrow scope]
    \draw (1) -- node[above] {$\id$} (2);
\end{scope}
\end{tz}
\end{equation}

\item (Modularity) The following equation holds, together with its rotation about the $z$-axis $\eqref{MOD}^z$:\begin{equation}
\label{MOD}
\begin{tz}[xscale=2, yscale=2]
\node (1) at (0,-0.5)
{$\begin{tikzpicture}
        \node[Cyl, top, tall] (A) at (0,0) {};
\end{tikzpicture}$};

\node (2) at (1,0)
{$\begin{tikzpicture}
        \node[Pants, top] (A) at (0,0) {};
        \node[Copants, anchor=leftleg] (B) at (A.leftleg) {};
        \selectpart[green, inner sep=1pt]{(A-leftleg)};
        \selectpart[red, inner sep=1pt] {(A-rightleg)};
\end{tikzpicture}$};

\node (3) at (2,0)
{$\begin{tikzpicture}
        \node[Pants, top] (A) at (0,0) {};
        \node[Copants, anchor=leftleg] (B) at (A.leftleg) {};
\end{tikzpicture}$};

\node (4) at (3,-0.5)
{$\begin{tikzpicture}
        \node[Cyl, top, tall] (A) at (0,0) {};
\end{tikzpicture}$};

\node(5) at (1.5, -1)
{$\begin{tikzpicture}
        \node[Cup, top] (A) at (0,0) {};
        \node[Cap] (B) at (0, -1.5*\cobheight) {};
\end{tikzpicture}$};
    
\begin{scope}[double arrow scope]
    \draw (1) -- node [above] {$\epsilon^\dagger$} (2);
    \draw[] ([xshift=0pt] 2.0) to node [above, inner sep=1pt] {${\color{green}\theta}, \color{red}{\theta^\inv}$} (3);
    \draw (3) -- node [above] {$\epsilon$} (4);
    \draw (1) -- node [below left] {$\mu^\dagger$} (5);
    \draw (5) -- node [below right] {$\mu$} (4);
\end{scope}
\end{tz}
\end{equation}

\item (Anomaly) The following equation holds:
\begin{equation} \label{Anom}
\begin{tz}
        \node[Cap] (A) at (0,0) {};
        \node[Cup] (B) at (0,0) {};
        \node [Cobordism Bottom End 3D] (C) at (0,0) {};
        \selectpart[green, inner sep=1pt] {(C)};
\end{tz}
\,\,
\Rarrow{\epsilon^\dagger}
\,\,
\begin{tz}
        \node[Cap] (A) at (0,0) {};
        \node[Pants, anchor=belt] (B) at (A.center) {};
        \node[Copants, anchor=leftleg] (C) at (B.leftleg) {};
        \node[Cup] (D) at (C.belt) {};
        \selectpart[green, inner sep=1pt] {(B-rightleg)};
\end{tz}
\,\, \Rarrow{\theta} \,\,
\begin{tz}
        \node[Cap] (A) at (0,0) {};
        \node[Pants, anchor=belt] (B) at (A.center) {};
        \node[Copants, anchor=leftleg] (C) at (B.leftleg) {};
        \node[Cup] (B) at (C.belt) {};
        \selectpart[green] {(A-center) (B-leftleg) (B-rightleg) (C-belt)};
\end{tz}
\,\, \Rarrow{\epsilon} \,\,
\begin{tz}
        \node[Cap] (A) at (0,0) {};
        \node[Cup] (B) at (0,0) {};
\end{tz}
\quad = \quad
\begin{tz}
    \node[Cap] (A) at (0,0) {};
    \node[Cup] (B) at (0,0) {};
\end{tz}
\,\, \Rarrow{\xi} \,\,
\begin{tz}
    \node[Cap] (A) at (0,0) {};
    \node[Cup] at (0,0) {};
\end{tz}
 \end{equation}

\end{itemize}
\end{definition}

\begin{remark}\label{anom_root1}
    If we ask that the anomaly relation holds for $\xi^3$ instead of $\xi$, as well as impose another relation (\cite[Relation (40)]{bartlett_modular_2015}, the resulting presentation should be equivalent to the 2-category of non-compact 3d cobordisms equipped with $p_1$ structure, $\bBordncp$. We do not give a definition of $\bBordncp$, but we make a note of this as we could have defined the non-compact TQFT as a functor out of the $p_1$ presentation $\Bordncp$. The only difference is that we would have to consider a sixth root of the anomaly of the MTC $\sqrt[\leftroot{-2}\uproot{2}6]{p_+/p_-}$, instead of $\sqrt{p_+/p_-}$ (see Remark \ref{anom_root2}).
\end{remark}

\subsection{Finite ribbon categories and projectives}\label{finiteproj}

Following \cite{etingof_tensor_2015} we have the following definition:

\begin{definition}
A linear category $\mathcal{C}$ is \textbf{finite} if it is abelian and:
\begin{itemize}
    \item it has a finite number of isomorphism classes of simple objects,
    \item it has enough projectives,
    \item its objects have finite length,
    \item its morphism spaces have finite dimension.
\end{itemize}
\end{definition}


For now by $\C$ we denote a finite category.

Let $I$ be the set of isomorphism classes of simples. For $j\in I$, let $L_j$ be a representative of the $j$-th class of simples. Then $P_j\twoheadrightarrow L_j$ is a projective cover of $L_j$. If $\C$ is monoidal and $\unit$ is simple, we will denote its projective cover by $(P_\unit,\varepsilon_\unit)$, where $\varepsilon_\unit\colon P_\unit\to \unit$ is the canonical surjection.

 By $\text{Proj}(\C)$ we denote the full subcategory of projective objects in $\C$. This will turn out to be a tensor ideal once we additionally assume $\C$ to be rigid monoidal.

For such a $\C$, we have the following results:

\begin{itemize}
    \item The indecomposable projective objects in $\C$ are exactly the projective covers of simple objects, and for $j,j^\prime$ simples, $P_j\cong P_{j^\prime}$ if and only if $L_j\cong L_{j^\prime}$.
    \item For a finite linear category $\C$, $\text{Proj}(\C)$ is generated under direct sums    by the projective covers of simples $P_j$.
    \item     Let $L_j,L_{j'}\in\C$ be simple objects in $\C$ and $P_j$ be the projective cover of $L_j$, then \[ \Hom_\C(P_j,L_{j^\prime})\cong \begin{cases} 
      k & \text{if} \quad j= j^\prime \\
      0 & \text{otherwise} 
   \end{cases}
\]
\item  Let $P\in \text{Proj}(\C)$, then by the above we have:
    $$\Hom_\C(P,\unit)\cong\Hom_\C(\oplus_j P_j^{\oplus n_j},\unit)\cong \oplus_{n_j}\Hom_\C(P_j,\unit)\cong \oplus_{n_\unit}\Hom_\C(P_\unit,\unit).$$
\end{itemize}

See \cite{leinster_bijection_2015} for a reference.


Now let $V$ be a finite dimensional vector space and $X,Y\in \C$. The fact that $\C$ is linear allows us to define $V\otimes Y$ by the following universal property: 

\begin{equation}\label{copower_def}
    \Hom_\C(V\otimes Y,X)\cong \Hom_{\text{Vec}_k}(V,\Hom_\C(Y,X)).
\end{equation}

From this equivalence we obtain the following natural (counit) map:

\begin{equation}\label{evaluation-counit}
    \text{EV}_Y\colon \text{Hom}_\C(Y,X)\otimes Y \to X.
\end{equation}

\subsection{Coends}

We now assume that $\C$ is a finite tensor category in the sense of \cite{etingof_tensor_2015}. We remind that this definition includes rigidity. 


We review the definition of a coend, as well as some results that will be used later on.

\begin{definition}
Let $\mathcal{D}$ be a finite linear category and $B:\mathcal{C}^{op}\times\mathcal{C}\to\mathcal{D}$ be a bifunctor. A \textbf{dinatural transformation} from $B$ to a constant bifunctor is given by an object $Z\in\mathcal{D}$ (the image of the constant bifunctor) equipped with a family $\{d_X:B(X,X)\to Z\}_{X\in\mathcal{C}}$ of morphisms of $\mathcal{C}$ satisfying:

$$d_X\circ(f^\lor\otimes id_X)=d_Y\circ(id_{Y^\lor}\otimes f)$$

for every morphism $f:X\to Y$ in $\mathcal{C}$.

For simplicity, we will call this a dinatural transformation of source B.
\end{definition}

\begin{definition}
The \textbf{coend} of $B$ is a universal dinatural transformation of source $B$, i.e an object $\displaystyle\int^{X\in\C}\hspace{-20pt}B(X,X)\in\mathcal{D}$ with a dinatural family of structure morphisms ${d_X:B(X,X)\to \displaystyle\int^{X\in\C}\hspace{-20pt}B(X,X)}$ satisfying the following condition: for every dinatural transformation $(Z,d^\prime)$ of source $B$, there exists a unique morphism
$${\int d^\prime:\displaystyle\int^{X\in\C}\hspace{-20pt}B(X,X)\to Z}$$ 
such that the diagram:

\[\begin{tikzcd}
	{Y^\lor\otimes X} &&& {X^\lor\otimes X} \\
	\\
	\\
	{Y^\lor\otimes Y} &&& {\displaystyle\int^{X\in\C}\hspace{-20pt}B(X,X)} \\
	\\
	&&&&& Z
	\arrow["{f^\lor\otimes id_X}", from=1-1, to=1-4]
	\arrow["{id_{Y^\lor}\otimes f}"', from=1-1, to=4-1]
	\arrow["{d_Y}"', from=4-1, to=4-4]
	\arrow["{d_X}", from=1-4, to=4-4]
	\arrow["{\int \!\!d^\prime}", from=4-4, to=6-6]
	\arrow["{d^\prime_X}", from=1-4, to=6-6]
	\arrow["{d^\prime_Y}", from=4-1, to=6-6]
\end{tikzcd}\]

commutes for every morphism $f:X\to Y$ in $\mathcal{C}$.
\end{definition}

In practice, we sometimes will denote $\int d^\prime$ by $d^\prime$.


As explained in \cite{kerler_non-semisimple_2001}, $EV$ is dinatural, and according to \cite[Prop 5.1.5]{kerler_non-semisimple_2001}:

\begin{prop}\label{coYoneda}
    The natural morphism $\int\text{EV}_Y\colon\displaystyle\int^{Y} \hspace{-10pt}\text{Hom}_\C(Y,X)\otimes Y\to X$ is an isomorphism. This is true for $Y\in\C$ as well as $Y\in\text{Proj}(\C)$, due to the fact that $\C$ has enough projectives.
\end{prop}

\begin{remark}
    This is an instance of the so-called co-Yoneda lemma, or ninja Yoneda lemma, as encountered in \cite[Chapter 2]{loregian_coend_2021}.
\end{remark}

Now, using the fact that the object $G=\oplus_j P_j$ is a projective generator, \cite[Prop 5.1.7]{kerler_non-semisimple_2001} gives us the following result:

\begin{prop}\label{coend_redux}
    For a bifunctor $B\colon \mathcal{C}^{op}\times\mathcal{C}\to\mathcal{D}$ that is exact in each variable, the following coends are equivalent: 
    $$\int^{X\in\C}\hspace{-15pt}B(X,X)\cong\int^{P\in\text{Proj}(\C)}\hspace{-35pt}B(P,P)\cong\int^{\{P_j\}_{j=simple}}\hspace{-40pt}B(P_j,P_j).$$
\end{prop}

Using the above, we can reduce certain coends to be over projectives, or even indecomposable projectives. We will use this fact later on.

From now on we assume that $\C$ is additionally ribbon, and that the monoidal unit $\unit$ is simple.

For a finite ribbon category $\C$, it was shown by Majid \cite{majid_reconstruction_1991} and Lyubashenko \cite{lyubashenko_modular_1995}, \cite{lyubashenko_squared_1999}, that the coend 

\begin{equation}\label{coend_def}
    \coend:= \int^{X\in\C}\hspace{-15pt} X^\lor \otimes X
\end{equation}

exists. We denote its family of structure morphisms by $i_X\colon X^\lor \otimes X \to \coend$.


Additionally, they showed that the coend $\coend$ is a Hopf algebra object in $\C$.

The structure morphisms for the Hopf algebra structure will be denoted as:

\begin{align*}
    \text{Product} \quad m \colon \coend\otimes\coend\to \coend && \text{Unit} \quad u \colon \unit\to \coend \\
    \text{Coproduct} \quad \Delta \colon \coend\to \coend\otimes\coend && \text{Counit} \quad \varepsilon \colon \coend\to \unit \\
    \text{Antipode} \quad S \colon \coend\to\coend.
\end{align*}

and are defined as follows:

\tikzset{morphism/.style={draw, fill=white}}
\tikzset{box/.style={draw, font=\small, node on layer=foreground, fill=white, inner sep=0pt, minimum height=0.51cm, minimum width=1.5cm}}
\tikzset{halfbox/.style={box, minimum width=0.4cm}}
\tikzset{1halfbox/.style={box, minimum width=0.6cm}}
\tikzset{2halfbox/.style={box, minimum width=0.8cm}}
\tikzset{every node/.style={font=\tiny}}
\def\p{{\vphantom{(}}}

\begin{align}
    \label{algebra}
\begin{tz}[xscale=0.5,yscale=0.7]
\node [2halfbox] at (2,1.5) {$i_Y$};
\node [2halfbox] at (0,1.5) {$i_X$};
\node [box] at (1,3) {$m$};
\draw [purple strand] (0,1.5) to (0,3);
\draw [purple strand] (2,1.5) to (2,3);
\draw [purple strand] (1,3) to (1,4);
\draw (-0.5,0) node [below] {$X^\lor$} to (-0.5,1.5);
\draw (0.5,0) node [below] {$X$} to (0.5,1.5);
\draw (1.5,0) node [below] {$Y^\lor$} to (1.5,1.5);
\draw (2.5,0) node [below] {$Y$} to (2.5,1.5);
\node at (2,2.25) [right] {$\mathcal{F}$};
\node at (0,2.25) [left] {$\mathcal{F}$};
\end{tz}
\gap=\gap
\begin{tz}[xscale=0.4,yscale=0.7]
\draw [purple strand] (1.5,2) to (1.5,3) node [above] {$\mathcal{F}$};
\node [box] at (1.5,2) {$i_{X\otimes Y}$};
\draw [black strand] (2,0) node [below] {$Y^\lor$} to [out=up, in=down] (0,1.5);
\draw [black strand] (0,0) node [below] {$X^\lor$} to [out=up, in=down] (1,1.5);
\draw [black strand] (1,0) node [below] {$X$} to [out=up, in=down] (2,1.5);
\draw [black strand] (3,0) node [below] {$Y$} to [out=up, in=down] (3,2);
\draw [black strand] (2,1.5) to [out=up, in=down] (2,2);
\draw [black strand] (0,1.5) to [out=up, in=down] (0,2);
\draw [black strand] (1,1.5) to [out=up, in=down] (1,2);
\end{tz}
&&
    \begin{tz}[xscale=0.6,yscale=0.5]
\node [halfbox] at (0.5,0.5) {$u$};
\draw [purple strand] (0.5,0.5) to (0.5,2);
\node at (0.5,2) [above] {$\mathcal{F}$};
\end{tz}
\gap=\gap
\begin{tz}[xscale=0.6,yscale=0.5]
\node [halfbox] at (0.5,0.5) {$i_\unit$};
\draw [purple strand] (0.5,0.5) to (0.5,2);
\node at (0.5,2) [above] {$\mathcal{F}$};
     \end{tz}
\end{align}

\begin{align}
    \label{coalgebra}
\begin{tz}[xscale=0.5,yscale=0.7]
\node (1) [box] at (1,1) {$i_X$};
\node [box] at (1,2.5) {$\Delta$};
\draw [purple strand] (0,2.5) to (0,3.5) node [above] {$\coend$};
\draw [purple strand] (2,2.5) to (2,3.5) node [above] {$\mathcal{F}$};
\draw [purple strand] (1,1) to (1,2.5);
\draw (0,0) node [below] {$X^\lor$} to (0,1);
\draw (2,0) node [below] {$X$} to (2,1);
\node at (1,1.75) [right] {$\mathcal{F}$};
\end{tz}
\gap=\gap
\begin{tz}[xscale=0.5,yscale=0.7]
\draw [purple strand] (0,2.5) to (0,3.5) node [above] {$\mathcal{F}$};
\draw [purple strand] (2,2.5) to (2,3.5) node [above] {$\mathcal{F}$};
\node [2halfbox] at (2,2.5) {$i_X$};
\node [2halfbox] at (0,2.5) {$i_X$};
\draw [black strand] (2,0) node [below] {$X$} to (2,1);
\draw [black strand] (0,0) node [below] {$X^\lor$} to (0,1);
\draw [black strand] (0,1) to [out=up, in=up, looseness=0.9] (-0.5,2.5);
\draw [black strand] (2,1) to [out=up, in=up, looseness=0.9] (2.5,2.5);
\draw [black strand] (0.5,2.5) to [out=down, in=down, looseness=4] (1.5,2.5);
\node at (2,1) {\arrownode[black]};
\node [rotate=180] at (0,1) {\arrownode[black]};
\node at (0.54,2) {\arrownode[black]};
\end{tz}
&&
    \begin{tz}[xscale=0.8,yscale=0.7]
\node (1) [box] at (0.5,1) {$i_X$};
\node [halfbox] at (0.5,2.5) {$\varepsilon$};
\draw [purple strand] (0.5,1) to (0.5,2.5);
\draw (0,0) node [below] {$X^\lor$} to (0,1);
\draw (1,0) node [below] {$X$} to (1,1);
\node at (0.5,1.75) [right] {$\mathcal{F}$};
\end{tz}
\gap=\gap
\begin{tz}[xscale=0.8,yscale=0.7]
\draw (0,0) node [below] {$X^\lor$} to (0,1.5);
\draw (1,0) to (1,1.5);
\draw [black strand] (0,1.5) to [out=up, in=up, looseness=1.2] (1,1.5);
\draw [black strand] (1,0) node [below] {$X$} to (1,1.5);
\node at (1,1) {\arrownode[black]};
\node [rotate=180] at (0,1) {\arrownode[black]};
     \end{tz}
\end{align}

\begin{align}
    \begin{tz}[xscale=0.5,yscale=0.7]
\node (1) [box] at (1,1) {$i_X$};
\node [halfbox] at (1,2.5) {$S$};
\draw [purple strand] (1,2.5) to (1,3.5) node [above] {$\coend$};
\draw [purple strand] (1,1) to (1,2.5);
\draw (0,0) node [below] {$X^\lor$} to (0,1);
\draw (2,0) node [below] {$X$} to (2,1);
\node at (1,1.75) [right] {$\mathcal{F}$};
\end{tz}
\gap=\gap
\begin{tz}[xscale=0.5,yscale=0.7]
\draw [purple strand] (1,2.5) to (1,3.5) node [above] {$\mathcal{F}$};
\node [box] at (1,2.5) {$i_X$};
\draw [black strand] (0,0) node [below] {$X^\lor$} to (0,1);
\draw [black strand] (2.5,0.5) to [out=down, in=down, looseness=2] (2,1);
\draw [black strand] (2,0) node [below] {$X$} to [out=up, in=up, looseness=2] (2.5,0.5);
\draw [black strand] (0,2) to (0,2.5);
\draw [black strand] (2,2) to (2,2.5);
\draw [black strand] (2,1) to [out=up, in=down, looseness=0.5] (0,2);
\draw [black strand] (0,1) to [out=up, in=down, looseness=0.5] (2,2);
\end{tz}
\gap=\gap
\begin{tz}[xscale=0.5,yscale=0.7]
\draw [purple strand] (1,2.5) to (1,3.5) node [above] {$\mathcal{F}$};
\node [box] at (1,2.5) {$i_X$};
\node[halfbox] at (2,0.75) {$\vartheta_X$};
\draw [black strand] (0,0) node [below] {$X^\lor$} to (0,1);
\draw [black strand] (2,0) node [below] {$X$} to (2,1);
\draw [black strand] (0,2) to (0,2.5);
\draw [black strand] (2,2) to (2,2.5);
\draw [black strand] (2,1) to [out=up, in=down, looseness=0.5] (0,2);
\draw [black strand] (0,1) to [out=up, in=down, looseness=0.5] (2,2);
\end{tz}
\end{align}

Additionally, the coend $\coend$ has a Hopf pairing $\omega \colon \coend\otimes\coend\to \unit$. It is uniquely defined in the following way:

\begin{align}
    \label{omega_pairing}
\begin{tz}[xscale=0.5,yscale=0.7]
\node [2halfbox] at (2,1.5) {$i_Y$};
\node [2halfbox] at (0,1.5) {$i_X$};
\node [box] at (1,3) {$\omega$};
\draw [purple strand] (0,1.5) to (0,3);
\draw [purple strand] (2,1.5) to (2,3);
\draw (-0.5,0) node [below] {$X^\lor$} to (-0.5,1.5);
\draw (0.5,0) node [below] {$X$} to (0.5,1.5);
\draw (1.5,0) node [below] {$Y^\lor$} to (1.5,1.5);
\draw (2.5,0) node [below] {$Y$} to (2.5,1.5);
\node at (2,2.25) [right] {$\mathcal{F}$};
\node at (0,2.25) [left] {$\mathcal{F}$};
\end{tz}
\gap=\gap
\begin{tz}[xscale=0.4,yscale=0.7]
\draw [black strand] (2,0) node [below] {$Y^\lor$} to [out=up, in=down] (1,1);
\draw [black strand] (1,0) node [below] {$X$} to [out=up, in=down] (2,1);
\draw [black strand] (2,1) to [out=up, in=down] (1,2);
\draw [black strand] (1,1) to [out=up, in=down] (2,2);
\draw (0,0) node [below] {$X^\lor$} to (0,2);
\draw (3,0) node [below] {$Y$} to (3,2);
\draw [black strand] (0,2) to [out=up, in=up, looseness=1.2] (1,2);
\draw [black strand] (2,2) to [out=up, in=up, looseness=1.2] (3,2);
\end{tz}
\end{align}

\begin{definition}\label{MTC_def}
A braided finite tensor category $\C$ is \textbf{non-degenerate}, if the Hopf pairing $\omega$ is non-degenerate.
\end{definition}

For an object $Y\in \C$ we can define $\varrho_Y$ by the composite: 
\begin{equation}\label{canonical_coaction}
    \varrho_Y\colon Y\xrightarrow{\text{coev}_Y\otimes\id_Y} Y\otimes Y^\lor\otimes Y\xrightarrow{\id_Y\otimes i_Y} Y\otimes \coend.
\end{equation}

Given the definition of the comultiplication $\Delta$ of $\coend$, it is easy to see that $(Y,\varrho_Y)$ is a $\coend$-comodule. The morphism $\varrho_Y$ is also referred to as the canonical coaction. In particular, the coaction for the free comodule $X\otimes\coend$ is given by the comultiplication $\Delta$ of $\coend$. We will be using this fact implicitly in our computations. For a proof of this, see Corollary \ref{coaction_coend}.
By $\C^\coend$ we denote the category of right $\coend$ comodules in $\C$.

  An important fact we will make use of in our construction in Section \ref{construction}, is the following: 

\begin{lemma}\label{comod_equiv}
    For a braided finite tensor category $\C$, the functor
    $$\C\boxtimes\C\to \C^\coend, \quad X\boxtimes Y\mapsto (X\otimes Y,\id_V\otimes \varrho_Y)$$
    is an equivalence.
\end{lemma}

This is a special case of a more general result of Lyubashenko’s, using the notion of squared coalgebras \cite[Section 2.7]{lyubashenko_squared_1999}. 



\subsection{Integrals and modular tensor categories}

\begin{definition}
A \textbf{modular tensor category} is a finite ribbon category that is non-degenerate.
\end{definition}

Note that this definition does not assume semisimplicity. See \cite{shimizu_non-degeneracy_2019} for other equivalent conditions to non-degeneracy of $\omega$.

We will now state some properties of modular tensor categories, that will be useful for our construction.


\begin{definition}\label{integral_def}
    A morphism $\Lambda\colon \unit\to\coend$ is called a (two-sided) \textbf{integral} of $\coend$ if it satisfies

    $$m\circ(\Lambda\otimes \id_\coend)=\Lambda\circ\varepsilon = m\circ(\id_\coend\otimes\Lambda).$$

    The dual notion is that of a (two-sided) \textbf{cointegral}, which is a morphism $\Lambda^{co}\colon \coend\to \unit$ that satisfies

    $$(\Lambda^{co}\otimes \id_\coend)\circ \Delta = \eta\circ \Lambda^{co} = (id_\coend\otimes\Lambda^{co})\circ \Delta. $$
\end{definition}

\begin{prop}\label{exint}
    If $\C$ is modular, integrals and cointegrals exist and are unique up to scalar.
\end{prop}

\begin{remark}
    Proposition \ref{exint} holds when $\C$ is unimodular. But as shown in \cite[Lemma 5.2.8]{kerler_non-semisimple_2001}, modularity implies unimodularity.
\end{remark}

\begin{lemma}\label{integral_antipode_relations}
    Let $\Lambda$ and $\Lambda^{co}$ be a non-zero integral and a non-zero cointegral of $\coend$ respectively.
    Then the following hold:
    \begin{enumerate}
        \item $\Lambda^{co}\circ\Lambda \neq 0$
        \item $S\circ\Lambda=\Lambda$
        \item $\Lambda^{co}\circ m\circ(S\otimes \id_\coend)=\Lambda^{co}\circ m\circ(\id_\coend\otimes S)$
    \end{enumerate}
\end{lemma}

\begin{proof}
    Part (1) is proven in \cite[Theorem 4.2.5]{kerler_non-semisimple_2001} and (2) follows from equation (a) in \cite[Lemma 3.13]{davydov_mathbbz2mathbbz-extensions_2013}, after composing with $\id_\coend\otimes \Lambda^{co}$ from the left and $\Lambda$ from the right.

    We will now prove part (3):

\begin{align*}
    \begin{tz}[xscale=0.5,yscale=0.7]
\node [halfbox] at (0,1) {$S$};
\node [2halfbox] at (0.5,2) {$m$};
\node [halfbox] at (0.5,3) {$\Lambda^{co}$};
\draw [purple strand] (1,0) node [below] {$\coend$} to (1,2);
\draw [purple strand] (0,0) node [below] {$\coend$} to (0,2);
\draw [purple strand] (0.5,2) to (0.5,3);
\end{tz}
\gap=\gap
    \begin{tz}[xscale=0.5,yscale=0.7]
\node [halfbox] at (0,2) {$S$};
\node [2halfbox] at (0.5,3) {$m$};
\node [halfbox] at (0.5,4) {$\Lambda^{co}$};
\draw [purple strand] (1,1) to (1,3);
\draw [purple strand] (0,0) node [below] {$\coend$} to (0,3);
\draw [purple strand] (0.5,3) to (0.5,4);
\node [2halfbox] at (1.5,1) {$\Delta$};
\node [halfbox] at (3,2) {$S$};
\draw [purple strand] (2,1) to (2,3);
\draw [purple strand] (1.5,0) node [below] {$\coend$} to (1.5,1);
\node [2halfbox] at (2.5,3) {$m$};
\draw [purple strand] (2.5,3) to (2.5,4);
\node [halfbox] at (2.5,4) {$\Lambda^{co}$};
\draw [purple strand] (3,0) node [below] {$\coend$} to (3,3);
\end{tz}
\gap=\gap
    \begin{tz}[xscale=0.5,yscale=0.7]
\node [halfbox] at (1,1) {$S$};
\node [2halfbox] at (0.5,2) {$m$};
\node [halfbox] at (0.5,3) {$\Lambda^{co}$};
\draw [purple strand] (1,0) node [below] {$\coend$} to (1,2);
\draw [purple strand] (0,0) node [below] {$\coend$} to (0,2);
\draw [purple strand] (0.5,2) to (0.5,3);
\end{tz}
\end{align*}

Where in the first equality we used \cite[Figure 4.3(c)]{kerler_non-semisimple_2001}, precomposed with the antipode $S$, and in the second, \cite[Figure 4.3(d)]{kerler_non-semisimple_2001}, again precomposed with $S$. Note that for the above to be true, we used the fact that $\C$ is modular, and hence the object of integrals appearing in \cite{kerler_non-semisimple_2001} is in fact, the tensor unit.
    
\end{proof}

\begin{definition}\label{twistnondegen}
    The morphism $\mathcal{T}\colon\coend\to\coend$ is the unique morphism satisfying for all $X\in\C$ the following identity:

    $$\mathcal{T}\circ i_X = i_X\circ(id_{X^\lor}\otimes \mathcal{\vartheta}_X).$$

    For a modular tensor category $\C$ with an integral $\Lambda$, there exist non-zero constants $p_\pm$, such that:

    $$\varepsilon\circ\mathcal{T}^\pm\circ \Lambda = p_\pm \id_\unit.$$
    
\end{definition}





This lemma was proven in \cite[Thm. 5]{kerler_genealogy_1996}:

\begin{lemma}\label{modularity_parameter}
    Let $\C$ be modular and let $\Lambda$ and $\Lambda^{co}$ be an integral and a cointegral of $\coend$ satisfying $\Lambda^{co}\circ\Lambda = \id_\unit$. Then, there exists a non-zero coefficient $\zeta \in k^\times$ satisfying the equation $$\omega\circ (\Lambda\otimes \id_\coend) = \zeta\Lambda^{co}.$$
We call $\zeta$ the \textbf{modularity parameter} of $\Lambda$. 
\end{lemma}

The following lemma is proven in \cite[Cor 4.6]{de_renzi_3-dimensional_2022}:

\begin{lemma}\label{zeta_p}
    If $\C$ is a modular category, then $\zeta=p_-\cdot p_+$.
\end{lemma}

The result below was proven in \cite[Cor. 6.4]{gainutdinov_projective_2020}.

\begin{lemma}\label{coint_eta_eps}
    If $\C$ is a modular category, then there exists a unique morphism $\eta_\unit :\unit \to P_\unit$ satisfying, for every simple object $j\in \C$, the equation 
    $$\Lambda^{co}\circ i_{P_j} = \delta_{P_j,\unit} \eta_\unit^\lor\otimes \varepsilon_\unit.$$
\end{lemma}



\subsection{Target bicategory}\label{target}

\begin{definition}[Lincat]
    The symmetric monoidal bicategory Lincat is defined as follows:
    \begin{itemize}
        \item Its objects are finite abelian linear categories 
        \item Its 1-morphisms are right exact functors, i.e functors that preserve finite colimits. We denote by $\text{Fun}^{Lincat}(-,-)\subset \text{Fun}(-,-)$, the full subcategory of right exact linear functors.
        \item Its 2-morphisms are natural transformations.
    \end{itemize}
\end{definition}

There is a natural tensor product on Lincat defined by Kelly \cite{kelly_basic_1982} (generalizing the Deligne tensor product of
finite abelian categories) which is uniquely characterized by the universal property that a functor $\C\otimes \mathcal{D}\to \mathcal{E}$
preserving finite colimits is the same as a functor $\C\times \mathcal{D}\to \mathcal{E}$ preserving finite colimits separately in $\C$ and $\mathcal{D}$. The unit of this tensor product is the category $\text{Vec}_k$ of finite-dimensional $k$-vector spaces.

\begin{remark}\label{proj_res}
The 1-morphisms in Lincat are right exact functors, and the modular tensor category we will be working with has enough projectives. In this case, a right exact functor is determined by its values on projectives. This means that $Fun^{Lincat}(\C,\text{Vec}_k)\cong Fun(\text{Proj}(\C),\text{Vec}_k)$. When defining the TFT, this will allow us to specify certain 1-morphism generators only on the projective objects. Better yet, it is sufficient to talk about the indecomposable projectives, since the direct summands of a projective are projective.

Moreover, we will implicitly identify functors $F\colon \text{Vec}_k \to \C$ with their value $F(k)$, an object of $\C$.
\end{remark}

\newpage

\section{Construction of the TQFT}\label{construction}

According to the definition of a presentation of a symmetric monoidal 2-category (Section \ref{presentations}) and Conjecture \ref{conjecture}, to define the non-compact 3d TQFT, we need to specify the images of the generators in the target bicategory Lincat and to check that all the relevant relations of the non-compact presentation given in Section \ref{modpres} hold.

\subsection{Assignment to generators} \label{genass}

Fix $\C$ to be a (not necessarily semisimple) modular tensor category, and $\cP:=\text{Proj}(\C)$. Thus, $\C\in \text{Lincat}$. Additionally, let $\mathscrsfs{D}\in k^\times$ be a constant whose value is going to be fixed later by certain relations, and fix an integral $\Lambda$, whose existence is guaranteed by Proposition \ref{exint}. This uniquely determines a cointegral $\Lambda^{co}$ satisfying $\Lambda^{co}\circ\Lambda = \id_\unit$ and a modularity parameter $\zeta\in k^\times$, according to Lemma \ref{modularity_parameter}. 

For the convenience of the reader, before we give the full assignment in detail, we repeat our table \ref{genass_table_data} from the introduction. 

\scalecobordisms{0.6}

\begin{table}[h!]
\centering
 \begin{tabular}{||c | c||} 
 \hline
 Generators of $\Bordncsig$& Value in Lincat \\ [0.5ex] 
 \hline\hline
$\begin{tz}
\node[Cyl, top, height scale=0]  at (0,0) {};
\end{tz}  $ & $\C$    \\ [0.5ex] 
 $\tinypants$&   $
 \text{Forget}\colon\C^\coend\to \C$ \\ [0.5ex] 
 $\tinycopants$ &$\text{Co-free}\colon\C\to \C^\coend$\\ [0.5ex] 
 $\tinycup$    & $-\otimes\unit\colon\Vecfd\to\C$ \\ [0.5ex] 
 $\tinycap$&   $\Hom_\C(-,\unit)^*\colon\C\to \Vecfd$\\ [0.5ex] 
 $\begin{tz} 
        \node[Pants, top, bot] (A) at (0,0) {};
        \node[Copants, bot, anchor=leftleg] at (A.leftleg) {};
    \end{tz}
    \quad
    \xRightarrow{\epsilon}
    \quad
    \begin{tz} 
        \node[Cyl, top, bot, tall] (A) at (0,0) {};
    \end{tz}$ & $\id_X \otimes \varepsilon \colon X\otimes\coend\to X$\\ [1ex] 
$\begin{tz} 
        \node[Cyl, tall, top, bot] (A) at (0,0) {};
        \node[Cyl, tall, top, bot] (B) at (2*\cobwidth, 0) {};
    \end{tz}
    \quad
    \xRightarrow{\eta}
    \quad
    \begin{tz} 
        \node[Pants, bot] (A) at (0,0) {};
        \node[Copants, top, bot, anchor=belt] at (A.belt) {};
    \end{tz}$& Coaction $\delta_X$ \\ [1ex] 
    $\begin{tz} 
        \node[Cyl, top, bot, tall] (A) at (0,0) {};
    \end{tz}
    \quad
    \xRightarrow{\epsilon^\dagger}
    \quad
    \begin{tz} 
        \node[Pants, top, bot] (A) at (0,0) {};
        \node[Copants, bot, anchor=leftleg] at (A.leftleg) {};
    \end{tz}$ & $\id_X\otimes \mathscrsfs{D}^\inv \Lambda\colon X\to X\otimes\coend$ \\ [1ex] 
    $\begin{tz} 
        \node[Pants, bot] (A) at (0,0) {};
        \node[Copants, top, bot, anchor=belt] at (A.belt) {};
    \end{tz}
    \quad
    \xRightarrow{\eta^\dagger}
    \quad
    \begin{tz} 
        \node[Cyl, tall, top, bot] (A) at (0,0) {};
        \node[Cyl, tall, top, bot] (B) at (2*\cobwidth, 0) {};
    \end{tz}$ & 
    $(\text{id}_X\otimes (\mathscrsfs{D}\Lambda^{co}\circ m \circ (S \otimes\text{id}_\coend)))\circ(\delta_X \otimes \text{id}_\coend)$ \\ [1ex] 
    $\begin{tz}
        \node[Cup, top] (A) at (0,0) {};
        \node[Cap, bot] (B) at (0,-2*\cobheight) {};
    \end{tz}
    \quad
    \xRightarrow{\mu}
    \quad
    \begin{tz}
        \node[Cyl, top, bot, tall] (A) at (0,0) {};
    \end{tz}$ & $\mathscrsfs{D}(\text{ev}_{\varepsilon_\unit}\otimes\eta_\unit)$ \\ [1ex]
    $\begin{tz}
        \node[Cyl, top, bot, tall] (A) at (0,0) {};
    \end{tz}
    \quad
    \xRightarrow{\mu^\dagger}
    \quad\begin{tz}
        \node[Cup, top] (A) at (0,0) {};
        \node[Cap, bot] (B) at (0,-2*\cobheight) {};
    \end{tz}$ & $(\text{id}\otimes EV_{P})\circ(\text{coev}_{\text{Hom}_\mathcal{C}(P,\unit)})$  \\ [1ex]
    $\begin{tz}
        \node[Cap, bot] (A) at (0,0) {};
        \node[Cup] at (0,0) {};
    \end{tz}
    \quad
    \xRightarrow{\nu^\dagger}
    \quad
    \begin{tz}
        \draw[green] (0,0) rectangle (0.6, 0.6);  
    \end{tz}$ & Evaluation of vector spaces \\ [1ex]
 \hline
 \end{tabular}
  \caption{Assignment $\Znc$ to generators of $\Bordncsig$}
    \label{genass_table}
\end{table}

Note that, having developed the necessary algebraic prerequisites, we have given a more precise description of the assignments. In particular, we make use of the equivalence $\C\boxtimes\C \xrightarrow{\sim} \C^\coend$ (Lemma \ref{comod_equiv}).

    
Now, in more detail, we define the assignment $\mathcal{Z}:\Bordncsig \to \text{Lincat}$ as follows:

$$\mathcal{Z}(S^1)=\C$$

Call 

\begin{itemize}
    \item $T:=\mathcal{Z}(\tinypants)$
    \item $T^R:=\mathcal{Z}(\tinycopants)$
    \item $U:=\mathcal{Z}(\tinycup)$
    \item $U^L:=\mathcal{Z}(\tinycap)$
\end{itemize}


and define:

\begin{itemize}
    \item $T$ to be given by the monoidal product: 
    
      \begin{align*}
      \C\boxtimes\C&\xrightarrow{T} \C\\
      X\boxtimes Y&\mapsto X\otimes Y
      \end{align*}



    \item $T^R$ to be:
    
    \begin{align*}
\C&\xrightarrow{T^R}\C\boxtimes\C \\
X&\mapsto \int^{Z\in\C} X\otimes Z^\lor \boxtimes Z
\end{align*}



    We call this map $T^R$ as, when we later check the relations \eqref{adj_eta_epsilon1}, \eqref{adj_eta_epsilon2}, this will show that it's the right adjoint of $T$.
    
    \item $U\colon\text{Vec}_k\to \C$
    
    \begin{align*}
        \text{Vec}_k&\xrightarrow{U}\C \\
V&\mapsto V\otimes\unit
    \end{align*}



    \item $U^L$ is (as we will see) the left adjoint of $U$. We define it to be: 
    
    \begin{align*}
       \C&\xrightarrow{U^L} \text{Vec}_k \\
X&\mapsto \text{Hom}_\C(X,\unit)^* \\
    \end{align*}


\end{itemize}

    As mentioned in Section \ref{target}, it will be beneficial to work with $\C^\coend$ instead of $\C\boxtimes\C$. So, under the equivalence $\C\boxtimes\C\cong \C^\coend$ the functor:
\begin{itemize}
    \item $T$ corresponds to the forgetful functor: 
    
      \begin{align*}
      \C^\coend&\xrightarrow{T} \C\\
      (x,\delta_X)&\mapsto X
      \end{align*}



    \item $T^R$ corresponds to:
    
    \begin{align*}
\C&\xrightarrow{T^R}\C^\coend \\
X&\mapsto (X\otimes \coend,\Delta)
\end{align*}

\end{itemize}








Now we will define the images of 2-morphisms. 

The assignment for the generating 2-morphisms $\alpha, \lambda, \rho, \beta, \theta$ is straightforward. They are respectively given by the associator, left/right unitor, braiding ($c_{X,Y}$) and twist ($\vartheta$) morphisms of the modular tensor category $\C$.


The maps $\mathcal{Z}(\epsilon), \mathcal{Z}(\eta)$ are given by (and will prove to be the counit and unit, respectively, of the adjunction of $T \dashv T^R$):

\begin{itemize}
    \item $\mathcal{Z}(\epsilon)$ is a natural transformation: $T\circ T^R\Rightarrow \text{id}_\C$. For $X\in\C$, we define it to be:
    
    $$\mathcal{Z}(\epsilon)\colon X\otimes\coend\xrightarrow{\id_X \otimes  \varepsilon} X$$

\vspace{5mm}
    
    \item $\mathcal{Z}(\eta)$ is a natural transformation: $\text{id}_{\C^\coend}\Rightarrow T^R\circ T$. For $(X,\delta_X) \in \C^\coend$, we define it to be:
    
    $$\mathcal{Z}(\eta)\colon(X,\delta_X) \xrightarrow{\delta_X} (X\otimes \coend, \Delta),$$


\vspace{5mm}

\end{itemize}

where $\varepsilon$ is the counit of $\coend$ and $\delta_X\colon X\to X\otimes \coend$ is the coaction for the $\coend$-comodule $(X,\delta_X)$. 


To define the maps $\mathcal{Z}(\epsilon^\dagger), \mathcal{Z}(\eta^\dagger)$, we will make use of the integral $\Lambda\colon\unit\to\coend$ and cointegral $\Lambda^{co}\colon\coend\to \unit$ that we defined earlier:

\begin{itemize}
    \item $\mathcal{Z}(\epsilon^\dagger)\colon \text{id}_\C \Rightarrow T^R\circ T$, meaning, a natural transformation with components: $X\to X\otimes\coend$. We define this to be:
    
    \begin{equation*}
        X\xrightarrow{\id_X\otimes \mathscrsfs{D}^\inv \Lambda} X\otimes \coend.
    \end{equation*}

    As mentioned, the constant $\mathscrsfs{D}\in\ k^\times$ will be determined later by some of the axioms. 
    
    \item $\mathcal{Z}(\eta^\dagger)$: We want a natural transformation with components: $(X\otimes\coend,\Delta)\to (X,\delta_X)$
    
    \begin{gather*}
        (X\otimes\coend,\Delta)\xrightarrow{\delta_X \otimes \text{id}_\coend} (X\otimes\coend\otimes\coend,\Delta)
     \xrightarrow{\text{id}_X\otimes S \otimes\text{id}_\coend} (X\otimes\coend\otimes\coend,\Delta) \\
    \xrightarrow{\text{id}_X\otimes m} (X\otimes\coend,\Delta)\xrightarrow{\text{id}_X\otimes\mathscrsfs{D}\Lambda^{co}} (X,\delta_X)
    \end{gather*}

\end{itemize}

\begin{remark}
    The intuition behind the definition of $\mathcal{Z}(\epsilon^\dagger), \mathcal{Z}(\eta^\dagger)$ is related to the following : The pairing $\Lambda^{co}\circ m$ is non degenerate (with the copairing involving the integral $\Lambda$) (\cite{kerler_non-semisimple_2001}, Corol 4.2.13), and thus introduces an isomorphism $\coend\cong \coend^\lor$. Now, as $T^L(X)=(X\otimes \coend^\lor,(\id_X\otimes\Delta^\lor\otimes\id_\coend)\circ(\id_X\otimes\id_{\coend^\lor}\otimes \text{coev}_{\coend^\lor}))$ and $T^R(X)=(X\otimes\coend,\id_X\otimes\Delta)$, utilizing the isomorphism above, we obtain an isomorphism $T^L\cong T^R$. This is precisely why $\mathcal{Z}(\epsilon^\dagger), \mathcal{Z}(\eta^\dagger)$ will be the counit and unit of the adjunction $T^R \dashv T$, and will therefore satisfy the desired axioms.
\end{remark}


Now we define the natural transformations $\mathcal{Z}(\mu^\dagger)$, $\mathcal{Z}(\nu^\dagger)$ on projectives as follows:


\begin{itemize}
    \item $\mathcal{Z}(\mu^\dagger)\colon \text{id}\Rightarrow U\circ U^L$, meaning, for $P\in \cP$, we have a natural transformation with components:

    $$ P \xrightarrow{\mathcal{Z}(\mu^\dagger)_P} \Hom_\C(P,\unit)^*\otimes\unit $$
given by the following composite: 

\begin{align*}
    \begin{split}
        P\xrightarrow{\text{coev}_{\text{Hom}_\mathcal{C}(P,\unit)}\otimes \text{id}_P}\text{Hom}_\mathcal{C}(P,\unit)^*\otimes\text{Hom}_\mathcal{C}(P,\unit)\otimes P \xrightarrow{\text{id}\otimes EV_{P}}\text{Hom}_\mathcal{C}(P,\unit)^*\otimes\unit
    \end{split}
\end{align*}

 where in the first step we used co-evaluation for the vector space $\Hom_\C(P,\unit)$.

    \item $\mathcal{Z}(\nu^\dagger)\colon U^L\circ U \Rightarrow \text{id}$, in other words, a natural transformation

    $$\Hom_\C(\unit,\unit)^*\xrightarrow{\mathcal{Z}(\nu^\dagger)}k$$
     is given by evaluation on $\id_\unit$.
\end{itemize}



As for $\mathcal{Z}(\mu)$, it is not an obvious map. 


\begin{itemize}
    \item $\mathcal{Z}(\mu)\colon U\circ U^L \Rightarrow \text{id}$, i.e a natural transformation with components:

    \begin{equation}\label{mudef}
        \Hom(P,\unit)^*\otimes\unit \xrightarrow{\mu_P} P
    \end{equation}
\end{itemize}

To define it, it suffices to define $\mu_{P_\unit}\colon\Hom(P_\unit,\unit)^*\otimes\unit \to P_\unit$.

This morphism is given by:

$$\Hom_\C(P_\unit,\unit)^*\otimes\unit\xrightarrow{\mathscrsfs{D}(\text{ev}_{\varepsilon_\unit}\otimes\eta_\unit)} P_\unit$$


where by $\text{ev}_{\varepsilon_\unit}$ we mean the map $\Hom_\C(P_\unit,\unit)^*\to k$ obtained by evaluating an element $f\in  \Hom_\C(P_\unit,\unit)^*$ on the canonical element $\varepsilon_\unit\in \Hom_\C(P_\unit,\unit)$.

However, it will be useful to think about the map $\text{ev}_{\varepsilon_\unit}$ in the following way:

 We have the canonical element $\varepsilon_\unit\in \Hom_\C(P_\unit,\unit)$ which gives rise to a linear map $\Bar{\varepsilon}_1: k\to \Hom_\C(P_\unit,\unit)$ and therefore we can define the map $\Hom_\C(P_\unit,\unit)^* \xrightarrow{ev_{\varepsilon_\unit}} k $ as: 
 $$k\otimes \Hom_\C(P_\unit,\unit)^* \xrightarrow{\bar{\varepsilon}_\unit\otimes\id} \Hom_\C(P_\unit,\unit)\otimes\Hom_\C(P_\unit,\unit)^* \xrightarrow{ev_{vect}} k$$


The usefulness of this description of $\text{ev}_{\varepsilon_\unit}$ is showcased in the proof of the following lemma, which will be useful when checking relations \eqref{piv_on_sphere} and \eqref{MOD}.

\begin{remark}\label{cap_description}
    The functor $U^L\colon\C\to \text{Vec}_k$ has two equivalent descriptions. Since the functor $\Hom_\C(-,\unit)^*\colon\C\to \text{Vec}_k$ is right exact, it suffices to define it only on projectives: $U^L(P)\cong \Hom_\C(P,\unit)^*$. This gives us another way to describe it, using Proposition \ref{coYoneda}. 
    
    This way, $U^L(-)\cong \displaystyle\int^{P\in \cP} \hspace{-20pt}\Hom_\C(P,-)\otimes\Hom_\C(P,\unit)^*$, and $\Znc(\nu^\dagger)$ is correspondingly given by evaluation of vector spaces $\displaystyle\int^{P\in \cP}\hspace{-20pt}\Hom_\C(P,\unit)\otimes\Hom_\C(P,\unit)^*\xrightarrow{\mathcal{Z}(\nu^\dagger)}k$. In fact, this coincides with the action of $\nu^\dagger$ when the target is Bimod, something we are concerned with in chapters \ref{3-mfld-inv} and \ref{mod-trace}.

\end{remark}






\subsection{Checking the relations}

One of the main results of this thesis is proving the following:

\begin{theorem}\label{main_th}
    The assignment $\mathcal{Z}$ in Section \ref{genass} defines a symmetric monoidal functor between bicategories $\Bordncsig\to \text{Lincat}$.
\end{theorem}


In order to prove Theorem \ref{main_th}, it is sufficient to check that the relations \eqref{eq:pentagon}-\eqref{Anom} are preserved after applying the assignment $\mathcal{Z}$. We will be referring to this as `checking the relations' without mentioning the assignment $\mathcal{Z}$. This amounts to checking equalities between natural transformations. We will therefore work in components, denoting by $X,Y,Z$ objects in $\C$, $P,Q\in\cP$. 


Before we proceed to the proof of Theorem \ref{main_th}, we prove some useful lemmas:

\begin{lemma}\label{mu_mu_dag_comp}


    
$(\mathcal{Z}(\mu)\circ\mathcal{Z}(\mu^\dagger))(P_j)
=
\begin{cases} 
      \mathscrsfs{D}(\eta_\unit\circ \varepsilon_\unit) & P_j\cong P_\unit \\
      0 & P_j\ncong P_\unit 
   \end{cases}
$. 



\end{lemma}

\begin{proof}
    First, if $P_j\ncong P_\unit$, then this composite is the zero morphism (Section \ref{finiteproj}).

    Now, for the case where $P_j\cong P_\unit$, the composite $\mu_{P_\unit}\circ\mu^\dagger_{P_\unit}$ looks as follows in string diagram notation:

    \begin{align*}
        \begin{tz}[xscale=0.5,yscale=0.8]
\node [2halfbox] at (2.5,1) {$EV_{P_\unit}$};
\node [halfbox] at (0,2) {$\Bar{\varepsilon}_\unit$};
\node [halfbox] at (2.5,2) {$\eta_\unit$};
\draw [blue strand] (1,0.5) to (1,2.5);
\draw [blue strand] (0,2) to (0,2.5);
\draw [blue strand] (2,0.5) to (2,1);
\draw (2.5,2) to (2.5,3) node [above] {$P_\unit$};
\draw (3,0) node [below] {$P_\unit$} to (3,1);
\draw [blue strand] (1,0.5) to [out=down, in=down, looseness=1] (2,0.5);
\draw [blue strand] (0,2.5) to [out=up, in=up, looseness=1] (1,2.5);
\end{tz}
\gap=\gap
        \begin{tz}[xscale=0.5,yscale=0.8]
\node [halfbox] at (2.5,1) {$\varepsilon_\unit$};
\node [halfbox] at (2.5,2) {$\eta_\unit$};
\draw (2.5,2) to (2.5,3) node [above] {$P_\unit$};
\draw (2.5,0) node [below] {$P_\unit$} to (2.5,1);
\end{tz}
\end{align*}

Where the blue string is meant to correspond to the vector space $\Hom_\C(P_\unit,\unit)$ and the equality holds because of the duality relations as well as \cite[Lemma 5.1.3]{kerler_non-semisimple_2001}.

    
\end{proof}



\begin{lemma} \label{coaction_weird}
    The coaction $\delta_{X\otimes \coend\otimes Y}$ appearing in $\mathcal{Z}(\phi_r)$ and $\mathcal{Z}(\phi_r^\inv)$ is given by:
    $$X\otimes \coend\otimes Y \xrightarrow{\id_X\otimes\Delta\otimes\varrho_Y}X\otimes \coend\otimes\coend\otimes Y\otimes \coend \xrightarrow{\id_X\otimes\id_\coend\otimes c_{\coend,Y}\otimes\id_\coend} X\otimes \coend\otimes Y\otimes \coend\otimes \coend $$ 
    $$\xrightarrow{\id_{X\otimes\coend\otimes Y}\otimes m} X\otimes \coend\otimes Y\otimes \coend.$$

As a diagram this is represented as:
    \begin{align*}
 \begin{tz}[xscale=0.4,yscale=0.7]
\draw [black strand] (4.25,1) to [out=up, in=down] (1,2.5);
\draw [purple strand] (3.5,3) to (3.5,4) node [above] {$\mathcal{F}$};
\draw [purple strand] (2.5,1) to (2.5,3);
\draw [purple strand] (4.75,1) to (4.75,3);
\draw [purple strand] (1,0) node [below] {$\mathcal{F}$} to (1,1);
\draw [purple strand] (0,1) to (0,4);
\node [box] at (3.5,3) {$m$};
\node [box] at (1,1) {$\Delta$};
\node [1halfbox] at (4.5,1) {$\varrho_Y$};
\draw (1,2.5) to (1,4) node [above] {$Y$};
\draw (-1.5,0) node [below] {$X$} to (-1.5,4);
\draw [black strand] (4.5,0) node [below] {$Y$} to [out=up, in=down] (4.5,1);
\end{tz}
    \end{align*}
\end{lemma}



\begin{proof}
    The coaction $\delta_{X\otimes \coend\otimes Y}$ is given by $\varrho$, applied on the appropriate object. To understand what that object is, we need to look at the level at which $\eta$ (resp. $\eta^\dagger$) is applied on \eqref{defn_of_phiright} (resp. \eqref{explicit_phiright_inverse}). $\eta$ is applied on the two cylinders (resp. $\eta^\dagger$ on $\tinycopants\circ\tinypants$). In both cases, the object at the source of that 1-morphism is $\displaystyle\int^{Z\in\C}\!\!\!\!\!\!\!\!\!\!\!\!X\otimes Z^\lor\boxtimes Z\otimes Y$. $\eta$ is given by the coaction of this object and the first part of the $\eta^\dagger$ composite is the coaction of this object. So we are looking to compute the coaction induced on $X\otimes \coend\otimes Y$ from the equivalence $\C\boxtimes\C\cong \C^\coend$, in other words from $\varrho_{Z\otimes Y}$.

    Recall that $(Z\otimes Y)^\lor \cong  Y^\lor\otimes Z^\lor$. We therefore have the following:

\begin{align*}
\begin{tz}[xscale=0.4,yscale=0.7]
\node [box] at (3.5,1) {$i_{Z\otimes Y}$};
\draw [purple strand] (3.5,1) to (3.5,2) node [above] {$\mathcal{F}$};
\draw (5,0) node [below] {$Y$} to (5,1);
\draw (4,0) node [below] {$Z$} to (4,1);
\draw (3,0) to (3,1);
\draw (2,0) to (2,1);
\draw (1,0) to (1,2) node [above] {$Y$};
\draw (0,0) to (0,2) node [above] {$Z$};
\draw [black strand] (0,0) to [out=down, in=down, looseness=1] (3,0);
\draw [black strand] (1,0) to [out=down, in=down, looseness=1] (2,0);
\end{tz}
\gap=\gap
\begin{tz}[xscale=0.4,yscale=0.7]
\draw [purple strand] (3.5,4) to (3.5,5) node [above] {$\mathcal{F}$};
\node [box] at (3.5,4) {$i_{Z\otimes Y}$};
\draw (1,0) to (1,5) node [above] {$Y$};
\draw (0,0) to (0,5) node [above] {$Z$};
\draw [black strand] (4,2) to [out=up, in=down] (2,3.5);
\draw [black strand] (2,2) to [out=up, in=down] (3,3.5);
\draw [black strand] (3,2) to [out=up, in=down] (4,3.5);
\draw [black strand] (5,2) to [out=up, in=down] (5,4);
\draw [black strand] (4,3.5) to [out=up, in=down] (4,4);
\draw [black strand] (2,3.5) to [out=up, in=down] (2,4);
\draw [black strand] (3,3.5) to [out=up, in=down] (3,4);
\draw [black strand] (2,0) to [out=up, in=down] (4,1.5);
\draw [black strand] (3,0) to [out=up, in=down] (2,1.5);
\draw [black strand] (4,0) node [below] {$Z$} to [out=up, in=down] (3,1.5);
\draw [black strand] (5,0) node [below] {$Y$} to [out=up, in=down] (5,2);
\draw [black strand] (4,1.5) to [out=up, in=down] (4,2);
\draw [black strand] (2,1.5) to [out=up, in=down] (2,2);
\draw [black strand] (3,1.5) to [out=up, in=down] (3,2);
\draw [black strand] (0,0) to [out=down, in=down, looseness=1] (3,0);
\draw [black strand] (1,0) to [out=down, in=down, looseness=1] (2,0);
\end{tz}
\gap=\gap
\begin{tz}[xscale=0.4,yscale=0.7]
\draw [purple strand] (3.5,3) to (3.5,4) node [above] {$\mathcal{F}$};
\draw [purple strand] (2.5,2) to (2.5,3);
\draw [purple strand] (4.5,2) to (4.5,3);
\node [box] at (3.5,3) {$m$};
\node [1halfbox] at (2.5,2) {$i_{Z}$};
\node [1halfbox] at (4.5,2) {$i_{Y}$};
\draw (1,0) to (1,4) node [above] {$Y$};
\draw (0,0) to (0,4) node [above] {$Z$};
\draw [black strand] (2,0) to [out=up, in=down] (4,1.5);
\draw [black strand] (3,0) to [out=up, in=down] (2,1.5);
\draw [black strand] (4,0) node [below] {$Z$} to [out=up, in=down] (3,1.5);
\draw [black strand] (5,0) node [below] {$Y$} to [out=up, in=down] (5,2);
\draw [black strand] (4,1.5) to [out=up, in=down] (4,2);
\draw [black strand] (2,1.5) to [out=up, in=down] (2,2);
\draw [black strand] (3,1.5) to [out=up, in=down] (3,2);
\draw [black strand] (0,0) to [out=down, in=down, looseness=1] (3,0);
\draw [black strand] (1,0) to [out=down, in=down, looseness=1] (2,0);
\end{tz}
\gap=\gap
\end{align*}

\begin{align*}
    \gap=\gap
\begin{tz}[xscale=0.4,yscale=0.7]
\draw [purple strand] (3.5,3) to (3.5,4) node [above] {$\mathcal{F}$};
\draw [purple strand] (2.5,2) to (2.5,3);
\draw [purple strand] (4.5,2) to (4.5,3);
\node [box] at (3.5,3) {$m$};
\node [1halfbox] at (2.5,2) {$i_{Z}$};
\node [1halfbox] at (4.5,2) {$i_{Y}$};
\draw [black strand] (3.5,0) to [out=up, in=down] (1,2);
\draw (1,2) to (1,4) node [above] {$Y$};
\draw (0,0) to (0,4) node [above] {$Z$};
\draw [black strand] (4,0) to [out=up, in=down] (4,2);
\draw [black strand] (2,0) to [out=up, in=down] (2,1.5);
\draw [black strand] (3,0) node [below] {$Z$} to [out=up, in=down] (3,1.5);
\draw [black strand] (5,0) node [below] {$Y$} to [out=up, in=down] (5,2);
\draw [black strand] (4,1.5) to [out=up, in=down] (4,2);
\draw [black strand] (2,1.5) to [out=up, in=down] (2,2);
\draw [black strand] (3,1.5) to [out=up, in=down] (3,2);
\draw [black strand] (0,0) to [out=down, in=down, looseness=1] (2,0);
\draw [black strand] (3.5,0) to [out=down, in=down, looseness=1] (4,0);
\end{tz}
\gap=\gap
\begin{tz}[xscale=0.4,yscale=0.7]
\draw [black strand] (3.5,1.5) to [out=up, in=down] (1,2.5);
\draw [purple strand] (3.5,3) to (3.5,4) node [above] {$\mathcal{F}$};
\draw [purple strand] (2.5,1) to (2.5,3);
\draw [purple strand] (4.5,2) to (4.5,3);
\node [box] at (3.5,3) {$m$};
\node [1halfbox] at (2.5,1) {$i_{Z}$};
\node [1halfbox] at (4.5,2) {$i_{Y}$};
\draw (3.5,0) to (3.5,1.5);
\draw (1,2.5) to (1,4) node [above] {$Y$};
\draw (0,0) to (0,4) node [above] {$Z$};
\draw [black strand] (4,0) to [out=up, in=down] (4,2);
\draw [black strand] (2,0) to [out=up, in=down] (2,1);
\draw [black strand] (3,0) node [below] {$Z$} to [out=up, in=down] (3,1);
\draw [black strand] (5,0) node [below] {$Y$} to [out=up, in=down] (5,2);
\draw [black strand] (0,0) to [out=down, in=down, looseness=1] (2,0);
\draw [black strand] (3.5,0) to [out=down, in=down, looseness=1] (4,0);
\end{tz}
\gap=\gap
\begin{tz}[xscale=0.4,yscale=0.7]
\draw [black strand] (4.25,1) to [out=up, in=down] (1,2.5);
\draw [purple strand] (3.5,3) to (3.5,4) node [above] {$\mathcal{F}$};
\draw [purple strand] (2.5,1) to (2.5,3);
\draw [purple strand] (4.75,1) to (4.75,3);
\node [box] at (3.5,3) {$m$};
\node [1halfbox] at (2.5,1) {$i_{Z}$};
\node [1halfbox] at (4.5,1) {$\varrho_Y$};
\draw (1,2.5) to (1,4) node [above] {$Y$};
\draw (0,0) to (0,4) node [above] {$Z$};
\draw [black strand] (2,0) to [out=up, in=down] (2,1);
\draw [black strand] (3,0) node [below] {$Z$} to [out=up, in=down] (3,1);
\draw [black strand] (4.5,0) node [below] {$Y$} to [out=up, in=down] (4.5,1);
\draw [black strand] (0,0) to [out=down, in=down, looseness=1] (2,0);
\end{tz}
\end{align*}

Therefore, the action $\delta_{X\otimes \coend\otimes Y}:X\otimes \coend\otimes Y\to X\otimes \coend\otimes Y\otimes \coend$ is given by the unique morphism that is determined by the universal property of $\coend$ as follows:

\begin{align*}
    \begin{tz}[xscale=0.4,yscale=0.7]
\draw [black strand] (4.25,1) to [out=up, in=down] (1,2.5);
\draw [purple strand] (3.5,3) to (3.5,4) node [above] {$\mathcal{F}$};
\draw [purple strand] (2.5,1) to (2.5,3);
\draw [purple strand] (4.75,1) to (4.75,3);
\draw [purple strand] (-0.5,1) to (-0.5,4);
\node [box] at (3.5,3) {$m$};
\node [1halfbox] at (2.5,1) {$i_{Z}$};
\node [1halfbox] at (-0.5,1) {$i_{Z}$};
\node [1halfbox] at (4.5,1) {$\varrho_Y$};
\draw (1,2.5) to (1,4) node [above] {$Y$};
\draw (0,0) to (0,1);
\draw (-1,0) node [below] {$Z^\lor$} to (-1,1);
\draw (-2,0) node [below] {$X$} to (-2,4);
\draw [black strand] (2,0) to [out=up, in=down] (2,1);
\draw [black strand] (3,0) node [below] {$Z$} to [out=up, in=down] (3,1);
\draw [black strand] (4.5,0) node [below] {$Y$} to [out=up, in=down] (4.5,1);
\draw [black strand] (0,0) to [out=down, in=down, looseness=1] (2,0);
\end{tz}
\gap=\gap
    \begin{tz}[xscale=0.4,yscale=0.7]
\draw [black strand] (4.25,1) to [out=up, in=down] (1,2.5);
\draw [purple strand] (3.5,3) to (3.5,4) node [above] {$\mathcal{F}$};
\draw [purple strand] (2.5,1) to (2.5,3);
\draw [purple strand] (4.75,1) to (4.75,3);
\draw [purple strand] (1,0) node [below] {$\mathcal{F}$} to (1,1);
\draw [purple strand] (0,1) to (0,4);
\node [box] at (3.5,3) {$m$};
\node [box] at (1,1) {$\Delta$};
\node [1halfbox] at (4.5,1) {$\varrho_Y$};
\draw (1,2.5) to (1,4) node [above] {$Y$};
\draw (-1.5,0) node [below] {$X$} to (-1.5,4);
\draw [black strand] (4.5,0) node [below] {$Y$} to [out=up, in=down] (4.5,1);
\end{tz}
\end{align*}

\end{proof}

\begin{corol}\label{coaction_coend}
    In the case where $Y=\unit$, we get that the coaction of the free $\coend$ comodule $X\otimes \coend$ is given by the comultiplication $\Delta$.
\end{corol}

Now we are ready to prove Theorem \ref{main_th}:

\begin{proof}
    
\vspace{1cm}

\begin{itemize}
     
 \item Relations \eqref{eq:pentagon}, \eqref{eq:triangle}, \eqref{eq:hexagon}, \eqref{balanced1}, \eqref{balanced2} hold trivially, as they correspond to the pentagon, triangle, hexagon and balancing relations of the modular category $\C$.

\item The first relations we need to check, are the images of \eqref{defn_of_phileft}-\eqref{explicit_phileft_inverse} and \eqref{defn_of_phiright}-\eqref{explicit_phiright_inverse} under $\mathcal{Z}$, namely, that $\mathcal{Z}(\phi_l)^\inv=\mathcal{Z}(\phi_l^\inv)$ and $\mathcal{Z}(\phi_r)^\inv=\mathcal{Z}(\phi_r^\inv)$.

We begin by seeing that $\mathcal{Z}(\phi_l)$ and $\mathcal{Z}(\phi_l^\inv)$ are both the identity natural transformation, so the desired relation is satisfied trivially:

\begin{align} \label{phi_l}
    \begin{tz}[xscale=0.5,yscale=0.7]
\node [2halfbox] at (2.5,1) {$\Delta$};
\node [halfbox] at (2,2) {$\varepsilon$};
\draw [purple strand] (3,1) to (3,2) node [above] {$\coend$};
\draw [purple strand] (2,1) to (2,2);
\draw [purple strand] (2.5,0) node [below] {$\coend$} to (2.5,1);
\draw (0,0) node [below] {$X$} to (0,2);
\draw (1,0) node [below] {$Y$} to (1,2);
\end{tz}
\gap=\gap
\begin{tz}[xscale=0.5,yscale=0.7]
\draw [purple strand] (2,0) node [below] {$\coend$} to (2,2);
\draw (0,0) node [below] {$X$} to (0,2);
\draw (1,0) node [below] {$Y$} to (1,2);
\end{tz}
&&
    \begin{tz}[xscale=0.5,yscale=0.7]
\node [halfbox] at (2.5,2) {$S$};
\node [2halfbox] at (2,1) {$\Delta$};
\node [halfbox] at (2,0) {$\Lambda$};
\node [2halfbox] at (3,3) {$m$};
\node [halfbox] at (3,4) {$\Lambda^{co}$};
\draw [purple strand] (3.5,0) node [below] {$\coend$} to (3.5,3);
\draw [purple strand] (2,0) to (2,1);
\draw [purple strand] (2.5,1) to (2.5,3);
\draw [purple strand] (1.5,1) to (1.5,4);
\draw [purple strand] (3,3) to (3,4);
\draw (0,0) node [below] {$X$} to (0,4);
\draw (1,0) node [below] {$Y$} to (1,4);
\end{tz}
\gap=\gap
\begin{tz}[xscale=0.5,yscale=0.7]
\draw [purple strand] (2,0) node [below] {$\coend$} to (2,2);
\draw (0,0) node [below] {$X$} to (0,2);
\draw (1,0) node [below] {$Y$} to (1,2);
\end{tz}
\end{align}

The first equality holds because of the counit-comultiplication relation while the second holds because of the fact that $\Delta\circ\Lambda$ is a non-degenerate pairing and more precisely, because of \cite[Lemma 4.2.12]{kerler_non-semisimple_2001}. 


\item Now, for \eqref{defn_of_phiright}-\eqref{explicit_phiright_inverse}, things require a bit more computation:

First, we would like to simplify both $\mathcal{Z}(\phi_r)$ and $\mathcal{Z}(\phi_r^\inv)$. 
The former is given by the composite:

$$X\otimes \coend \otimes Y \xrightarrow{\delta_{X\otimes \coend\otimes Y}} X\otimes\coend\otimes Y\otimes \coend \xrightarrow{\id_X\otimes\varepsilon\otimes\id_{Y\otimes\coend}} X\otimes Y\otimes \coend,$$

And the latter involves the composite:

$$X\otimes Y\otimes \coend \xrightarrow{\id_X\otimes\Lambda\otimes\id_{Y\otimes\coend}} 
X\otimes\coend\otimes Y\otimes \coend \xrightarrow{\delta_{X\otimes \coend\otimes Y}\otimes\id_\coend}X\otimes\coend\otimes Y\otimes\coend\otimes \coend.$$

Now using Lemma \ref{coaction_weird}, $\mathcal{Z}(\phi_r)$ becomes:

\begin{align*}
        \begin{tz}[xscale=0.4,yscale=0.7]
\draw [black strand] (4.25,1) to [out=up, in=down] (1,2.5);
\draw [purple strand] (3.5,3) to (3.5,4) node [above] {$\mathcal{F}$};
\draw [purple strand] (2.5,1) to (2.5,3);
\draw [purple strand] (4.75,1) to (4.75,3);
\draw [purple strand] (1,0) node [below] {$\mathcal{F}$} to (1,1);
\draw [purple strand] (0,1) to (0,3);
\node [box] at (3.5,3) {$m$};
\node [box] at (1,1) {$\Delta$};
\node [1halfbox] at (4.5,1) {$\varrho_Y$};
\node [halfbox] at (0,3) {$\varepsilon$};
\draw (1,2.5) to (1,4) node [above] {$Y$};
\draw (-1.5,0) node [below] {$X$} to (-1.5,4);
\draw [black strand] (4.5,0) node [below] {$Y$} to [out=up, in=down] (4.5,1);
\end{tz}
\gap=\gap
\begin{tz}[xscale=0.4,yscale=0.7]
\draw [black strand] (4.25,1) to [out=up, in=down] (1,2.5);
\draw [purple strand] (3.5,3) to (3.5,4) node [above] {$\mathcal{F}$};
\draw [purple strand] (2.5,0) node [below] {$\mathcal{F}$} to (2.5,3);
\draw [purple strand] (4.75,1) to (4.75,3);
\node [box] at (3.5,3) {$m$};
\node [1halfbox] at (4.5,1) {$\varrho_Y$};
\draw (1,2.5) to (1,4) node [above] {$Y$};
\draw (0,0) node [below] {$X$} to (0,4);
\draw [black strand] (4.5,0) node [below] {$Y$} to [out=up, in=down] (4.5,1);
\end{tz}
\end{align*}

And $\mathcal{Z}(\phi_r^\inv)$:

\begin{align*}
    \begin{tz}[xscale=0.5,yscale=0.7]
\node [2halfbox] at (3.5,1) {$\varrho_Y$};
\node [halfbox] at (1.5,0) {$\Lambda$};
\node [halfbox] at (4.25,6) {$\Lambda^{co}$};
\node [2halfbox] at (1.5,1) {$\Delta$};
\node [2halfbox] at (3.5,3) {$m$};
\node [halfbox] at (3.5,4) {$S$};
\node [box] at (4.25,5) {$m$};
\draw [purple strand] (5,0) node [below] {$\coend$} to (5,5);
\draw [purple strand] (4,1) to (4,3);
\draw [purple strand] (1.5,0) to (1.5,1);
\draw [purple strand] (1,1) to (1,6);
\draw [purple strand] (3.5,3) to (3.5,5);
\draw [purple strand] (4.25,5) to (4.25,6);
\draw [purple strand] (2,1) to (2,1.5);
\draw [purple strand] (3,2.5) to (3,3);
\draw (3,1) to (3,1.5);
\draw (2,2.5) to (2,6);
\draw [black strand] (3,1.5) to [out=up, in=down, looseness=0.9] (2,2.5);
\draw [purple strand] (2,1.5) to [out=up, in=down, looseness=0.9] (3,2.5);
\draw (0,0) node [below] {$X$} to (0,6);
\draw (3.5,0) node [below] {$Y$} to (3.5,1);
\end{tz}
\gap=\gap
\begin{tz}[xscale=0.5,yscale=0.7]
\node [box] at (4,0.5) {$\varrho_Y$};
\node [halfbox] at (4,1.5) {$\Lambda$};
\node [halfbox] at (5,7) {$\Lambda^{co}$};
\node [2halfbox] at (3.5,3) {$\Delta$};
\node [2halfbox] at (4.5,4) {$m$};
\node [halfbox] at (4.5,5) {$S$};
\node [box] at (5,6) {$m$};
\draw [purple strand] (6,-0.5) node [below] {$\coend$} to (6,6);
\draw [purple strand] (5,0.5) to (5,4);
\draw [purple strand] (4,1.5) to (3.5,3);
\draw [purple strand] (1,5) to (1,7);
\draw [purple strand] (4,3) to (4,4);
\draw [purple strand] (4.5,4) to (4.5,6);
\draw [purple strand] (5,6) to (5,7);
\draw (3,0.5) to (3,1.5);
\draw (2,2.5) to (2,7);
\draw [black strand] (3,1.5) to [out=up, in=down, looseness=0.9] (2,2.5);
\draw [purple strand] (3,3) to [out=up, in=down, looseness=0.9] (1,5);
\draw (0,-0.5) node [below] {$X$} to (0,7);
\draw (4,-0.5) node [below] {$Y$} to (4,0.5);
\end{tz}
\gap=\gap
 \begin{tz}[xscale=0.5,yscale=0.7]
\node [box] at (4,0.5) {$\varrho_Y$};
\node [halfbox] at (4,1.5) {$\Lambda$};
\node [halfbox] at (5,7) {$\Lambda^{co}$};
\node [2halfbox] at (3.5,3) {$\Delta$};
\node [2halfbox] at (4.5,4) {$m$};
\node [halfbox] at (6,5) {$S$};
\node [box] at (5,6) {$m$};
\draw [purple strand] (6,-0.5) node [below] {$\coend$} to (6,6);
\draw [purple strand] (5,0.5) to (5,4);
\draw [purple strand] (4,1.5) to (3.5,3);
\draw [purple strand] (1,5) to (1,7);
\draw [purple strand] (4,3) to (4,4);
\draw [purple strand] (4.5,4) to (4.5,6);
\draw [purple strand] (5,6) to (5,7);
\draw (3,0.5) to (3,1.5);
\draw (2,2.5) to (2,7);
\draw [black strand] (3,1.5) to [out=up, in=down, looseness=0.9] (2,2.5);
\draw [purple strand] (3,3) to [out=up, in=down, looseness=0.9] (1,5);
\draw (0,-0.5) node [below] {$X$} to (0,7);
\draw (4,-0.5) node [below] {$Y$} to (4,0.5);
\end{tz}
\gap=\gap
\end{align*}

\begin{align*}
    \gap=\gap
 \begin{tz}[xscale=0.5,yscale=0.7]
\node [box] at (4,0.5) {$\varrho_Y$};
\node [halfbox] at (4,1.5) {$\Lambda$};
\node [halfbox] at (5,7) {$\Lambda^{co}$};
\node [2halfbox] at (3.5,3) {$\Delta$};
\node [2halfbox] at (5.5,5) {$m$};
\node [halfbox] at (6,4) {$S$};
\node [box] at (5,6) {$m$};
\draw [purple strand] (6,-0.5) node [below] {$\coend$} to (6,5);
\draw [purple strand] (5.5,5) to (5.5,6);
\draw [purple strand] (5,0.5) to (5,5);
\draw [purple strand] (4,1.5) to (3.5,3);
\draw [purple strand] (1,5) to (1,7);
\draw [purple strand] (4,3) to (4,6);
\draw [purple strand] (5,6) to (5,7);
\draw (3,0.5) to (3,1.5);
\draw (2,2.5) to (2,7);
\draw [black strand] (3,1.5) to [out=up, in=down, looseness=0.9] (2,2.5);
\draw [purple strand] (3,3) to [out=up, in=down, looseness=0.9] (1,5);
\draw (0,-0.5) node [below] {$X$} to (0,7);
\draw (4,-0.5) node [below] {$Y$} to (4,0.5);
\end{tz}
\gap=\gap
 \begin{tz}[xscale=0.5,yscale=0.7]
\node [2halfbox] at (2.5,1) {$\varrho_Y$};
\node [halfbox] at (4,2) {$S$};
\node [2halfbox] at (3.5,3) {$m$};
\node [halfbox] at (3.5,4) {$S^\inv$};
\draw [purple strand] (4,0) node [below] {$\coend$} to (4,3);
\draw [purple strand] (3,1) to (3,3);
\draw (2,1) to (2,6);
\draw [purple strand] (3.5,3) to (3.5,4);
\draw [purple strand] (3.5,4) to [out=up, in=down, looseness=0.9] (1,6);
\draw (0,0) node [below] {$X$} to (0,6);
\draw (2.5,0) node [below] {$Y$} to (2.5,1);
\end{tz}
\end{align*}

Where in the second equality we used Lemma \ref{integral_antipode_relations} (iii) and in the third equality, we used associativity of the multiplication $m$. 
The last equality is due to \cite[Relation (c) of figure 4.3]{kerler_non-semisimple_2001}. 

Now using the above , we check the composite $\mathcal{Z}(\phi_r^\inv)\circ \mathcal{Z}(\phi_r)$:

\begin{align*}
 \begin{tz}[xscale=0.5,yscale=0.7]
\node [2halfbox] at (2.5,5) {$\varrho_Y$};
\node [halfbox] at (4,6) {$S$};
\node [2halfbox] at (3.5,7) {$m$};
\node [halfbox] at (3.5,8) {$S^\inv$};
\draw [purple strand] (4,5) to (4,7);
\draw [purple strand] (3,5) to (3,7);
\draw (2,5) to (2,10);
\draw [purple strand] (3.5,7) to (3.5,8);
\draw [purple strand] (3.5,8) to [out=up, in=down, looseness=0.9] (1,10);
\draw (0,4) to (0,10);
\draw (2.5,4) to (2.5,5);
\draw [black strand] (5.75,1) to [out=up, in=down] (2.5,2.5);
\draw [purple strand] (5,3) to (5,4);
\draw [purple strand] (5,4) to [out=up, in=down, looseness=0.9] (4,5);
\draw [purple strand] (4,0) node [below] {$\mathcal{F}$} to (4,3);
\draw [purple strand] (6.25,1) to (6.25,3);
\node [box] at (5,3) {$m$};
\node [1halfbox] at (6,1) {$\varrho_Y$};
\draw (2.5,2.5) to (2.5,4);
\draw (0,0) node [below] {$X$} to (0,4);
\draw [black strand] (6,0) node [below] {$Y$} to [out=up, in=down] (6,1);
\end{tz}
\gap=\gap
 \begin{tz}[xscale=0.5,yscale=0.7]
\node [2halfbox] at (2.5,5) {$\varrho_Y$};
\node [halfbox] at (4,6) {$S$};
\node [2halfbox] at (3.5,7) {$m$};
\node [halfbox] at (3.5,8) {$S^\inv$};
\draw [purple strand] (4,4) to (4,7);
\draw [purple strand] (3,5) to (3,7);
\draw (2,5) to (2,10);
\draw [purple strand] (3.5,7) to (3.5,8);
\draw [purple strand] (3.5,8) to [out=up, in=down, looseness=0.9] (1,10);
\draw (2.5,4) to (2.5,5);
\draw [purple strand] (3,2) to [out=up, in=down, looseness=0.9] (4.5,4);
\draw [black strand] (6,0) node [below] {$Y$} to [out=up, in=down] (3,2);
\draw [purple strand] (4.5,0) node [below] {$\mathcal{F}$} to (4.5,1);
\draw [purple strand] (4.5,1) to [out=up, in=down, looseness=0.9] (3.5,4);
\node [2halfbox] at (4,4) {$m$};
\node [1halfbox] at (3,2) {$\varrho_Y$};
\draw (2.5,2) to (2.5,4);
\draw (0,0) node [below] {$X$} to (0,10);
\end{tz}
\gap=\gap
 \begin{tz}[xscale=0.5,yscale=0.7]
\node [halfbox] at (4,6) {$S$};
\node [2halfbox] at (3.5,7) {$m$};
\node [halfbox] at (3.5,8) {$S^\inv$};
\draw [purple strand] (4,5) to (4,7);
\draw [purple strand] (3,2) to (3,7);
\draw (2,2) to (2,10);
\draw [purple strand] (3.5,7) to (3.5,8);
\draw [purple strand] (3.5,8) to [out=up, in=down, looseness=0.9] (1,10);
\draw [purple strand] (3.5,3) to [out=up, in=down, looseness=0.9] (4.5,5);
\draw [black strand] (6,0) node [below] {$Y$} to [out=up, in=down] (2.5,2);
\draw [purple strand] (4.5,0) node [below] {$\mathcal{F}$} to (4.5,3);
\draw [purple strand] (4.5,3) to [out=up, in=down, looseness=0.9] (3.5,5);
\node [2halfbox] at (4,5) {$m$};
\node [1halfbox] at (3.25,3) {$\Delta$};
\node [2halfbox] at (2.5,2) {$\varrho_Y$};
\draw (0,0) node [below] {$X$} to (0,10);
\end{tz}
\gap=\gap
\end{align*}

\begin{align*}
     \begin{tz}[xscale=0.5,yscale=0.7]
\node [2halfbox] at (3.5,7) {$m$};
\node [halfbox] at (3.5,8) {$S^\inv$};
\draw [purple strand] (4,5) to (4,7);
\draw [purple strand] (3,3) to [out=up, in=down, looseness=0.9] (2.5,5);
\draw [purple strand] (2.5,5) to [out=up, in=down, looseness=0.9] (3,7);
\draw [purple strand] (3,2) to (3,3);
\draw (2,2) to (2,10);
\draw [purple strand] (3.5,7) to (3.5,8);
\draw [purple strand] (3.5,8) to [out=up, in=down, looseness=0.9] (1,10);
\draw [purple strand] (3.5,3) to (3.5,5);
\draw [black strand] (6,0) node [below] {$Y$} to [out=up, in=down] (2.5,2);
\draw [purple strand] (4.5,0) node [below] {$\mathcal{F}$} to (4.5,5);
\node [2halfbox] at (4,5) {$m$};
\node [1halfbox] at (3.25,3) {$\Delta$};
\node [2halfbox] at (2.5,2) {$\varrho_Y$};
\node [halfbox] at (3.5,4) {$S$};
\node [halfbox] at (4.5,4) {$S$};
\draw (0,0) node [below] {$X$} to (0,10);
\end{tz}
\gap=\gap
     \begin{tz}[xscale=0.5,yscale=0.7]
\node [2halfbox] at (3.5,7) {$m$};
\node [halfbox] at (3.5,8) {$S^\inv$};
\draw [purple strand] (4,6) to (4,7);
\draw [purple strand] (4.5,5) to [out=up, in=down, looseness=0.9] (4,6);
\draw [purple strand] (3.25,5) to (3.25,7);
\draw [purple strand] (2.75,3) to (2.75,5);
\draw [purple strand] (3,2) to (3,3);
\draw (2,2) to (2,10);
\draw [purple strand] (3.5,7) to (3.5,8);
\draw [purple strand] (3.5,8) to [out=up, in=down, looseness=0.9] (1,10);
\draw [purple strand] (3.5,3) to (3.5,5);
\draw [black strand] (6,0) node [below] {$Y$} to [out=up, in=down] (2.5,2);
\draw [purple strand] (4.5,0) node [below] {$\mathcal{F}$} to (4.5,5);
\node [2halfbox] at (3.25,5) {$m$};
\node [2halfbox] at (3.25,3) {$\Delta$};
\node [2halfbox] at (2.5,2) {$\varrho_Y$};
\node [halfbox] at (3.5,4) {$S$};
\node [halfbox] at (4.5,4) {$S$};
\draw (0,0) node [below] {$X$} to (0,10);
\end{tz}
\gap=\gap
     \begin{tz}[xscale=0.5,yscale=0.7]
\node [2halfbox] at (3.5,7) {$m$};
\node [halfbox] at (3.5,8) {$S^\inv$};
\draw [purple strand] (4,6) to (4,7);
\draw [purple strand] (4.5,5) to [out=up, in=down, looseness=0.9] (4,6);
\draw [purple strand] (3.25,5) to (3.25,7);
\draw [purple strand] (3,2) to (3,3);
\draw (2,2) to (2,10);
\draw [purple strand] (3.5,7) to (3.5,8);
\draw [purple strand] (3.5,8) to [out=up, in=down, looseness=0.9] (1,10);
\draw [black strand] (6,0) node [below] {$Y$} to [out=up, in=down] (2.5,2);
\draw [purple strand] (4.5,0) node [below] {$\mathcal{F}$} to (4.5,5);
\node [halfbox] at (3.25,5) {$u$};
\node [halfbox] at (3,3) {$\varepsilon$};
\node [2halfbox] at (2.5,2) {$\varrho_Y$};
\node [halfbox] at (4.5,4) {$S$};
\draw (0,0) node [below] {$X$} to (0,10);
\end{tz}
\gap=\gap
\end{align*}

\begin{align*}
         \begin{tz}[xscale=0.5,yscale=0.7]
\node [halfbox] at (4,3) {$S^\inv$};
\draw (2,2) to (2,3);
\draw [black strand] (2,3) to [out=up, in=down] (4,5);
\draw [purple strand] (4,3) to [out=up, in=down, looseness=0.9] (2,5);
\draw [purple strand] (4,2) to (4,3);
\draw [black strand] (4,0) node [below] {$Y$} to [out=up, in=down] (2,2);
\draw [purple strand] (2,0) node [below] {$\mathcal{F}$} to [out=up, in=down, looseness=0.9] (4,2);
\node [halfbox] at (4,2) {$S$};
\draw (0,0) node [below] {$X$} to (0,5);
\end{tz}
\gap=\gap
         \begin{tz}[xscale=0.5,yscale=0.7]
\draw [black strand] (4,0) node [below] {$Y$} to (4,5);
\draw [purple strand] (2,0) node [below] {$\mathcal{F}$} to (2,5);
\draw (0,0) node [below] {$X$} to (0,5);
\end{tz}
\end{align*}

Where in the second equality we used the compatibility of the coaction $\varrho_Y$ with the comultiplication $\Delta$ of $\coend$, in the third equality, the so called 'anti-multiplicativity' \cite[6.2.1]{turaev_monoidal_2017} of the antipode of a Hopf algebra, in the fourth, the fundamental relation for the antipode of the Hopf algebra $\coend$, and finally in the penultimate equality, the unit and counit axioms for $\coend$.

As for the composite $\mathcal{Z}(\phi_r^\inv)\circ \mathcal{Z}(\phi_r)$:

\begin{align*}
 \begin{tz}[xscale=0.5,yscale=0.7]
\node [2halfbox] at (2.5,1) {$\varrho_Y$};
\node [halfbox] at (4,2) {$S$};
\node [2halfbox] at (3.5,3) {$m$};
\node [halfbox] at (3.5,4) {$S^\inv$};
\draw [purple strand] (4,0) node [below] {$\coend$} to (4,3);
\draw [purple strand] (3,1) to (3,3);
\draw (2,1) to (2,4);
\draw [black strand] (2,4) to [out=up, in=down, looseness=0.9] (4.5,6);
\draw [purple strand] (3.5,3) to (3.5,4);
\draw [purple strand] (3.5,4) to [out=up, in=down, looseness=0.9] (2.5,6);
\draw (0,0) node [below] {$X$} to (0,6);
\draw (2.5,0) node [below] {$Y$} to (2.5,1);
\draw [black strand] (4.25,7) to [out=up, in=down] (1,8.5);
\draw [purple strand] (3.5,9) to (3.5,10) node [above] {$\mathcal{F}$};
\draw [purple strand] (2.5,6) to (2.5,9);
\draw [purple strand] (4.75,7) to (4.75,9);
\node [box] at (3.5,9) {$m$};
\node [1halfbox] at (4.5,7) {$\varrho_Y$};
\draw (1,8.5) to (1,10) node [above] {$Y$};
\draw (0,6) to (0,10);
\draw [black strand] (4.5,6) to [out=up, in=down] (4.5,7);
\end{tz}
\gap=\gap
 \begin{tz}[xscale=0.5,yscale=0.7]
\node [2halfbox] at (2.5,1) {$\varrho_Y$};
\node [1halfbox] at (2,2) {$\varrho_Y$};
\node [halfbox] at (4,2) {$S$};
\node [2halfbox] at (3.5,3) {$m$};
\node [halfbox] at (3.5,4) {$S^\inv$};
\draw [purple strand] (4,0) node [below] {$\coend$} to (4,3);
\draw [purple strand] (3,1) to (3,3);
\draw (2,1) to (2,2);
\draw [black strand] (1.75,2) to [out=up, in=down, looseness=0.9] (1,4);
\draw [purple strand] (2.25,2) to (2.25,4);
\draw [purple strand] (2.25,4) to [out=up, in=down] (3,6);
\draw [purple strand] (3.5,3) to (3.5,4);
\draw [purple strand] (3.5,4) to [out=up, in=down, looseness=0.9] (2,6);
\draw (0,0) node [below] {$X$} to (0,7);
\draw (2.5,0) node [below] {$Y$} to (2.5,1);
\draw [purple strand] (2.5,6) to (2.5,7) node [above] {$\mathcal{F}$};
\node [2halfbox] at (2.5,6) {$m$};
\draw (1,4) to (1,7) node [above] {$Y$};
\end{tz}
\gap=\gap
 \begin{tz}[xscale=0.5,yscale=0.7]
\node [2halfbox] at (1.5,1) {$\varrho_Y$};
\node [1halfbox] at (2,2) {$\Delta$};
\node [halfbox] at (4,2) {$S$};
\node [2halfbox] at (3.5,3) {$m$};
\node [halfbox] at (3.5,4) {$S^\inv$};
\draw [purple strand] (4,0) node [below] {$\coend$} to (4,3);
\draw [purple strand] (2,1) to (2,2);
\draw [purple strand] (1.75,2) to (1.75,4);
\draw [purple strand] (1.75,4) to [out=up, in=down] (3,6);
\draw [purple strand] (2.25,2) to [out=up, in=down] (3,3);
\draw [purple strand] (3.5,3) to (3.5,4);
\draw [purple strand] (3.5,4) to [out=up, in=down, looseness=0.9] (2,6);
\draw (0,0) node [below] {$X$} to (0,7);
\draw (1.5,0) node [below] {$Y$} to (1.5,1);
\draw [purple strand] (2.5,6) to (2.5,7) node [above] {$\mathcal{F}$};
\node [2halfbox] at (2.5,6) {$m$};
\draw (1,1) to (1,7) node [above] {$Y$};
\end{tz}
\gap=\gap
\end{align*}

\begin{align*}
    \begin{tz}[xscale=0.5,yscale=0.7]
\node [2halfbox] at (1.5,1) {$\varrho_Y$};
\node [1halfbox] at (2,2) {$\Delta$};
\node [halfbox] at (4,2) {$S$};
\node [halfbox] at (3.5,5) {$S$};
\node [halfbox] at (1.75,5) {$S$};
\node [2halfbox] at (3.5,3) {$m$};
\node [halfbox] at (3.5,4) {$S^\inv$};
\node [halfbox] at (3,7) {$S^\inv$};
\draw [purple strand] (4,0) node [below] {$\coend$} to (4,3);
\draw [purple strand] (2,1) to (2,2);
\draw [purple strand] (1.75,2) to (1.75,6);
\draw [purple strand] (2.25,2) to [out=up, in=down] (3,3);
\draw [purple strand] (3.5,3) to (3.5,6);
\draw (0,0) node [below] {$X$} to (0,8);
\draw (1.5,0) node [below] {$Y$} to (1.5,1);
\draw [purple strand] (3,6) to (3,8) node [above] {$\mathcal{F}$};
\node [box] at (3,6) {$m$};
\draw (1,1) to (1,8) node [above] {$Y$};
\end{tz}
\gap=\gap
\begin{tz}[xscale=0.5,yscale=0.7]
\node [2halfbox] at (1.5,1) {$\varrho_Y$};
\node [1halfbox] at (2,2) {$\Delta$};
\node [halfbox] at (4,2) {$S$};
\node [halfbox] at (1.75,3) {$S$};
\node [2halfbox] at (3.5,3) {$m$};
\node [halfbox] at (3,5) {$S^\inv$};
\draw [purple strand] (4,0) node [below] {$\coend$} to (4,3);
\draw [purple strand] (2,1) to (2,2);
\draw [purple strand] (1.75,2) to (1.75,4);
\draw [purple strand] (2.25,2) to [out=up, in=down] (3,3);
\draw [purple strand] (3.5,3) to (3.5,4);
\draw (0,0) node [below] {$X$} to (0,6);
\draw (1.5,0) node [below] {$Y$} to (1.5,1);
\draw [purple strand] (3,4) to (3,6) node [above] {$\mathcal{F}$};
\node [box] at (3,4) {$m$};
\draw (1,1) to (1,6) node [above] {$Y$};
\end{tz}
\gap=\gap
\begin{tz}[xscale=0.5,yscale=0.7]
\node [2halfbox] at (1.5,1) {$\varrho_Y$};
\node [1halfbox] at (2,2) {$\Delta$};
\node [halfbox] at (4,2) {$S$};
\node [halfbox] at (1.75,3) {$S$};
\node [1halfbox] at (2,4) {$m$};
\node [halfbox] at (3,6) {$S^\inv$};
\draw [purple strand] (4,0) node [below] {$\coend$} to (4,5);
\draw [purple strand] (2,1) to (2,2);
\draw [purple strand] (1.75,2) to (1.75,4);
\draw [purple strand] (2.25,2) to [out=up, in=down] (2.75,3);
\draw [purple strand] (2.75,3) to [out=up, in=down] (2.25,4);
\draw [purple strand] (2,4) to (2,5);
\draw (0,0) node [below] {$X$} to (0,7);
\draw (1.5,0) node [below] {$Y$} to (1.5,1);
\draw [purple strand] (3,5) to (3,7) node [above] {$\mathcal{F}$};
\node [box] at (3,5) {$m$};
\draw (1,1) to (1,7) node [above] {$Y$};
\end{tz}
\gap=\gap
\end{align*}

\begin{align*}
    \begin{tz}[xscale=0.5,yscale=0.7]
\node [2halfbox] at (1.5,1) {$\varrho_Y$};
\node [halfbox] at (2,2) {$\varepsilon$};
\node [halfbox] at (4,2) {$S$};
\node [halfbox] at (2,4) {$u$};
\node [halfbox] at (3,6) {$S^\inv$};
\draw [purple strand] (4,0) node [below] {$\coend$} to (4,5);
\draw [purple strand] (2,1) to (2,2);
\draw [purple strand] (2,4) to (2,5);
\draw (0,0) node [below] {$X$} to (0,7);
\draw (1.5,0) node [below] {$Y$} to (1.5,1);
\draw [purple strand] (3,5) to (3,7) node [above] {$\mathcal{F}$};
\node [box] at (3,5) {$m$};
\draw (1,1) to (1,7) node [above] {$Y$};
\end{tz}
\gap=\gap
\begin{tz}[xscale=0.5,yscale=0.7]
\node [halfbox] at (2,1) {$S$};
\node [halfbox] at (2,2) {$S^\inv$};
\draw [purple strand] (2,0) node [below] {$\coend$} to (2,3);
\draw (0,0) node [below] {$X$} to (0,3);
\draw (1,0) node [below] {$Y$} to (1,3);
\end{tz}
\gap=\gap
\begin{tz}[xscale=0.5,yscale=0.7]
\draw [purple strand] (2,0) node [below] {$\coend$} to (2,3);
\draw (0,0) node [below] {$X$} to (0,3);
\draw (1,0) node [below] {$Y$} to (1,3);
\end{tz}
\end{align*}

Where the second equality is due to the compatibility of the coaction with the comultiplication $\Delta$. The third equality is again due to the 'anti-multiplicativity' of the antipode $S$, but now we have pre-composed it with the inverse braiding and post-composed with $S^\inv$. The fourth equality is given by the associativity of $m$, the fifth by the defining relation of the antipode $S$, and the penultimate by the relations between unit-multiplication and counit-coaction.

\item Fixing $X\boxtimes Y\in \C\boxtimes\C$ , relation \eqref{tortile} demands that for all $\phi\in\Hom(X\otimes Y,\unit)^*$, it is true that $\phi(-\circ\vartheta_X\otimes\text{id}_Y)=\phi(-\circ\text{id}_X\otimes\vartheta_Y)$. Since this has to hold for all $\phi$, that implies that the arguments of $\phi$ have to be equal.

In other words, we will show the equality of the following 2-morphisms:

$$\Hom(X\otimes Y,\unit)\xrightarrow{-\circ(\vartheta_X\otimes\id_Y)}\Hom(X\otimes Y,\unit)=
\Hom(X\otimes Y,\unit)\xrightarrow{-\circ(\id_X\otimes\vartheta_Y)}\Hom(X\otimes Y,\unit).$$

To do this, let $g\in \Hom(X\otimes Y,\unit)$. Then we have the following:

\begin{align*}
\begin{tz}[xscale=0.8]
\draw (2,0) node [below] {$Y$} to (2,2);
\draw (1,0) node [below] {$X$} to (1,2);
\node [halfbox] at (2,1) {$\vartheta_{Y}$};
\node [box] at (1.5,2) {$g$};
\end{tz}
\gap=\gap
\begin{tz}[xscale=0.8]
\draw (1,0) node [below] {$X$} to (1,2);
\draw (4,0) node [below] {$Y$} to (4,3);
\node [halfbox] at (4,1) {$\vartheta_Y$};
\node [box] at (1.5,2) {$g$};
\draw [black strand] (2,1) to [out=down, in=down, looseness=1] (3,1);
\draw [black strand] (3,3) to [out=up, in=up, looseness=1] (4,3);
\draw (3,1) to (3,3);
\draw (2,1) to (2,2);
\end{tz}
\gap=\gap
\begin{tz}[xscale=0.8]
\draw (1,0) node [below] {$X$} to (1,1.5);
\draw (4,0) node [below] {$Y$} to (4,3);
\node [halfbox] at (3,2.5) {$\vartheta_{Y^\lor}$};
\node [box] at (1.5,1.5) {$g$};
\draw [black strand] (2,1) to [out=down, in=down, looseness=1] (3,1);
\draw [black strand] (3,3) to [out=up, in=up, looseness=1] (4,3);
\draw (3,1) to (3,3);
\draw (2,1) to (2,1.5);
\end{tz}
\end{align*}

\begin{align*}
\gap=\gap
\begin{tz}[xscale=0.8]
\draw (1,0) node [below] {$X$} to (1,2);
\draw (4,0) node [below] {$Y$} to (4,3);
\node [halfbox] at (1,0.5) {$\vartheta_{X}$};
\node [box] at (1.5,2) {$g$};
\draw [black strand] (2,1) to [out=down, in=down, looseness=1] (3,1);
\draw [black strand] (3,3) to [out=up, in=up, looseness=1] (4,3);
\draw (3,1) to (3,3);
\draw (2,1) to (2,2);
\end{tz}
\gap=\gap
\begin{tz}[xscale=0.8]
\draw (2,0) node [below] {$Y$} to (2,2);
\draw (1,0) node [below] {$X$} to (1,2);
\node [halfbox] at (1,1) {$\vartheta_{X}$};
\node [box] at (1.5,2) {$g$};
\end{tz}
\end{align*}

Where in the second equality we used the duality property of the twist $\vartheta$, and in the third the naturality of $\vartheta$ with respect to the morphism $$(g\otimes\id_{Y^\lor})\circ(\id_X\otimes\text{coev}_Y)\colon X\to Y^\lor.$$



\item We then continue by checking relation \eqref{adj_eta_epsilon1}. This implies that we want:

$$ X \xrightarrow{\delta_X} X\otimes \coend \xrightarrow{\varepsilon} X = \id_X $$

Which is true because of the relation between the coaction $\delta_X$ and the counit $\varepsilon$. 

\item Relation \eqref{adj_eta_epsilon2} is verified in exactly the same way, and in some sense it is a special case of \eqref{adj_eta_epsilon1}:

$$(X\otimes\coend,\Delta)\xrightarrow{\Delta}((X\otimes\coend)\otimes\coend,\Delta)\xrightarrow{\id_X\otimes\varepsilon\otimes\id_\coend} (X\otimes\coend,\Delta)$$


\item Now, checking the relations \eqref{adj_eta_epsilon_dag1} and \eqref{adj_eta_epsilon_dag2} respectively:

\begin{align*}
    \begin{tz}[xscale=0.5,yscale=0.7]
\node [2halfbox] at (2.5,1) {$\varrho_X$};
\node [halfbox] at (3,2) {$S$};
\node [2halfbox] at (3.5,3) {$m$};
\node [halfbox] at (3.5,4) {$\Lambda^{co}$};
\node [halfbox] at (4,0) {$\Lambda$};
\draw [purple strand] (4,0) to (4,3);
\draw [purple strand] (3,1) to (3,3);
\draw (2,1) to (2,4);
\draw [purple strand] (3.5,3) to (3.5,4);
\draw (2.5,0) node [below] {$X$} to (2.5,1);
\end{tz}
\gap=\gap
    \begin{tz}[xscale=0.5,yscale=0.7]
\node [2halfbox] at (2.5,1) {$\varrho_X$};
\node [halfbox] at (4,2) {$S$};
\node [2halfbox] at (3.5,3) {$m$};
\node [halfbox] at (3.5,4) {$\Lambda^{co}$};
\node [halfbox] at (4,0) {$\Lambda$};
\draw [purple strand] (4,0) to (4,3);
\draw [purple strand] (3,1) to (3,3);
\draw (2,1) to (2,4);
\draw [purple strand] (3.5,3) to (3.5,4);
\draw (2.5,0) node [below] {$X$} to (2.5,1);
\end{tz}
\gap=\gap
    \begin{tz}[xscale=0.5,yscale=0.7]
\node [2halfbox] at (2.5,1) {$\varrho_X$};
\node [2halfbox] at (3.5,3) {$m$};
\node [halfbox] at (3.5,4) {$\Lambda^{co}$};
\node [halfbox] at (4,0) {$\Lambda$};
\draw [purple strand] (4,0) to (4,3);
\draw [purple strand] (3,1) to (3,3);
\draw (2,1) to (2,4);
\draw [purple strand] (3.5,3) to (3.5,4);
\draw (2.5,0) node [below] {$X$} to (2.5,1);
\end{tz}
\gap=\gap
    \begin{tz}[xscale=0.5,yscale=0.7]
\node [2halfbox] at (2.5,1) {$\varrho_X$};
\node [halfbox] at (3,2) {$\varepsilon$};
\node [halfbox] at (3,4) {$\Lambda^{co}$};
\node [halfbox] at (3,3) {$\Lambda$};
\draw [purple strand] (3,1) to (3,2);
\draw (2,1) to (2,4);
\draw [purple strand] (3,3) to (3,4);
\draw (2.5,0) node [below] {$X$} to (2.5,1);
\end{tz}
\gap=\gap
    \begin{tz}[xscale=0.5,yscale=0.7]
\draw (0,0) node [below] {$X$} to (0,4);
\end{tz}
\end{align*}

Where in the first and second equality we used Lemma \ref{integral_antipode_relations} (iii), (ii) respectively, and in the third, the definition of the integral $\Lambda$.

\begin{align*}
        \begin{tz}[xscale=0.5,yscale=0.7]
\node [halfbox] at (2.5,2) {$S$};
\node [2halfbox] at (2,1) {$\Delta$};
\node [halfbox] at (2,0) {$\Lambda$};
\node [2halfbox] at (3,3) {$m$};
\node [halfbox] at (3,4) {$\Lambda^{co}$};
\draw [purple strand] (3.5,0) node [below] {$\coend$} to (3.5,3);
\draw [purple strand] (2,0) to (2,1);
\draw [purple strand] (2.5,1) to (2.5,3);
\draw [purple strand] (1.5,1) to (1.5,4);
\draw [purple strand] (3,3) to (3,4);
\draw (1,0) node [below] {$X$} to (1,4);
\end{tz}
\gap=\gap
\begin{tz}[xscale=0.5,yscale=0.7]
\draw [purple strand] (2,0) node [below] {$\coend$} to (2,4);
\draw (1,0) node [below] {$X$} to (1,4);
\end{tz}
\end{align*}

By \cite[Lemma 4.2.12]{kerler_non-semisimple_2001}.

\item Now we turn to relations \eqref{adj_nu_mu_dag1} and \eqref{adj_nu_mu_dag2}. Note that these manifest $U^L$ as the left adjoint of $U$.

To check relation \eqref{adj_nu_mu_dag1}, we write $\unit$ as a colimit of projectives. We can do this using Proposition \ref{coYoneda}: 
$$\unit\cong\int^{P\in\cP}\hspace{-20pt}\Hom_\C(P,\unit)\otimes P.$$ 


The isomorphism is given by $\int EV$, so we have a map
$$\unit\xrightarrow{(\int EV)^\inv}\int^{P\in\cP}\hspace{-20pt}\Hom_\C(P,\unit)\otimes P.$$

So, for relation \eqref{adj_nu_mu_dag1} to hold, the following diagram needs to commute:

\tikzset{every node/.style={font=\normalsize}}

\[\begin{tikzcd}[row sep=large,column sep=huge]
	{\int^P\text{Hom}_\mathcal{C}(P,\unit)\otimes P} & \unit \\
	{\int^P\text{Hom}_\mathcal{C}(P,\unit)\otimes\text{Hom}_\mathcal{C}(P,\unit)^*\otimes\text{Hom}_\mathcal{C}(P,\unit)\otimes P} \\
	{\int^P\text{Hom}_\mathcal{C}(P,\unit)\otimes\text{Hom}_\mathcal{C}(P,\unit)^*\otimes\unit} & \unit
	\arrow["{(\int EV)^{-1}}"', from=1-2, to=1-1]
	\arrow["{\int(\text{id}_{\text{Hom}_\mathcal{C}(P,\unit)}\otimes\text{coev}_{\text{Hom}_\mathcal{C}(P,\unit)}\otimes \text{id}_P)}"', from=1-1, to=2-1]
	\arrow["{\int(\text{id}_{\text{Hom}_\mathcal{C}(P,\unit)\otimes\text{Hom}_\mathcal{C}(P,\unit)^*}\otimes EV_P)}"', from=2-1, to=3-1]
	\arrow["{\text{id}_\unit}", from=1-2, to=3-2]
	\arrow["{\int(ev_\text{vect})\otimes\text{id}_\unit}"', from=3-1, to=3-2]
\end{tikzcd}\]

Indeed, this diagram is commutative as the evaluation and coevaluation of vector spaces satisfy a snake relation, and then $\int EV$ cancels with $(\int EV)^\inv$.

\item As for relation \eqref{adj_nu_mu_dag2}, we want the following diagram to be commutative: 

\[\begin{tikzcd}
	{\text{Hom}_\mathcal{C}(P,\unit)^*} & {\text{Hom}_\mathcal{C}(P,\unit)^*} \\
	{\text{Hom}_\mathcal{C}(P,\unit)^*\otimes\displaystyle\int^{Q\in\cP}\hspace{-20pt}\text{Hom}_\mathcal{C}(Q,\unit)\otimes\text{Hom}_\mathcal{C}(Q,\unit)^*}
	\arrow["{\mathcal{Z}(\mu^\dagger)}", from=1-1, to=2-1]
	\arrow["{\text{id}_{\text{Hom}_\mathcal{C}(P,\unit)^*}}", from=1-1, to=1-2]
	\arrow["{\mathcal{Z}(\nu^\dagger)}"', from=2-1, to=1-2]
\end{tikzcd}\]


Note that $\mathcal{Z}(\mu^\dagger)$ is applied on $P$, the argument of $\text{Hom}_\mathcal{C}(P,\unit)^*$. For sake of simplicity, we will denote the maps as applied on $P$, instead of $\text{Hom}_\mathcal{C}(P,\unit)^*$. Then the above diagram looks like this:

\[\!\!\!\!\!\!\!\!\!\!\!\!\!\!\!\!\!\!\!\!\!\!\!\!\!\!\!\!\!\!\!\!\!\!\!\!\begin{tikzcd}
	{\text{Hom}_\mathcal{C}(P,\unit)^*} & {\text{Hom}_\mathcal{C}(P,\unit)^*} \\
	{\text{Hom}_\mathcal{C}(\text{Hom}_\mathcal{C}(P,\unit)^*\otimes\text{Hom}_\mathcal{C}(P,\unit)\otimes P,\unit)^*} & {\text{Hom}_\mathcal{C}(P,\unit)^*\otimes\displaystyle\int^Q\hspace{-10pt}\text{Hom}_\mathcal{C}(Q,\unit)\otimes\text{Hom}_\mathcal{C}(Q,\unit)^*} \\
	{\text{Hom}_\mathcal{C}(\text{Hom}_\mathcal{C}(P,\unit)^*\otimes\unit,\unit)^*} \\
	& {\text{Hom}_\mathcal{C}(\text{Hom}_\mathcal{C}(P,\unit)^*\otimes\displaystyle\int^Q\hspace{-10pt}\text{Hom}_\mathcal{C}(Q,\unit)\otimes Q,\unit)^*}
	\arrow["{\text{id}_{\text{Hom}_\mathcal{C}(P,\unit)^*}}"', from=1-1, to=1-2]
	\arrow["{\text{coev}_{vect}\otimes \text{id}_P}", from=1-1, to=2-1]
	\arrow["{\text{id}_{\text{Hom}_\mathcal{C}(P,\unit)}\otimes\text{EV}_P}", from=2-1, to=3-1]
	\arrow["{\int \text{ev}_{vect}}", from=2-2, to=1-2]
	\arrow["{\text{id}_{\text{Hom}_\mathcal{C}(P,\unit)^*}\otimes(\int\text{EV})^{-1}}"{pos=0.3}, from=3-1, to=4-2]
	\arrow["\sim", from=4-2, to=2-2]
\end{tikzcd}\]


where by $\sim$ we have denoted the isomorhism of coends induced by $$\Hom(V\otimes X,Y)^*\cong V \otimes \Hom(X,Y)^*,$$ for $V=\text{Hom}_\mathcal{C}(P,\unit)^*\otimes\text{Hom}_\mathcal{C}(Q,\unit)$, $X=Q$ and $Y=\unit$.

This diagram is indeed commutative and we now explain why.

First, looking at the composite $(\int\text{EV})^{-1}\circ\text{EV}_P$, this is equal to $d^1_P$ , where $d^1$ is the dinatural family of maps for the coend $\int^P \text{Hom}_\mathcal{C}(P,\unit)\otimes P$. Then this, (post-)composed with $\sim$, is equal to $d^2_P$, where by $d^2$ we denote the dinatural family of maps for the coend $\int^P \text{Hom}_\mathcal{C}(P,\unit)\otimes \Hom(P,\unit)^*$.

Then, by the definition of $\int \text{ev}_{vect}$ and by the snake relation satisfied by $\text{coev}_{\Hom_\mathcal{C}(P,\unit)}$ and $\text{ev}_{\Hom_\mathcal{C}(P,\unit)}$ we get the desired result.



\item The pivotality relation \eqref{piv_on_sphere} says that a certain composite of 2-morphisms is the identity 2-morphism on $S^2  $. However, since this composite is in fact an endomorphism of the cylinder, we will check something stronger. Namely, we will show that the relation holds for $\mathcal{Z}(\tinycup)$, instead of $\mathcal{Z}(\begin{tikzpicture}
    \node [Cap] at (0,0) {};
    \node[Cup] at (0,0) {};
    \end{tikzpicture})$.

 We first need to make sense of $\mu\circ\mu^\dagger$ applied on 'one leg' as a morphism $\coend\to\coend$. To do this, we start with $P_\unit^\lor\otimes P_\unit$. Relation \eqref{piv_on_sphere} instructs us to apply $\mu\circ\mu^\dagger$ on $P_\unit$. So we have the following:

$$P_\unit^\lor\otimes P_\unit\xrightarrow{\id_{P_\unit^\lor}\otimes (\eta_\unit\circ\varepsilon_\unit)}P_\unit^\lor\otimes P_\unit\xrightarrow{i_{P_\unit}}\coend$$

Diagrammatically, we have the following equality:

\tikzset{every node/.style={font=\tiny}}

\begin{align*}
    \begin{tz}[xscale=0.4,yscale=0.7]
\node [box] at (1,3) {$i_{P_\unit}$};
\node [halfbox] at (2,1) {$\varepsilon_\unit$};
\node [halfbox] at (2,2) {$\eta_\unit$};
\draw [purple strand] (1,3) to (1,4) node [above] {$\mathcal{F}$};
\draw (2,2) to (2,3);
\draw (2,0) node [below] {$P_\unit$} to (2,1);
\draw (0,0) node [below] {$P_\unit^\lor$} to (0,3);
\end{tz}
\gap=\gap
    \begin{tz}[xscale=0.4,yscale=0.7]
\node [2halfbox] at (1.5,3) {$i_{P_\unit}$};
\node [halfbox] at (4,1) {$\varepsilon_\unit$};
\node [halfbox] at (4,2) {$\eta_\unit$};
\draw [purple strand] (1.5,3) to (1.5,4) node [above] {$\mathcal{F}$};
\draw (2,1) to (2,3);
\draw (3,1) to (3,3);
\draw (4,2) to (4,3);
\draw (4,0) node [below] {$P_\unit$} to (4,1);
\draw (1,0) node [below] {$P_\unit^\lor$} to (1,3);
\draw [black strand] (2,1) to [out=down, in=down, looseness=1] (3,1);
\draw [black strand] (3,3) to [out=up, in=up, looseness=1] (4,3);
\end{tz}
\gap=\gap
    \begin{tz}[xscale=0.4,yscale=0.7]
\node [2halfbox] at (1.5,3) {$i_{P_\unit}$};
\node [halfbox] at (4.5,2) {$\varepsilon_\unit$};
\node [halfbox] at (3,2) {$\eta_\unit^\lor$};
\draw [purple strand] (1.5,3) to (1.5,4) node [above] {$\mathcal{F}$};
\draw (2,1) to (2,3);
\draw (3,1) to (3,2);
\draw (4.5,0) node [below] {$P_\unit$} to (4.5,2);
\draw (1,0) node [below] {$P_\unit^\lor$} to (1,3);
\draw [black strand] (2,1) to [out=down, in=down, looseness=1] (3,1);
\end{tz}
\gap=\gap
\begin{tz}[xscale=0.5,yscale=0.7]
\node [halfbox] at (2,3.5) {$\Lambda^{co}$};
\draw [purple strand] (0,2.5) to (0,3.5) node [above] {$\mathcal{F}$};
\draw [purple strand] (2,2.5) to (2,3.5);
\node [2halfbox] at (2,2.5) {$i_{P_j}$};
\node [2halfbox] at (0,2.5) {$i_{P_j}$};
\draw [black strand] (2,0) node [below] {$P_j$} to (2,1);
\draw [black strand] (0,0) node [below] {$P_j^\lor$} to (0,1);
\draw [black strand] (0,1) to [out=up, in=up, looseness=0.9] (-0.5,2.5);
\draw [black strand] (2,1) to [out=up, in=up, looseness=0.9] (2.5,2.5);
\draw [black strand] (0.5,2.5) to [out=down, in=down, looseness=4] (1.5,2.5);
\end{tz}
\end{align*}

Where in the last step we make use of Lemma \ref{coint_eta_eps}, and the fact that the morphism is zero for all $j\neq 1$. 

This then induces the following morphism $\coend\to\coend$:

\begin{align*}
        \begin{tz}[xscale=0.4,yscale=1]
\node [box] at (1,1) {$\Delta$};
\node [halfbox] at (2,2) {$\Lambda^{co}$};
\draw [purple strand] (1,0) node [below] {$\mathcal{F}$} to (1,1);
\draw [purple strand] (0,1) to (0,2);
\draw [purple strand] (2,1) to (2,2);
\end{tz}
\end{align*}



Given the above, checking that relation \eqref{piv_on_sphere} holds, boils down to showing that the composite $\unit\xrightarrow{\Lambda} \coend \xrightarrow{(\id\otimes \Lambda^{co})\circ \Delta} \coend \xrightarrow{\varepsilon} \unit$ is equal to the identity. This is true, due to the relation of $\varepsilon$ and $\Delta$, as well as the fact that $\Lambda^{co}\circ\Lambda= \id_\unit$.

 Note that for this relation to be satisfied, we also had $\gamma$ from $\mathcal{Z}(\epsilon^\dagger)$ cancel out $\gamma^\inv$ from $\mathcal{Z}(\mu)$.

\item To understand the composite given by the modularity relation \eqref{MOD}, let us first understand how we're meant to apply "$\vartheta,\vartheta^\inv$".

The object at the level where $\vartheta,\vartheta^\inv$ is applied is 
$\int^Y X\otimes Y^\lor\boxtimes Y$. Therefore, '$\mathcal{Z}(\theta,\theta^\inv)$' is the coend morphism induced by $\vartheta_{X\otimes Y^\lor}\boxtimes\vartheta_Y^\inv$.

After using the twist relation for $\vartheta_{X\otimes Y^\lor}$, this looks as follows:

\begin{align*}
\begin{tz}[xscale=0.8]
\draw (2,0) node [below] {$Y$} to (2,3);
\draw (1,0) node [below] {$Y^\lor$} to [out=up, in=down] (0,1);
\draw [black strand] (0,0) node [below] {$X$} to [out=up, in=down] (1,1);
\draw [black strand] (1,1) to [out=up, in=down] (0,2);
\draw [black strand] (0,2) to [out=up, in=down] (0,4);
\draw [black strand] (0,1) to [out=up, in=down] (1,2);
\draw [black strand] (1,2) to [out=up, in=down] (1,3);
\node [halfbox] at (0,2) {$\vartheta_X$};
\node [halfbox] at (1,2) {$\vartheta_{Y^\lor}$};
\node [halfbox] at (2,2) {$\vartheta^\inv_Y$};
\node [box] at (1.5,3) {$i_Y$};
\draw [purple strand] (1.5,3) to (1.5,4);
\end{tz}
\gap=\gap
\begin{tz}[xscale=0.8]
\draw (2,0) node [below] {$Y$} to (2,3);
\draw (1,0) node [below] {$Y^\lor$} to [out=up, in=down] (0,1);
\draw [black strand] (0,0) node [below] {$X$} to [out=up, in=down] (1,1);
\draw [black strand] (1,1) to [out=up, in=down] (0,2);
\draw [black strand] (0,2) to [out=up, in=down] (0,4);
\draw [black strand] (0,1) to [out=up, in=down] (1,2);
\draw [black strand] (1,2) to [out=up, in=down] (1,3);
\node [halfbox] at (0,3) {$\vartheta_X$};
\node [halfbox] at (2,2) {$\vartheta_{Y}$};
\node [halfbox] at (2,1) {$\vartheta^\inv_Y$};
\node [box] at (1.5,3) {$i_Y$};
\draw [purple strand] (1.5,3) to (1.5,4);
\end{tz}
\gap=\gap
\begin{tz}[xscale=0.8]
\draw (2,0) node [below] {$Y$} to (2,3);
\draw (1,0) node [below] {$Y^\lor$} to [out=up, in=down] (0,1);
\draw [black strand] (0,0) node [below] {$X$} to [out=up, in=down] (1,1);
\draw [black strand] (1,1) to [out=up, in=down] (0,2);
\draw [black strand] (0,2) to [out=up, in=down] (0,3);
\draw [black strand] (0,1) to [out=up, in=down] (1,2);
\draw [black strand] (1,2) to [out=up, in=down] (1,3);
\node [halfbox] at (-2,2) {$\vartheta_X$};
\node [box] at (1.5,3) {$i_Y$};
\draw [purple strand] (1.5,3) to (1.5,4);
\draw [black strand] (-2,1) to [out=down, in=down, looseness=1] (-1,1);
\draw [black strand] (-1,3) to [out=up, in=up, looseness=1] (0,3);
\draw (-2,1) to (-2,4);
\draw (-1,1) to (-1,3);
\end{tz}
\end{align*}

Where in the first equality we used the coend property as well as the fact that $\vartheta_{X^\lor}=\vartheta^\lor_{X}$ for all $X\in \mathcal{C}$.

If we include the $\varepsilon$ that we have in the modularity composite we get:

\begin{align*}
    \begin{tz}[xscale=0.8]
\draw (2,0) node [below] {$Y$} to (2,3);
\draw (1,0) node [below] {$Y^\lor$} to [out=up, in=down] (0,1);
\draw [black strand] (0,0) node [below] {$X$} to [out=up, in=down] (1,1);
\draw [black strand] (1,1) to [out=up, in=down] (0,2);
\draw [black strand] (0,2) to [out=up, in=down] (0,3);
\draw [black strand] (0,1) to [out=up, in=down] (1,2);
\draw [black strand] (1,2) to [out=up, in=down] (1,3);
\node [halfbox] at (-2,2) {$\vartheta_X$};
\node [box] at (1.5,3) {$i_Y$};
\draw [purple strand] (1.5,3) to (1.5,4);
\draw [black strand] (-2,1) to [out=down, in=down, looseness=1] (-1,1);
\draw [black strand] (-1,3) to [out=up, in=up, looseness=1] (0,3);
\draw (-2,1) to (-2,4);
\draw (-1,1) to (-1,3);
\node [halfbox] at (1.5,4) {$\varepsilon$};
\end{tz}
\gap=\gap
\begin{tz}[xscale=0.8]
\draw (2,0) node [below] {$Y$} to (2,3);
\draw (1,0) node [below] {$Y^\lor$} to [out=up, in=down] (0,1);
\draw [black strand] (0,0) node [below] {$X$} to [out=up, in=down] (1,1);
\draw [black strand] (1,1) to [out=up, in=down] (0,2);
\draw [black strand] (0,2) to [out=up, in=down] (0,3);
\draw [black strand] (0,1) to [out=up, in=down] (1,2);
\draw [black strand] (1,2) to [out=up, in=down] (1,3);
\node [halfbox] at (-2,2) {$\vartheta_X$};
\draw [black strand] (1,3) to [out=up, in=up, looseness=1] (2,3);
\draw [black strand] (-2,1) to [out=down, in=down, looseness=1] (-1,1);
\draw [black strand] (-1,3) to [out=up, in=up, looseness=1] (0,3);
\draw (-2,1) to (-2,4);
\draw (-1,1) to (-1,3);
\end{tz}
\gap=\gap
\begin{tz}[xscale=0.8]
\draw (0.5,0.5) node [below] {$X$} to (0.5,3);
\draw (1,0.5) node [below] {$Y^\lor$} to (1,3);
\draw (1.5,0.5) node [below] {$Y$} to (1.5,3);
\draw [purple strand] (1.25,3) to (1.25,4);
\draw [purple strand] (0.25,3) to (0.25,4);
\node [halfbox] at (-1,2) {$\vartheta_X$};
\node [1halfbox] at (1.25,3) {$i_Y$};
\node [1halfbox] at (0.25,3) {$i_X$};
\node [box] at (0.75,4) {$\omega$};
\draw [black strand] (-1,1) to [out=down, in=down, looseness=1] (0,1);
\draw (-1,1) to (-1,4);
\draw (0,1) to (0,3);
\end{tz}
\end{align*}

Now, by the universal property of the coend $\coend$, the above gives rise to a morphism $X\otimes\coend\to X$.

So now, seeing the full composite, we have:

\begin{align*}
    \begin{tz}[xscale=0.8]
\draw (0.5,0.5) node [below] {$X$} to (0.5,3);
\draw [purple strand] (1.25,0.5) to (1.25,4);
\draw [purple strand] (0.25,3) to (0.25,4);
\node [halfbox] at (1.25,0.5) {$\Lambda$};
\node [halfbox] at (-1,2) {$\vartheta_X$};
\node [1halfbox] at (0.25,3) {$i_X$};
\node [box] at (0.75,4) {$\omega$};
\draw [black strand] (-1,1) to [out=down, in=down, looseness=1] (0,1);
\draw (-1,1) to (-1,4);
\draw (0,1) to (0,3);
\end{tz}
\gap=\gap
    \begin{tz}[xscale=0.8]
\draw (0.5,0.5) node [below] {$X$} to (0.5,3);
\draw [purple strand] (0.25,3) to (0.25,4);
\node [halfbox] at (-1,2) {$\vartheta_X$};
\node [2halfbox] at (0.25,3) {$i_X$};
\node [2halfbox] at (0.25,4) {$\Lambda^{co}$};
\draw [black strand] (-1,1) to [out=down, in=down, looseness=1] (0,1);
\draw (-1,1) to (-1,4);
\draw (0,1) to (0,3);
\end{tz}
\end{align*}

where we used Lemma \ref{modularity_parameter} since $\Lambda^{co}\circ \Lambda = \id_\unit$.

Here we remind that it is sufficient to check the relation for $X=P_j$, a projective cover of a simple $j$, by Lemma \ref{coint_eta_eps} we have that all the above morphisms are zero, apart from:


\begin{align*}
    \begin{tz}[xscale=0.8]
\draw (0.5,0.5) node [below] {$P_\unit$} to (0.5,3);
\draw [purple strand] (0.25,3) to (0.25,4);
\node [halfbox] at (-1,2) {$\vartheta_{P_\unit}$};
\node [2halfbox] at (0.25,3) {$i_{P_\unit}$};
\node [2halfbox] at (0.25,4) {$\Lambda^{co}$};
\draw [black strand] (-1,1) to [out=down, in=down, looseness=1] (0,1);
\draw (-1,1) to (-1,4);
\draw (0,1) to (0,3);
\end{tz}
\gap=\gap
\begin{tz}[xscale=0.8]
\draw (1,0.5) node [below] {$P_\unit$} to (1,3);
\node [halfbox] at (-1,2) {$\vartheta_{P_\unit}$};
\node [halfbox] at (0,3) {$\eta_\unit^\lor$};
\node [halfbox] at (1,3) {$\varepsilon_\unit$};
\draw [black strand] (-1,1) to [out=down, in=down, looseness=1] (0,1);
\draw (-1,1) to (-1,4);
\draw (0,1) to (0,3);
\end{tz}
\gap=\gap
\begin{tz}[xscale=0.8]
\draw (1,0) node [below] {$P_\unit$} to (1,1);
\node [halfbox] at (-1,2) {$\vartheta_{P_\unit}$};
\node [halfbox] at (1,2) {$\eta_\unit$};
\node [halfbox] at (1,1) {$\varepsilon_\unit$};
\draw [black strand] (-1,1) to [out=down, in=down, looseness=1] (0,1);
\draw [black strand] (0,3) to [out=up, in=up, looseness=1] (1,3);
\draw (-1,1) to (-1,4);
\draw (0,1) to (0,3);
\draw (1,2) to (1,3);
\end{tz}
\gap=\gap
\begin{tz}[xscale=0.8]
\draw (1,0) node [below] {$P_\unit$} to (1,1);
\node [halfbox] at (1,2) {$\eta_\unit$};
\node [halfbox] at (1,1) {$\varepsilon_\unit$};
\draw (1,2) to (1,3);
\end{tz}
\end{align*}

which, by Lemma \ref{mu_mu_dag_comp} is the desired composite.

Note that we used the naturality of $\vartheta$ and the fact that $\vartheta_\unit=\id_\unit$.

As for the constants involved, the relation that they have to satisfy is: $\mathscrsfs{D}^\inv\cdot\zeta=\mathscrsfs{D}$ which gives $\mathscrsfs{D}=\sqrt{\zeta}$, fixing the constant $\mathscrsfs{D}\in k^\times$.


\begin{remark}
    Note that the rotated forms around the $z$-axis of $\eqref{piv_on_sphere}$ and $\eqref{MOD}$ hold because of the universal property of the coend $\coend$. 
\end{remark}


\item In order to check the anomaly relation \eqref{Anom}, we will make use of the fact that $\mathcal{Z}(S^2)\cong k$. Since $\mathcal{Z}(\xi)\colon \mathcal{Z}(S^2)\to \mathcal{Z}(S^2)$, then $\mathcal{Z}(\xi)$ is a scalar multiple of the identity, and abusing notation we write: $\mathcal{Z}(\xi)=\xi\cdot \id_{\mathcal{Z}(S^2)}$, where $\xi\in k$.

Now, we check the value of the left composite of \eqref{Anom} on $\mathcal{Z}(\tinycup)$, instead of $\mathcal{Z}(\begin{tikzpicture}
    \node [Cap] at (0,0) {};
    \node[Cup] at (0,0) {};
    \end{tikzpicture})$, which, from the twist non-degeneracy \eqref{twistnondegen} of $\C$ we know is:

\tikzset{every node/.style={font=\normalsize}}
\[\begin{tikzcd}
	\unit & \coend \\
	\unit & \coend
	\arrow["\mathscrsfs{D}^\inv\Lambda", from=1-1, to=1-2]
	\arrow["{\mathcal{T}}", from=1-2, to=2-2]
	\arrow["\varepsilon", from=2-2, to=2-1]
	\arrow["{\mathscrsfs{D}^\inv\cdot p_+\id_\unit}"', from=1-1, to=2-1]
\end{tikzcd}\]


Therefore, for the anomaly relation \eqref{Anom} to hold, we want to fix 
$$\xi=\mathscrsfs{D}^\inv \cdot p_+,$$ which, together with what we have computed above ($\mathscrsfs{D}^\inv=1/\sqrt{p_+p_-}$), gives $\xi=\sqrt{p_+/p_-}$.


\end{itemize}

With this, we have checked that all the necessary relations are satisfied.

\end{proof}

\begin{remark}
    The special case when $\xi=\id$ in $\Bordncsig$ gives the oriented non-compact presentation $\Bord$. 
    Having computed that $\xi=\sqrt{p_+/p_-}$ we see that we can define an oriented non-compact TQFT (i.e a functor out of $\Bord$), if and only if $p_+/p_-=1$. This is often referred to as `$\C$ being anomaly-free'.
\end{remark}

\begin{remark}\label{anom_root2}
   As mentioned in Remark \ref{anom_root1}, if we change the source 2-category in Theorem \ref{main_th} from $\Bordncsig$ to $\Bordncp$, we would instead compute $\xi=\sqrt[\leftroot{-2}\uproot{2}6]{p_+/p_-}$.
\end{remark}


\section{Mapping class group representations}\label{mcg_actions}

\subsection{Mapping class group representations from the non-compact TQFT}\label{MCG_nc}

Let $\Sigma_g$ be a (smooth, compact, oriented, closed) surface of genus $g\geq 0$. We will denote by $\text{Mod}(\Sigma_g)$ the mapping class group, i.e the group of isotopy classes of diffeomorphisms of $\Sigma_g$. It is a well known fact \cite{farb_primer_2011} that $\text{Mod}(\Sigma_g)$ is finitely generated by the so-called Dehn twists. We will denote a (right-handed) Dehn twist along a curve $\gamma$ by $T_\gamma$. In particular, there exists a specific set of generators, called the Lickorish generators, which is given by Dehn twists

$\{T_{\alpha_1},\dots,T_{\alpha_g}, T_{\beta_1},\dots, T_{\beta_{g-1}}, T_{\gamma_1},\dots, T_{\gamma_g}\},$
along the curves $\alpha_i, \beta_i, \gamma_i$, as shown in Figure \ref{fig:mcggenerators}. Note that this is an independent use of the letters $\alpha, \beta$ and $\gamma$ than the one we had in previous chapters.

\begin{figure}[H]
    \centering
    \begin{overpic}[scale = 0.7]{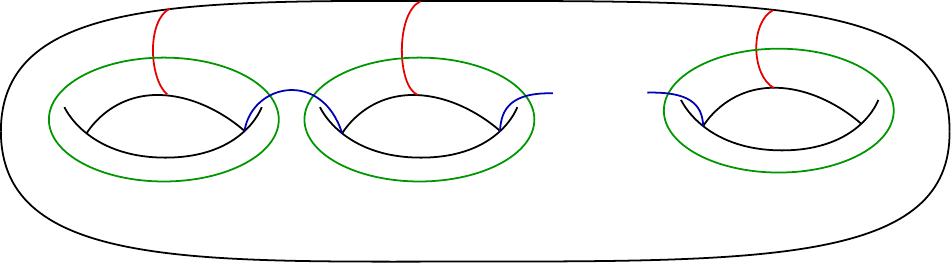}
        \put (17,27.5) {$\alpha_1$}
        \put (43,28.5) {$\alpha_2$}
        \put (81,27.5) {$\alpha_g$}
        \put (15,5.5) {$\beta_1$}
        \put (43,5.5) {$\beta_2$}
        \put (82,6) {$\beta_g$}
        \put (30,19.5) {$\gamma_1$}
        \put (54,19.5) {$\gamma_2$}
        \put (64.5,19.5) {$\gamma_{g-1}$}
        \put (61,17.5) {$\dots$}
    \end{overpic}
    \caption{Lickorish generators of $\text{Mod}_g$. Taken from \cite{romaidis_mapping_2023}.}
    \label{fig:mcggenerators}
\end{figure}
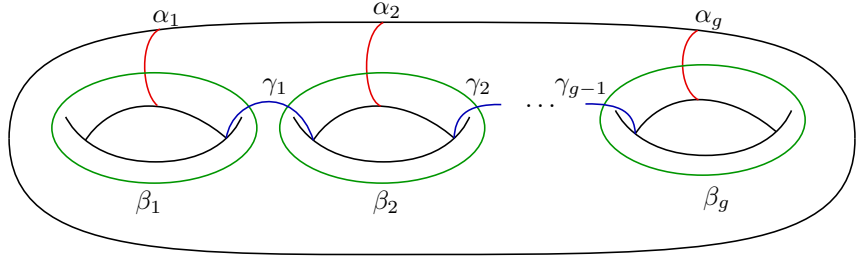

We will use an alternative set of generators for $\text{Mod}(\Sigma_g)$ by replacing $T_{\beta_i}$ by the \textbf{inverse} of the composite $S_i:=T_{\alpha_i}\circ T_{\beta_i}\circ T_{\alpha_i}$.

In $\Bordncsig$, $\Sigma_g$ should correspond to:

\normalbordisms
$$\begin{tikzpicture}
        \node[Cap, top] (A) at (0,0) {};
        \node[Pants, anchor=belt] (B) at (A.center) {};
        \node[Copants, anchor=leftleg] (C) at (B.leftleg) {};
        \node[draw=none] (G) at (0,-2) {$\vdots$};
        \node[Pants, top] (D) at (0,-3) {};
        \node[Copants, anchor=leftleg] (E) at (D.leftleg) {};
        \node[Cup] (F) at (E.belt) {};
\end{tikzpicture}$$


For simplicity, we also denote this 1-morphism by $\Sigma_g$. 
This way, $\mathcal{Z}(\Sigma_g)\cong\Hom(\coend^{\otimes g},\unit)^*$.

By $\Sigma_g^*$ we denote the 1-morphism as above, but without the $\tinycap$ applied at the top (corresponding to the once punctured genus $g$ surface). Then $\mathcal{Z}(\Sigma_g^*)=\coend^{\otimes g}$.

Now, we define the following composites in $\Bordncsig$:
\scalecobordisms{0.6}
\tikzset{every node/.style={font=\normalsize}}

\begin{align} \label{defn_of_I}
\I \quad&:=\quad
        \begin{tz}[xscale=2, yscale=4]
\node (1) at (0,0)
{
$\begin{tikzpicture}
        \node[Pants, top] (A) at (0,0) {};
        \node[Copants, anchor=leftleg] (B) at (A.leftleg) {};
        \selectpart[green, inner sep=1pt] {(A-rightleg)};
\end{tikzpicture}$
};
\node (2) at (1,0)
{
$\begin{tikzpicture}
        \node[Pants, top] (A) at (0,0) {};
        \node[Copants, anchor=leftleg] (B) at (A.leftleg) {};
\end{tikzpicture}$
};
\begin{scope}[double arrow scope]
    \draw (1) -- node[above] {$\theta$} (2);
\end{scope}
\end{tz}
\\
\label{defn_of_II}
\II \quad&:=\quad
\begin{aligned}  \begin{tikzpicture}[xscale=2, yscale=4]
\node (1) at (0,0)
{
$\begin{tikzpicture}
        \node[Copants, top] (A) at (0,0) {};
        \node[Pants, anchor=belt] (B) at (A.belt) {};
\end{tikzpicture}$
};
\node (2) at (1,0)
{
$\begin{tikzpicture}
        \node[Pants, top] (A) at (0,0) {};
        \node[Cyl, anchor=top] (B) at (A.leftleg) {};
        \node[Copants, anchor=leftleg] (C) at (A.rightleg) {};
        \node[Cyl, anchor=bottom, top] (D) at (C.rightleg) {};
        \selectpart[green, inner sep=1pt] {(A-rightleg)};
\end{tikzpicture}$
};
\node (3) at (2,0)
{
$\begin{tikzpicture}
        \node[Pants, top] (A) at (0,0) {};
        \node[Cyl, anchor=top] (B) at (A.leftleg) {};
        \node[Copants, anchor=leftleg] (C) at (A.rightleg) {};
        \node[Cyl, anchor=bottom, top] (D) at (C.rightleg) {};
\end{tikzpicture}$
};
\node (4) at (3,0)
{
$\begin{tikzpicture}
        \node[Copants, top] (A) at (0,0) {};
        \node[Pants, anchor=belt] (B) at (A.belt) {};
\end{tikzpicture}$
};
\begin{scope}[double arrow scope]
    \draw (1) --  node[above]{$\phileft^\inv$} (2);
    \draw (2) --  node[above]{$\theta$} (3);
    \draw (3) --  node[above]{$\phileft$} (4);
\end{scope} \end{tikzpicture} \end{aligned}
\\
\label{defn_of_A}
A^\dagger \quad&:=\quad
\begin{tz}[xscale=2, yscale=4]
\node (1) at (0,0)
{
$\begin{tikzpicture}
        \node[Pants, top] (A) at (0,0) {};
        \node[Copants, anchor=leftleg] (B) at (A.leftleg) {};
        \selectpart[green, inner sep=1pt] {(B-belt)};
\end{tikzpicture}$
};
\node [inner sep=0pt] (2) at (1,0)
{
$\begin{tikzpicture}
        \node[Pants, top] (A) at (0,0) {};
        \node[Copants, anchor=leftleg] (B) at (A.leftleg) {};
        \node[Pants, anchor=belt] (C) at (B.belt) {};
        \node[Copants, anchor=leftleg] (D) at (C.leftleg) {};
    \selectpart[green] {(A-leftleg) (A-rightleg) (C-leftleg) (C-rightleg)};
\end{tikzpicture}$
};
\node [inner sep=0pt] (3) at (2,0)
{
$\begin{tikzpicture}
        \node[Pants, top] (A) at (0,0) {};
        \node[Copants, anchor=leftleg] (B) at (A.leftleg) {};
        \node[Pants, anchor=belt] (C) at (B.belt) {};
        \node[Copants, anchor=leftleg] (D) at (C.leftleg) {};
        \selectpart[green] {(A-belt) (A-leftleg) (A-rightleg) (B-belt)};
\end{tikzpicture}$
};
\node (4) at (3,0)
{
$\begin{tikzpicture}
        \node[Pants, top] (A) at (0,0) {};
        \node[Copants, anchor=leftleg] (B) at (A.leftleg) {};
\end{tikzpicture}$
};
\begin{scope}[double arrow scope]
    \draw (1) -- node[above] {$\epsilon^\dagger$} (2);
    \draw (2) -- node[above] {$\II$} (3);
    \draw (3) -- node[above] {$\epsilon$} (4);
\end{scope}
\end{tz}
\\
\label{defn_of_III}
\III^\inv \quad&:=\quad
\begin{tz}[xscale=2, yscale=4]
\node (1) at (0,0)
{
$\begin{tikzpicture}
        \node[Pants, top] (A) at (0,0) {};
        \node[Copants, anchor=leftleg] (B) at (A.leftleg) {};
        \selectpart[green, inner sep=1pt] {(A-rightleg)};
\end{tikzpicture}$
};
\node (2) at (1,0)
{
$\begin{tikzpicture}
        \node[Pants, top] (A) at (0,0) {};
        \node[Copants, anchor=leftleg] (B) at (A.leftleg) {};
\end{tikzpicture}$
};
\node (3) at (2,0)
{
$\begin{tikzpicture}
        \node[Pants, top] (A) at (0,0) {};
        \node[Copants, anchor=leftleg] (B) at (A.leftleg) {};
        \selectpart[green, inner sep=1pt] {(A-rightleg)};
\end{tikzpicture}$
};
\node (4) at (3,0)
{
$\begin{tikzpicture}
        \node[Pants, top] (A) at (0,0) {};
        \node[Copants, anchor=leftleg] (B) at (A.leftleg) {};
\end{tikzpicture}$
};
\begin{scope}[double arrow scope]
    \draw (1) -- node[above] {$\theta^\inv$} (2);
    \draw (2) -- node[above] {$A^\dagger$} (3);
    \draw (3) -- node[above] {$\theta^\inv$} (4);
\end{scope}
\end{tz}
\end{align}

By $\I_i, \II_j, A^\dagger_i, \III^\inv_i \colon\Sigma_g\to\Sigma_g$ we denote the 2-morphisms that are obtained by applying the above on the $i$-th `hole', or between the $j$'th and ($j+1$)'th `holes' of $\Sigma_g$. As explained in \cite{bartlett_extended_2014}, these composites are invertible, and they are meant to correspond to the mapping cylinders for the generators of the mapping class group $\{T_{\alpha_i}, T_{\gamma_j}, T_{\beta_i}^\inv, S_i^\inv\}$ respectively.

The value of any of the $\mathcal{Z}(\I_i), \mathcal{Z}(\II_i), \mathcal{Z}(\III_i^\inv)\colon\mathcal{Z}(\Sigma_g^*) \to \mathcal{Z}(\Sigma_g^*)$, is really some invertible morphism $x\colon \coend^g\to\coend^g$. Then, if $f\in\Hom(\coend^g,\unit)^*$, the (left) action on the state space $\mathcal{Z}(\Sigma_g)=\Hom(\coend^g,\unit)^*$ is given by $f(-\circ x)$.

Now, we define the following unique morphisms in $\C$:

\begin{itemize}
    \item $\Omega\colon\coend\otimes\coend\to\coend\otimes\coend$, satisfying for all $X,Y\in \C$ the identity: 
    $$\Omega\circ(i_X\otimes i_Y)= (i_X\otimes i_Y)\circ(\id_{X^\lor}\otimes(c_{Y^\lor,X}\circ c_{X,Y^\lor})\otimes \id_Y). $$
    \item $\mathcal{S}\colon\coend\to\coend$ satisfying for all $X\in \C$ the identity: 
    $$\mathcal{S}\circ i_X= (\varepsilon\otimes \id_\coend)\circ \Omega\circ(\id_X\otimes \Lambda).$$
\end{itemize}

Additionally, we denote by $\mathcal{H}\colon\coend\otimes\coend\to\coend\otimes\coend$ the morphism $$\mathcal{H}:= \Omega\circ(\mathcal{T}\otimes\mathcal{T}).$$



\begin{lemma}\label{action_MCG_gens}
    
    Computing, we have that:

    \begin{align}
        \mathcal{Z}(\I_i)(f)&= 
        f(-\circ\id_{\coend^{i-1}}\otimes\mathcal{T}\otimes\id_{\coend^{g-i}}) \\
        \mathcal{Z}(\II_i)(f)&=
        f(-\circ\id_{\coend^{j-1}}\otimes\mathcal{H}\otimes\id_{\coend^{g-j-1}}) \\
        \mathcal{Z}(\III_i^\inv)(f)&=1/\sqrt{p_+p_-}\cdot
        f(-\circ\id_{\coend^{i-1}}\otimes\mathcal{S}\otimes\id_{\coend^{g-i}})
    \end{align}

\end{lemma}

\begin{proof}

We compute the values of the 2-morphisms in \cref{defn_of_I,defn_of_II,defn_of_A,defn_of_III} after applying the constructed TQFT $\mathcal{Z}$. This will be sufficient, as the result is independent of the incoming object of the functor $\mathcal{Z}(\tinycopants)$.

\begin{itemize}
    \item  It is easy to see that:

    $$\mathcal{Z}(\I)\colon X\otimes\coend\xrightarrow{\id_X\otimes\mathcal{T}}X\otimes\coend.$$

We have already used this when checking the relations in \cref{MOD,Anom}.
    
    \item We will compute $\mathcal{Z}(\II)$ on: 
    $$\begin{tikzpicture}
        \node[Pants, top] (A) at (0,0) {};
        \node[Copants, anchor=leftleg] (B) at (A.leftleg) {};
        \node[Pants, anchor=belt] (C) at (B.belt) {};
        \node[Copants, anchor=leftleg] (D) at (C.leftleg) {};
    \selectpart[green] {(A-leftleg) (A-rightleg) (C-leftleg) (C-rightleg)};
\end{tikzpicture}$$

So abusing notation, $\mathcal{Z}(\II)\colon Z\otimes\coend\otimes \coend \to Z\otimes\coend\otimes\coend$.

\tikzset{every node/.style={font=\tiny}}

The image of the above under $\mathcal{Z}$ is a functor $\C\to\C$. If the input is some object $Z\in \C$ and if by $X\in\C$ and $Y\in\C$ we respectively denote the objects over which we take the first and second coend coming from the two $\mathcal{Z}(\tinycopants)$, we find:

\begin{align*}
\begin{tz}[xscale=0.8]
\draw (-2,0) node [below] {$Z$} to (-2,4.5);
\draw (-1,0) node [below] {$X^\lor$} to (-1,3.5);
\draw (2,0) node [below] {$Y$} to (2,3.5);
\draw (0,0) node [below] {$X$} to (0,3.5);
\draw (1,0) node [below] {$Y^\lor$} to (1,3.5);
\node [box] at (0.5,2) {$\vartheta_{X\otimes Y^\lor}$};
\node [box] at (1.5,3.5) {$i_Y$};
\node [box] at (-0.5,3.5) {$i_X$};
\draw [purple strand] (-0.5,3.5) to (-0.5,4.5);
\draw [purple strand] (1.5,3.5) to (1.5,4.5);
\end{tz}
\gap=\gap
\begin{tz}[xscale=0.8]
\draw (-2,0) node [below] {$Z$} to (-2,4.5);
\draw (-1,0) node [below] {$X^\lor$} to (-1,3.5);
\draw (2,0) node [below] {$Y$} to (2,3.5);
\draw (0,0) node [below] {$X$} to (0,1);
\draw (1,0) node [below] {$Y^\lor$} to (1,1);
\draw (1,1) to [out=up, in=down] (0,2);
\draw [black strand] (0,1) to [out=up, in=down] (1,2);
\draw [black strand] (1,2) to [out=up, in=down] (0,3);
\draw [black strand] (0,3) to [out=up, in=down] (0,3.5);
\draw [black strand] (0,2) to [out=up, in=down] (1,3);
\draw [black strand] (1,3) to [out=up, in=down] (1,3.5);
\node [halfbox] at (0,0.5) {$\vartheta_X$};
\node [halfbox] at (1,0.5) {$\vartheta_{Y^\lor}$};
\node [box] at (1.5,3.5) {$i_Y$};
\node [box] at (-0.5,3.5) {$i_X$};
\draw [purple strand] (-0.5,3.5) to (-0.5,4.5);
\draw [purple strand] (1.5,3.5) to (1.5,4.5);
\end{tz}
\gap=\gap
\begin{tz}[xscale=0.8]
\draw (-2,0) node [below] {$Z$} to (-2,4.5);
\draw (-1,0) node [below] {$X^\lor$} to (-1,2.5);
\draw (2,0) node [below] {$Y$} to (2,2.5);
\draw (0,0) node [below] {$X$} to (0,2.5);
\draw (1,0) node [below] {$Y^\lor$} to (1,2.5);
\node [halfbox] at (0,1) {$\vartheta_X$};
\node [halfbox] at (1,1) {$\vartheta_{Y^\lor}$};
\node [box] at (1.5,2.5) {$i_Y$};
\node [box] at (-0.5,2.5) {$i_X$};
\node [box] at (0.5,3.5) {$\Omega$};
\draw [purple strand] (-0.5,2.5) to [out=up, in=down] (0,3.5);
\draw [purple strand] (1.5,2.5) to [out=up, in=down] (1,3.5);
\draw [purple strand] (0,3.5) to (0,4.5);
\draw [purple strand] (1,3.5) to (1,4.5);
\end{tz}
\end{align*}

\begin{align*}
\begin{tz}[xscale=0.8]
\draw (-2,0) node [below] {$Z$} to (-2,4.5);
\draw (-1,0) node [below] {$X^\lor$} to (-1,1);
\draw (2,0) node [below] {$Y$} to (2,1);
\draw (0,0) node [below] {$X$} to (0,1);
\draw (1,0) node [below] {$Y^\lor$} to (1,1);
\node [halfbox] at (-0.5,2) {$\mathcal{T}$};
\node [halfbox] at (1.5,2) {$\mathcal{T}$};
\node [box] at (1.5,1) {$i_Y$};
\node [box] at (-0.5,1) {$i_X$};
\node [box] at (0.5,3.5) {$\Omega$};
\draw [purple strand] (-0.5,1) to (-0.5,2.5);
\draw [purple strand] (1.5,1) to (1.5,2.5);
\draw [purple strand] (-0.5,2.5) to [out=up, in=down] (0,3.5);
\draw [purple strand] (1.5,2.5) to [out=up, in=down] (1,3.5);
\draw [purple strand] (0,3.5) to (0,4.5);
\draw [purple strand] (1,3.5) to (1,4.5);
\end{tz}
\end{align*}

Where in the first step, we used the property of the twist of a tensor product, and in the third we used dinaturality of the coend, as well as the property $\vartheta_{Y^\lor}=\vartheta^\lor_Y$.

\item Now we check the value of $\mathcal{Z}(A^\dagger)\colon Z\otimes\coend\to Z\otimes\coend$. This results to the following:


\begin{align*}
{\color{red} 1/\sqrt{p_+p_-}}\,
\begin{tz}[xscale=0.8]
\draw (-2,0) node [below] {$Z$} to (-2,4.5);
\draw (2,0) node [below] {$Y$} to (2,1);
\draw (1,0) node [below] {$Y^\lor$} to (1,1);
\node [halfbox] at (-0.5,0) {$\Lambda$};
\node [halfbox] at (-0.5,2) {$\mathcal{T}$};
\node [halfbox] at (1.5,2) {$\mathcal{T}$};
\node [box] at (1.5,1) {$i_Y$};
\node [box] at (0.5,3.5) {$\Omega$};
\node [halfbox] at (1,4.5) {$\varepsilon$};
\draw [purple strand] (-0.5,0) to (-0.5,2.5);
\draw [purple strand] (1.5,1) to (1.5,2.5);
\draw [purple strand] (-0.5,2.5) to [out=up, in=down] (0,3.5);
\draw [purple strand] (1.5,2.5) to [out=up, in=down] (1,3.5);
\draw [purple strand] (0,3.5) to (0,4.5);
\draw [purple strand] (1,3.5) to (1,4.5);
\end{tz}
\end{align*}

\item Using the above, as well as the equivalent descriptions of $\mathcal{H}$, we can compute the value of $\mathcal{Z}(\III^\inv)\colon Z\otimes\coend\to Z\otimes\coend$:

\begin{align*}
\begin{tz}[xscale=0.8]
\draw (-2,0) node [below] {$Z$} to (-2,4.75);
\node [halfbox] at (-0.5,0) {$\mathscrsfs{D}^\inv\Lambda$};
\node [halfbox] at (1.5,1) {$\mathcal{T}^\inv$};
\node [halfbox] at (-0.5,2) {$\mathcal{T}$};
\node [halfbox] at (1.5,2) {$\mathcal{T}$};
\node [box] at (0.5,3.5) {$\Omega$};
\node [halfbox] at (0,4.25) {$\mathcal{T}^\inv$};
\node [halfbox] at (1,4.5) {$\varepsilon$};
\draw [purple strand] (-0.5,0) to (-0.5,2.5);
\draw [purple strand] (1.5,0) node [below] {$\coend$} to (1.5,2.5);
\draw [purple strand] (-0.5,2.5) to [out=up, in=down] (0,3.5);
\draw [purple strand] (1.5,2.5) to [out=up, in=down] (1,3.5);
\draw [purple strand] (0,3.5) to (0,4.75);
\draw [purple strand] (1,3.5) to (1,4.5);
\end{tz}
\gap=\gap
\begin{tz}[xscale=0.8]
\draw (-2,0) node [below] {$Z$} to (-2,4.75);
\node [halfbox] at (-0.5,0) {$\mathscrsfs{D}^\inv\Lambda$};
\node [halfbox] at (1.5,1) {$\mathcal{T}^\inv$};
\node [halfbox] at (-0.5,1) {$\mathcal{T}$};
\node [halfbox] at (-0.5,2) {$\mathcal{T}^\inv$};
\node [halfbox] at (1.5,2) {$\mathcal{T}$};
\node [box] at (0.5,3.5) {$\Omega$};
\node [halfbox] at (1,4.5) {$\varepsilon$};
\draw [purple strand] (-0.5,0) to (-0.5,2.5);
\draw [purple strand] (1.5,0) node [below] {$\coend$} to (1.5,2.5);
\draw [purple strand] (-0.5,2.5) to [out=up, in=down] (0,3.5);
\draw [purple strand] (1.5,2.5) to [out=up, in=down] (1,3.5);
\draw [purple strand] (0,3.5) to (0,4.75);
\draw [purple strand] (1,3.5) to (1,4.5);
\end{tz}
\gap=\gap
{\color{red} 1/\sqrt{p_+p_-}}\,
\begin{tz}[xscale=0.8]
\draw (-2,0) node [below] {$Z$} to (-2,3.25);
\node [halfbox] at (-0.5,0) {$\Lambda$};
\node [box] at (0.5,2) {$\Omega$};
\node [halfbox] at (1,3) {$\varepsilon$};
\draw [purple strand] (-0.5,0) to (-0.5,1);
\draw [purple strand] (1.5,0) node [below] {$\coend$} to (1.5,1);
\draw [purple strand] (-0.5,1) to [out=up, in=down] (0,2);
\draw [purple strand] (1.5,1) to [out=up, in=down] (1,2);
\draw [purple strand] (0,2) to (0,3.25);
\draw [purple strand] (1,2) to (1,3);
\end{tz}
\end{align*}


The first equation follows from the naturality of the twist $\vartheta$ and braiding $c$.

\end{itemize}

\end{proof}


    



We could pause here and observe that the same morphisms appear in the projective mapping class group action of \cite{de_renzi_mapping_2023} (see section \ref{rel_to_lyub_MCG}). In order to form a mathematical statement out of this observation, we are forced to assume the validity of Conjecture \ref{conjecture}. Then, thanks to Theorem \ref{main_th}, we know that the non-compact TQFT $\Znc$ gives rise to a representation $\tilde{\rho}_{nc}$ of a central extension of the mapping class group. Equivalently, we obtain a projective representation $$\bar{\rho}_{nc}\colon\text{Mod}(\Sigma_g)\to \text{PGL}_k(\mathcal{Z}(\Sigma_g)).$$

Using Lemma \ref{action_MCG_gens} we have:

\begin{prop}
The projective mapping class group representation $$\bar{\rho}_{nc}\colon\text{Mod}(\Sigma_g)\to \text{PGL}_k(\mathcal{Z}(\Sigma_g))$$ satisfies:

\begin{align}
        \bar{\rho}_{nc}(T_{\alpha_i})(f)&= 
        f(-\circ\id_{\coend^{i-1}}\otimes\mathcal{T}\otimes\id_{\coend^{g-i}}) \\
        \bar{\rho}_{nc}(T_{\gamma_j})(f)&=
        f(-\circ\id_{\coend^{j-1}}\otimes\mathcal{H}\otimes\id_{\coend^{g-j-1}}) \\
        \bar{\rho}_{nc}(S_i^\inv)(f)&=
        f(-\circ\id_{\coend^{i-1}}\otimes\mathcal{S}\otimes\id_{\coend^{g-i}})
    \end{align}

\end{prop}

\subsection{Relation to the Lyubashenko projective representations}\label{rel_to_lyub_MCG}

The projective mapping class group representations $\bar{\rho}_X$ coming from the non-semisimple 3d TQFTs constructed in \cite{de_renzi_3-dimensional_2022} are computed in \cite{de_renzi_mapping_2023}. In addition, it is shown that they are equivalent to the family of projective representations $\bar{\rho}_L$ constructed by Lyubashenko in \cite{lyubashenko_invariants_1995}.

We will briefly review the construction of $\bar{\rho}_X$ for a surface without punctures, and proceed to show that the projective representation $\bar{\rho}_{nc}$ from Section \ref{MCG_nc} is dual to $\bar{\rho}_X$.

The vector space they consider is $X^\prime_g:=\Hom_\C(\coend^g,1)$. 

Let $x^\prime\in \Hom_\C(\coend^g,1)$, the group homomorphism $\bar{\rho}_X\colon\text{Mod}(\Sigma_g)\to \text{PGL}_k((X^\prime_g))$ satisfies:

\begin{align}
        \rho_X(T_{\alpha_i})(f)&= 
        x^\prime\circ(\id_{\coend^{i-1}}\otimes\mathcal{T}^\inv\otimes\id_{\coend^{g-i}}) \\
        \rho_X(T_{\gamma_j})(f)&=
        x^\prime\circ(\id_{\coend^{i-1}}\otimes\mathcal{H}^\inv\otimes\id_{\coend^{g-j-1}}) \\
        \rho_X(S_i^\inv)(f)&=
        x^\prime\circ(\id_{\coend^{i-1}}\otimes\mathcal{S}^\inv\otimes\id_{\coend^{g-i}})
    \end{align}

This is a group homomorphism as they consider inverse mapping cylinders \cite[Eq (17)]{de_renzi_mapping_2023}.





Note that the inverses appear because their choice of generators of $\text{Mod}(\Sigma_g)$ is inverse to ours.

We therefore get that:

\begin{theorem}\label{mapping-class-th}
    By definition, we have that $\bar{\rho}_{nc}=\bar{\rho}_X^*$.
\end{theorem}

\newpage

\section{From Lincat to Bimod}\label{Lincat-Bimod}

In order to tackle the questions concerning 3-manifold invariants (Section \ref{3-mfld-inv}) and modified traces (Section \ref{mod-trace}), we will adapt the non-compact TQFT \ref{main_th} to a different target, the bicategory Bimod.
This is due to the fact that, in order to compute the values of $\Znc$ on 3-manifolds, it is useful to have a `skein-theoretic description'. Working in Bimod provides such a setting.

\begin{definition}[Bimod]
    The symmetric monoidal bicategory Bimod is defined as follows:
    \begin{itemize}
        \item Its objects are linear categories;
        \item Its 1-morphisms are bimodules $F\colon\C\proarrow\mathcal{D}$, defined as linear functors $F\colon \mathcal{D}^{op}\times \C\to \text{Vect}_k$;
        \item Its 2-morphisms are natural transformations between the associated linear functors.
    \end{itemize}
\end{definition}

Composition of bimodules $F\colon \C\proarrow\mathcal{D}$ and $G\colon \mathcal{D}\proarrow \mathcal{E}$ is defined as:
$$G\circ F(e,c):= \int^{d\in \mathcal{D}} G(e,d)\otimes F(d,c).$$
The identity 1-morphisms are given by the Hom functor: $$\Hom(-,-)\colon \C^{op}\times \C\to \text{Vect}_k.$$


For a bimodule $F\colon\mathcal{D}_1\proarrow\mathcal{D}_2$, we will denote its value on $X\in \mathcal{D}_1$ and $Y\in\mathcal{D}_2$ by $F_X^Y$.

\begin{definition}\label{associated_bimod}
 Given a linear functor $F\colon \mathcal{D}_1\to\mathcal{D}_2$, we can construct an \textit{associated bimodule} 
 $F_*\colon \C\proarrow\mathcal{D}$ as:
 $F_*:=\Hom_\mathcal{D}(-,F(-))$.

Given a natural transformation $\omega\colon F\xRightarrow{} G$ between linear functors ${F,G\colon\mathcal{D}_1\to\mathcal{D}_2}$, there is an induced natural transformation between associated bimodules: 

${\omega_*\colon F_*\to G_*}$, given by postcomposition with $\omega$.
\end{definition}

In fact, the above is organized in a 2-functor $(-)^{fg,proj}$ from Lincat to Bimod, that sends a small linear category that admits finite colimits to its subcategory of compact projective objects. In general, a right exact functor does not preserve projective objects, but always gives rise to a bimodule.

Postcomposing $\Znc$ with this 2-functor we obtain a TQFT in Bimod. We will denote this TQFT by $\Znc^\prime$.


The values of the Bimod TQFT on object and 1-morphism generators are:

\begin{itemize}
    \item $\Znc^\prime(S^1) = \cP$
    \item $\Znc^\prime(\tinypants)_{B,C}^A= \Hom_\C(A,B\otimes C)$
    \item $\Znc^\prime(\tinycopants)_C^{A,B}= \Hom_\C(A\otimes B, C)$
    \item $\Znc^\prime(\tinycap)_A= \Hom_\C(A,1)^*$
    \item $\Znc^\prime(\tinycup)^A= \Hom_\C(A,1)$
\end{itemize}

\begin{remark}
    Strictly speaking, the values on $\tinycopants$ and $\tinycap$ are equivalent to the ones above.
In particular, directly computing the associated bimodules, we have: 
\begin{itemize}
    \item $(T^R)_*= \Hom_{\C\boxtimes\C}(-\boxtimes -,\displaystyle\int^{P\in\mathcal{P}}\hspace{-20pt}-\otimes P^\lor\boxtimes P)$
    \item $(U^L)_*= \Hom_{\Vect}(k,\Hom_\C(-,1)^*)$
\end{itemize}
\end{remark}

As explained in \cite{bartlett_modular_2015}, the bimodule vector space is the same as the vector space of string diagrams drawn internal to the volume of that particular surface embedded in $\mathbb{R}^3$. This can be thought of as the skein module of the 3-manifold bound by the surface. We are interested in this description as it facilitates computations.


\normalbordisms
\begin{calign}
\label{eq:embeddingR3}
\begin{tz}
    \node[Pants, bot, top, height scale=1.0] (A) at (0,0) {};
    \begin{scope}[internal string scope]
        \node (i) at ([yshift=\toff] A.belt) [above] {$A$};
        \node (j) at ([yshift=-\boff] A.leftleg) [below] {$B$};
        \node (k) at ([yshift=-\boff] A.rightleg) [below] {$C$};
        \node [stiny label] (g) at (0,0.02\cobheight) {$f$};
        \draw (i.south)
            to (g.center)
            to [out=-140, in=up] (j.north);
        \draw (g.center)
            to [out=-40, in=up] (k.north);
    \end{scope}
\end{tz}
&
\begin{tz}
    \node[Copants, bot, top, height scale=1.0] (A) at (0,0) {};
    \begin{scope}[internal string scope]
        \node (i) at ([yshift=-\boff] A.belt) [below] {$C$};
        \node (j) at ([yshift=\toff] A.leftleg) [above] {$A$};
        \node (k) at ([yshift=\toff] A.rightleg) [above] {$B$};
        \node [stiny label] (g) at (0,-0.1\cobheight) {$g$};
        \draw (i.north)
            to (g.center)
            to [out=140, in=down] (j.south);
        \draw (g.center)
            to [out=40, in=down] (k.south);
    \end{scope}
\end{tz}
&
\begin{aligned}
\begin{tikzpicture}
\setlength\cupheight{1.5\cupheight}
\node (i) at (0,0) [Cup, top] {};
\node (j) at (0,-\cobheight) [Bot3D, invisible] {};
\node (g) [stiny label] at (0,-0.4\cobheight) {$h$};
\draw [internal string] (g.center) to ([yshift=\toff] i.center) node [above, red] {$A$};
\node at ([yshift=-\boff] j.center) [below, white] {$A$};
\end{tikzpicture}
\end{aligned}
\\*
\nonumber
f \in \Znc^\prime \big( \tikztinypants \big) {}^A _{B,C}
&
g \in \Znc^\prime \big( \tikztinycopants \big) {}^{A,B} _{C}
&
h \in \Znc^\prime \big( \tikztinycup \big) {}^A
\end{calign}

\begin{remark}
    Morphisms in the interior are drawn in the opposite direction as that of the cobordism. As of now, we adopt the convention that red string diagrams are to be read from top to bottom and from left to right.
\end{remark}

\begin{remark}\label{cap-string-diag}
    Having defined $\Znc^\prime(\tinycap)_A= \Hom_\C(A,1)^*$, it appears that we do not have a string diagrammatic description of its elements. However, since $A\in\cP$, we can use the natural isomorphism $\Omega^1\colon\Hom_\C(A,1)^*\xrightarrow{\sim}\Hom_\C(A,1)$ defined in \ref{omega1def}, and work with string diagrams this way.
    
    \begin{gather}
        \begin{tikzpicture}
\setlength\cupheight{1.5\cupheight}
\node (i) at (0,0) [Cap, top] {};
\node (j) at (0,\cobheight) [Bot3D, invisible] {};
\node (g) [stiny label] at (0,0.4\cobheight) {$j$};
\draw [internal string] (g.center) to ([yshift=-\boff] i.center) node [below, red] {$A$};
\node at ([yshift=\toff] j.center) [above, white] {$A$};
\end{tikzpicture}
\\
\nonumber
 j \in \Omega^1\left(\Znc^\prime \big( \tikztinycap \big) {}_A\right)
    \end{gather}
    
    We will implicitly do this in later chapters when working with closed surfaces. It is important to note, however, that this is not necessary, but only for the sake of simplicity. We could instead work with a surface with one boundary component and the resulting (internal string) morphism $A\to 1$, which looks similar to the `admissibility condition' appearing in \cite{de_renzi_3-dimensional_2022}. This does not affect our analysis in later chapters, as we can `localize' this morphism in a cylinder near the cap $\tinycap$ (via the coend relation), and then evaluate it using $\Znc^\prime(\tinycap)_A= \Hom_\C(A,1)^*$.
\end{remark}

Having the string diagram notation makes computations easier in many cases. This is why, for the rest of this section, we express the actions of 2-morphism generators in Bimod, and whenever applicable, describe them by how they act on internal string diagrams. We will only focus on the ones that appear in sections \ref{3-mfld-inv} and \ref{mod-trace}.

We have two ways to describe what $\Znc^\prime$ assigns to a surface in Bimod. 

\begin{enumerate}
    \item Use the value in Lincat to obtain the associated bimodule;
    \item Decompose the given surface into a composite of the generators above, and obtain a composite in Bimod.
\end{enumerate}

The string diagrammatic approach is associated with the latter. Having defined $\Znc$ in Lincat however, we automatically get an action of 2-morphism generators on the associated bimodules. To connect the two, we use Definition \ref{associated_bimod}, as well as the fact that the two descriptions are related by (potentially several) applications of the co-Yoneda lemma.

First, define for $X,Y,Z\in \C$ the natural isomorphisms

\begin{align}
    (-)^\flat\colon \Hom_\C(X\otimes Y,Z)&\xrightarrow{\sim}\Hom_\C(X, Z\otimes Y^\lor) \label{flat} \\
    (-)^\sharp\colon \Hom_\C(Y,X\otimes Z)&\xrightarrow{\sim}\Hom_\C(X^\lor\otimes Y,Z) \label{sharp} \\
    (-)^\natural\colon \Hom_\C(X\otimes Y,Z)&\xrightarrow{\sim}\Hom_\C(Y,X^\lor\otimes Z) \label{naturalm}
\end{align}

by 

$$\!\!\!\!\!\!f^\flat=(f\otimes\id_{Y^\lor})\circ(\id_X\otimes\text{coev}_Y), 
\quad 
g^\sharp=(\text{ev}_X\otimes\id_{Z})\circ(\id_{X^\lor}\otimes g),
\quad
h^\natural=(\id_{X^\lor}\otimes h)\circ(\text{coev}^R_X\otimes\id_Y).$$

Throughout this section, fix $P,Q,P_1,P_2,Q_1,Q_2\in \mathcal{P}$.


\begin{prop}\label{eps-eta_bimod}
    $\Znc^\prime(\epsilon)$ and $\Znc^\prime(\eta)$ act as follows on internal string diagrams:

    \begin{align}
    \begin{tz}
        \node[Pants, bot, top] (A) at (0,0) {};
        \node[Copants, bot, anchor=leftleg] (B) at (A.leftleg) {};
        \begin{scope}[internal string scope]
                \node (i) at ([yshift=\toff] A.belt) [above] {};
                \node (i2) at ([yshift=-\boff] B.belt) [below] {};
                \node[stiny label] (p) at ([yshift=0.05\cobheight] A.center) {$f$};
                \node[stiny label] (q) at ([yshift=-0.1\cobheight] B.center) {$g$};
                \draw (i.south)
                    to (p.center)
                    to [out=-120, in=up]
                        (A.leftleg)
                        to[out=down, in=120] (q.center);
                \draw (p.center) to[out=-60, in=up]
                (A.rightleg) to[out=down, in=60] (q.center);
                \draw (q.center) -- (i2.north);
        \end{scope}
\end{tz}
\,\,\xmapsto{\Znc^\prime (\epsilon)}\,\,
\begin{tz}
        \node[Cyl, top, bot=false] (A) at (0,0) {};
        \node[Cyl, bot] (A2) at (0,-\cobheight) {};
        \begin{scope}[internal string scope]
                \node (i) at ([yshift=\toff] A.top) [above] {};
                \node (i2) at ([yshift=-\boff] A2.bottom) [below] {};
                \node[stiny label] (p) at (A.center) {$f$};
                \node[stiny label] (q) at (A2.center) {$g$};
                \draw (i.south) -- (p.center) to[out=-120, in=120] node[left, xshift=0.1cm] {} (q.center);
                \draw (p.center) to[out=-60, in=60] node[right, xshift=-0.1cm] {} (q.center);
                \draw (q.center) -- (i2.north);
        \end{scope}
\end{tz}
&\qquad\qquad
\begin{tz}
        \node[Cyl, tall, top, bot] (A) at (0,0) {};
        \node[Cyl, tall, top, bot] (B) at (\cobgap + \cobwidth, 0) {};
        \begin{scope}[internal string scope]
                \draw ([yshift=\toff] A.top) node[above]{} -- ([yshift=-\boff] A.bot) node[below]{};
                \draw ([yshift=\toff] B.top) node[above]{} -- ([yshift=-\boff] B.bot) node[below]{};
        \end{scope}
\end{tz}
\,\,\xmapsto{\Znc^\prime (\eta)}
\begin{tz}
        \node[Pants, bot, belt scale=1.5] (A) at (0,0) {};
        \node[Copants, bot, anchor=belt, belt scale=1.5, top] (B) at (A.belt) {}; 
        \begin{scope}[internal string scope]
                \node (i) at ([yshift=\toff] B.leftleg) [above] {};
                \node (j) at ([yshift=\toff] B.rightleg) [above] {};
                \node (i2) at ([yshift=-\boff] A.leftleg) [below] {};
                \node (j2) at ([yshift=-\boff] A.rightleg) [below] {};
                \draw (i.south) to[out=down, in=up] (B-belt.in-leftthird) to[out=down, in=up] (i2.north);
                \draw (j.south) to[out=down, in=up] (B-belt.in-rightthird) to[out=down, in=up] (j2.north);
        \end{scope}
\end{tz}
\end{align}
\end{prop}

\begin{proof}

We start with $\mathcal{Z}^\prime(\epsilon)$: 


    This corresponds to the natural transformation between the bimodules

    $$\Hom_\C(Q,P\otimes\coend)\xrightarrow{\Znc(\epsilon)_*}\Hom(Q,P),$$

    given by postcomposition with $\varepsilon$.

We want to understand:

$$\int^{X,Y}\hspace{-15pt}\Hom_\C(Q,X\otimes Y)\otimes\Hom_\C(X\otimes Y, P)\xrightarrow{\Znc^\prime(\epsilon)}\Hom(Q,P)$$

In other words, we aim to understand the isomorphism: 

$$\int^{X,Y}\hspace{-15pt}\Hom_\C(Q,X\otimes Y)\otimes\Hom_\C(X\otimes Y, P)\xrightarrow{\sim}\Hom_\C(Q,P\otimes\coend)$$

This is the composite of

\begin{equation*}
    \setlength{\arraycolsep}{0pt}
\renewcommand{\arraystretch}{1.2}
  \begin{array}{ c c c }
    \displaystyle\int^{X,Y}\hspace{-15pt}\Hom_\C(Q,X\otimes Y)\otimes\Hom_\C(X\otimes Y, P) & {} 
    \xrightarrow{\id\otimes(-)^\flat} {} & 
   \displaystyle\int^{X,Y}\hspace{-15pt}\Hom_\C(Q,X\otimes Y)\otimes\Hom_\C(X,P\otimes Y^\lor) \\
    f\otimes g           & {} \mapsto {} & f\otimes\left((g\otimes\id_{Y^\lor})\circ(\id_{P}\otimes\text{coev}_Y)\right)
  \end{array}
\end{equation*}


with the co-Yoneda isomorphism

\begin{equation*}
    \setlength{\arraycolsep}{0pt}
\renewcommand{\arraystretch}{1.2}
  \begin{array}{ c c c }
    \displaystyle\int^{X,Y}\hspace{-15pt}\Hom_\C(Q,X\otimes Y)\otimes\Hom_\C(X,P\otimes Y^\lor) & {} 
    \xrightarrow{\sim} {} & 
   \displaystyle\int^{Y}\hspace{-10pt}\Hom_\C(Q,P\otimes Y^\lor\otimes Y) \\
    f\otimes h           & {} \mapsto {} & (g\otimes\id_{Y})\circ f
  \end{array}
\end{equation*}


Thanks to the exactness of $\Hom_\C(Q,-)$ and the fact that $\varepsilon$ is given by $\text{ev}_Y$ \ref{coalgebra}, we conclude that $\Znc^\prime(\epsilon)$ acts by composition, as desired.


Now for $\Znc^\prime(\eta)$:

The corresponding natural transformation of associated bimodules is:

$$\Hom_{\C\boxtimes\C}(Q_1\boxtimes Q_2,P_1\boxtimes P_2)\xrightarrow{\Znc(\eta)_*}\Hom_{\C\boxtimes\C}(Q_1\boxtimes Q_2,P_1\otimes P_2\otimes\int^Y Y^\lor\boxtimes Y),$$

given by the coevaluation on $Q_2$ and inclusion in the coend, essentially as explained in Remark \ref{canonical_coaction}. We remind the reader that this is the unit of the adjunction $T\dashv T^R$.

We want to understand

\begin{equation}\label{eta_in_bimod}
    \Znc^\prime(\eta)\colon\Hom(Q_1,P_1)\otimes \Hom(Q_2,P_2)\to\Hom_\C(Q_1\otimes Q_2,P_1\otimes P_2)
\end{equation}

By definition, it is true that $\Hom_{\C\boxtimes\C}(Q_1\boxtimes Q_2,P_1\boxtimes P_2)\cong\Hom(Q_1,P_1)\otimes \Hom(Q_2,P_2)$.

Let $f\colon Q_1\to P_1$ and $g\colon Q_2\to P_2$.

The counit of the adjunction $T\dashv T^R$, gives:

\begin{equation*}
    \setlength{\arraycolsep}{0pt}
\renewcommand{\arraystretch}{1.2}
  \begin{array}{ c c c }
    \Hom_{\C\boxtimes\C}(Q_1\boxtimes Q_2,P_1\otimes P_2\otimes\!\!\!\displaystyle\int^Y \!\!\!\!\!\!Y^\lor\boxtimes Y) & {} 
    \xrightarrow{\sim} {} & 
   \Hom_\C(Q_1\otimes Q_2,P_1\otimes P_2) \\
    f\boxtimes g           & {} \mapsto {} & (\id_{P_1}\otimes \id_{P_2}\otimes\varepsilon)\circ(f\otimes g)
  \end{array}
\end{equation*}

Due to the cancellation of unit and counit (coevaluation and evaluation), we find that the desired morphism $\Znc^\prime(\eta)$ is given by the inclusion 
$${\Hom(Q_1,P_1)\otimes \Hom(Q_2,P_2)\xhookrightarrow{}\Hom_\C(Q_1\otimes Q_2,P_1\otimes P_2)}.$$

\end{proof}

Let $f\colon Q\to P$. Abusing notation, we call the image of $f\otimes\Lambda$ under 
$$\Hom_\C(Q,P\otimes \coend)\xrightarrow{\sim} \displaystyle\int^{Y}\hspace{-7pt}\Hom_\C(Q,P\otimes Y^\lor\otimes Y)$$ again $f\otimes\Lambda$.

\begin{prop}\label{epsdag_bimod}
    $\Znc^\prime(\epsilon^\dagger)$ acts as follows on internal string diagrams:

    $$\begin{tz}
    \node (A) [Cyl, tall, bot, top] at (0,0) {};
    \begin{scope}[curvein]
        \draw ([yshift=\toff] A.top) to ([yshift=-\boff] A.bot);
    \end{scope}
\end{tz}
\quad
 \xmapsto{\Znc^\prime (\epsilon ^\dag)}
\quad
{ \color{red} \frac{1}{\mathscrsfs{D}} }
\begin{tz}
    \node (A) [Pants, bot, top] at (0,0) {};
    \node (B) [Copants, bot, anchor=leftleg] at (A.leftleg) {};
    \begin{scope}[curvein]
        \draw ([yshift=\toff] A.belt)
            to [out=down, in=up, out looseness=1.5] (A-leftleg.in-leftthird)
            to [out=down, in=up, in looseness=1.7] ([yshift=-\boff] B.belt);
        \draw[red strand] (A-leftleg.in-rightthird)
            to [out=up,in=up,looseness=1.7]
                (A-rightleg.in-leftthird)
            to [out=down,in=down, looseness=1.7] (A-leftleg.in-rightthird);
    \end{scope}
\end{tz}$$

The convention for the red circle is that the top part is $\Lambda$ and the bottom is $\text{ev}$.
\end{prop}

\begin{proof}

We have that $\Znc(\epsilon^\dagger)_*$ is given by:

$$\Hom_\C(Q,P)\xrightarrow{\left(\mathscrsfs{D}^\inv(\id_P\otimes\Lambda)\right)_*}\Hom(Q,P\otimes\coend).$$

Computing $\Znc^\prime(\epsilon^\dagger)$ in terms of string diagrams involves the isomorphisms:

\begin{equation*}
    \setlength{\arraycolsep}{0pt}
\renewcommand{\arraystretch}{1.2}
  \begin{array}{ c c c }
    \displaystyle\int^{Y}\hspace{-7pt}\Hom_\C(Q,P\otimes Y^\lor\otimes Y) & {} 
    \xrightarrow{\sim} {} & 
    \displaystyle\int^{X,Y}\hspace{-15pt}\Hom_\C(Q,X\otimes Y)\otimes\Hom_\C(X,P\otimes Y^\lor) \\
    h           & {} \mapsto {} & h\otimes(\id_{P\otimes Y^\lor})
  \end{array}
\end{equation*}


and

\begin{equation*}
    \setlength{\arraycolsep}{0pt}
\renewcommand{\arraystretch}{1.2}
  \begin{array}{ c c c }
    \displaystyle\int^{X,Y}\hspace{-15pt}\Hom_\C(Q,X\otimes Y)\otimes\Hom_\C(X,P\otimes Y^\lor) & {} 
    \xrightarrow{\sim} {} & 
     \displaystyle\int^{X,Y}\hspace{-15pt}\Hom_\C(Q,X\otimes Y)\otimes\Hom_\C(X\otimes Y, P)\\
    h\otimes g & {} \mapsto {} & h\otimes\left((\id_P\otimes\text{ev}_Y)\circ (g\otimes \id_Y)\right)
  \end{array}
\end{equation*}

Under the full composite, the image of $f\colon Q\to P$ is:

$$f\mapsto \mathscrsfs{D}^\inv(f\otimes\Lambda)\otimes(\id_P\otimes\text{ev}_Y),$$

which is the desired morphism.
\end{proof}

\begin{remark}
    The red circle notation coincides with that from \cite{de_renzi_3-dimensional_2022}.
\end{remark}

\begin{prop}\label{beta-theta_bimod}
    The action of $\Znc^\prime(\beta)$ and $\Znc^\prime(\theta)$ on internal string diagrams is given by:

    \begin{align}
\label{braiding-and-twist-image}
\begin{tz}
    \node[Pants, bot, top, height scale=1.3] (A) at (0,0) {};
    \begin{scope}[internal string scope]
        \draw ([yshift=\toff] A-belt.in-leftthird)
                to +(0,-0.3)
            to [out=down, in=up] (A.leftleg)
            to +(0,-0.3);
        \draw ([yshift=\toff] A-belt.in-rightthird)
                to +(0,-0.3)
            to [out=down, in=up] (A.rightleg)
            to +(0,-0.3);
    \end{scope}
\end{tz}
&\gap\xmapsto{\Znc^\prime(\beta)}\gap
\begin{tz}
    \node[Pants, bot, top, height scale=1.3] (A) at (0,0) {};
    \node[BraidA, bot, anchor=topleft] (B) at (A.leftleg) {};
    \draw[internal string] ([yshift=\toff] A-belt.in-rightthird)
        to (A-belt.in-rightthird)
        to [out=down, in=up, out looseness=1.7, in looseness=0.7] (B.topleft);
    \draw [internal string] (B.bottomright) to +(0,-0.3);
    \draw[internal string] ([yshift=\toff] A-belt.in-leftthird)
        to +(0,-0.3)
        to [out=down, in=up, out looseness=1.7, in looseness=0.7] (A.rightleg)
        to [out=down, in=up, looseness=0.6] (B.bottomleft)
        to +(0,-0.3);
    \obscureA{B}
    {
\draw [solid, red] (B.topleft) to [out=down, in=up, looseness=0.6] (B.bottomright);
    }
\end{tz}
&
\begin{tz}
    \node[Cyl, bot, top, height scale=1.3] (A) at (0,0) {};
    \begin{scope}[internal string scope]
        \node (i) at ([yshift=\toff] A.top) [above] {};
        \node (i2) at ([yshift=-\boff] A.bot) [below] {};
        \draw (i.south) -- (i2.north);
    \end{scope}
\end{tz}
&\gap\xmapsto{\Znc^\prime(\theta)}\gap
\begin{tz}
    \node[Cyl, bot, top, height scale=1.3] (A) at (0,0) {};
    \begin{scope}[internal string scope]
        \node (i) at ([yshift=\toff] A.top) [above] {};
        \node (i2) at ([yshift=-\boff] A.bot) [below] {};
        \draw (i.south) -- (i2.north);
        \node [stiny label] at (A.center) {$\vartheta$};
    \end{scope}
\end{tz}
\end{align}
\end{prop}

\begin{proof}
    This is automatic from the action of $\Znc(\beta)_*$ and $\Znc(\theta)_*$.
\end{proof}

\begin{prop}\label{mudag-nudag_bimod}
$\Znc^\prime(\mu^\dagger)\colon\Hom_\C(Q,P)\to\Hom_\C(Q,\unit)\otimes\Hom_\C(P,\unit)^*$ is given by

\begin{align}
\begin{split}
    \Hom_\C(Q,P)\xrightarrow{\id\otimes\text{coev}_{\Hom_\mathcal{C}(P,\unit)}}& \Hom_\C(Q,P)\otimes\Hom_\C(P,\unit) \otimes\Hom_\mathcal{C}(P,\unit)^*\\
    \xrightarrow{\text{compose}\otimes\id}&\Hom_\C(Q,\unit)\otimes{\Hom_\mathcal{C}(P,\unit)^*}
\end{split}
\end{align}

and $\Znc^\prime(\nu^\dagger)$ is given by the evaluation of vector spaces:

\begin{equation}
    \mathcal{Z}^\prime(\nu^\dagger)\colon
\displaystyle\int^{P\in\cP}\hspace{-20pt}\Hom_\C(P,\unit)\otimes\Hom_\C(P,\unit)^*\xrightarrow{\text{ev}_{\Hom_\mathcal{C}(P,\unit)}}k
\end{equation}

\end{prop}

\begin{proof}
    The action of $\Znc^\prime(\mu^\dagger)$ follows directly from properties of $EV$ \ref{evaluation-counit}.

As for $\Znc^\prime(\nu^\dagger)$, this is also straightforward, as discussed in Remark \ref{cap_description}.
\end{proof}

\newpage

\section{3-manifold invariants}\label{3-mfld-inv}

In this section, we will compute the 3-manifold invariant the non-compact TQFT defined in Theorem \ref{main_th} gives rise to. We do this for the case where $\Znc$ is oriented, i.e when $\C$ is anomaly-free. We will have to be a bit loose and think of the generators of $\Bord$ geometrically. We will implicitly employ this strategy throughout this chapter. This procedure could be made formal via a `geometric realization' functor \cite{bartlett_extended_2014,bartlett_modular_2015}. 

Before we get into explicit computations, we need to explain how to extract a closed 3-manifold invariant from a non-compact TQFT.
Despite the fact that the noncompact 3d TQFT does not have assignments for closed 3-manifolds, we can still extract invariants associated to them, via the following procedure:

Given a closed, connected 3-manifold $M$, we can evaluate the non-compact TQFT on $M$, after removing two copies of $D^3$, one incoming, the other outgoing. That is: 

$$\Znc(M\setminus (D^3\sqcup D^3))\colon \Znc(S^2)\to \Znc(S^2).$$

The vector space $\Znc(S^2)$ is one dimensional, so the above linear map is just an element of $k$. We will prove that this number is the Lyubashenko 3-manifold invariant \cite{lyubashenko_invariants_1995} multiplied by a factor. We also note that, since we work in the non-compact setting, and the Lyubashenko invariant specializes to the Reshetikhin-Turaev invariant in the semisimple case, we also prove that the TQFT constructed by \cite{bartlett_modular_2015} recovers the Reshetikhin-Turaev invariant.



\subsection{3-manifold presentation in terms of generators}

To achieve our goal, the first thing we need to address is how to present any closed oriented 3-manifold in terms a composite of 2-morphism generators. We therefore work in the full cobordism bicategory $\Bordor$, and use the generator 
$$\begin{tz}
  \draw[green] (0,0) rectangle (0.6, -0.6);  
 \end{tz}
 \,\,
\Rarrow{\nu}
\,\,
 \begin{tz}
         \node[Cap, bot] (A) at (0,0) {};
         \node[Cup] at (0,0) {};
 \end{tz}$$

The decomposition we provide is based on a surgery presentation of the 3-manifold. 

\begin{theorem}[Lickorish-Wallace]\label{Lick-Wall}
    Any closed, oriented, connected 3-manifold $M$ is realized
by (integral) Dehn surgery on a link $L$ in $S^3$.
\end{theorem}

Assume $M$ has a surgery presentation in terms of a framed knot $K$.
Dehn surgery is the procedure of gluing a solid torus $S^1\times D^2$, into the knot complement $S^3\setminus T(K)$, twisted by $\varphi$.

Then, via Theorem \ref{Lick-Wall}, we get that $M= (S^3\setminus T(K))\cup_\varphi (S^1\times D^2)$, where $T(K)$ is the tubular neighborhood of the knot $K$ and $\varphi\colon T^2\to \partial\left(S^3\setminus T(K)\right)$ is the diffeomorphism determined by the framing of $K$.

Viewing this as a composite of cobordisms, we have:

\begin{equation}\label{3mfld-decomp}
    M= (S^3\setminus T(K))\circ I_\varphi\circ (S^1\times D^2),
\end{equation}

where $I_\varphi$ is the mapping cylinder of the diffeomorphism $\varphi$.

We can think of $\varphi$ as a composite $\varphi=\varphi^\prime\circ S$. First, applying an $S$  transformation ($S_1$ in the notation of Chapter \ref{mcg_actions}) on $T^2=\partial(S^1\times D^2)$, and then the diffeomorphism ${\varphi^\prime\colon T^2\to \partial\left(S^3\setminus T(K)\right)}$, which involves Dehn twists $T_\alpha$ to `match' the framing of $K$.

In order to express the knot complement $S^3\setminus T(K)$ in terms of generators, we write it as a composite of certain 3-cobordisms. This essentially involves `passing through' a genus $g$ surface:

\begin{equation}\label{decomp-knot-comp}
    S^3\setminus T(K)=\left((S^3\setminus H_g^\mathrm{o})\cup_{\id_{\Sigma_g}}(H_g\setminus T(K))\right),
\end{equation}

where $g$ depends on the choice of knot diagram of $K$ and $H_g$ is the genus $g$ handlebody.

\begin{figure}[ht]
    \centering
    \includegraphics[width=0.6\linewidth]{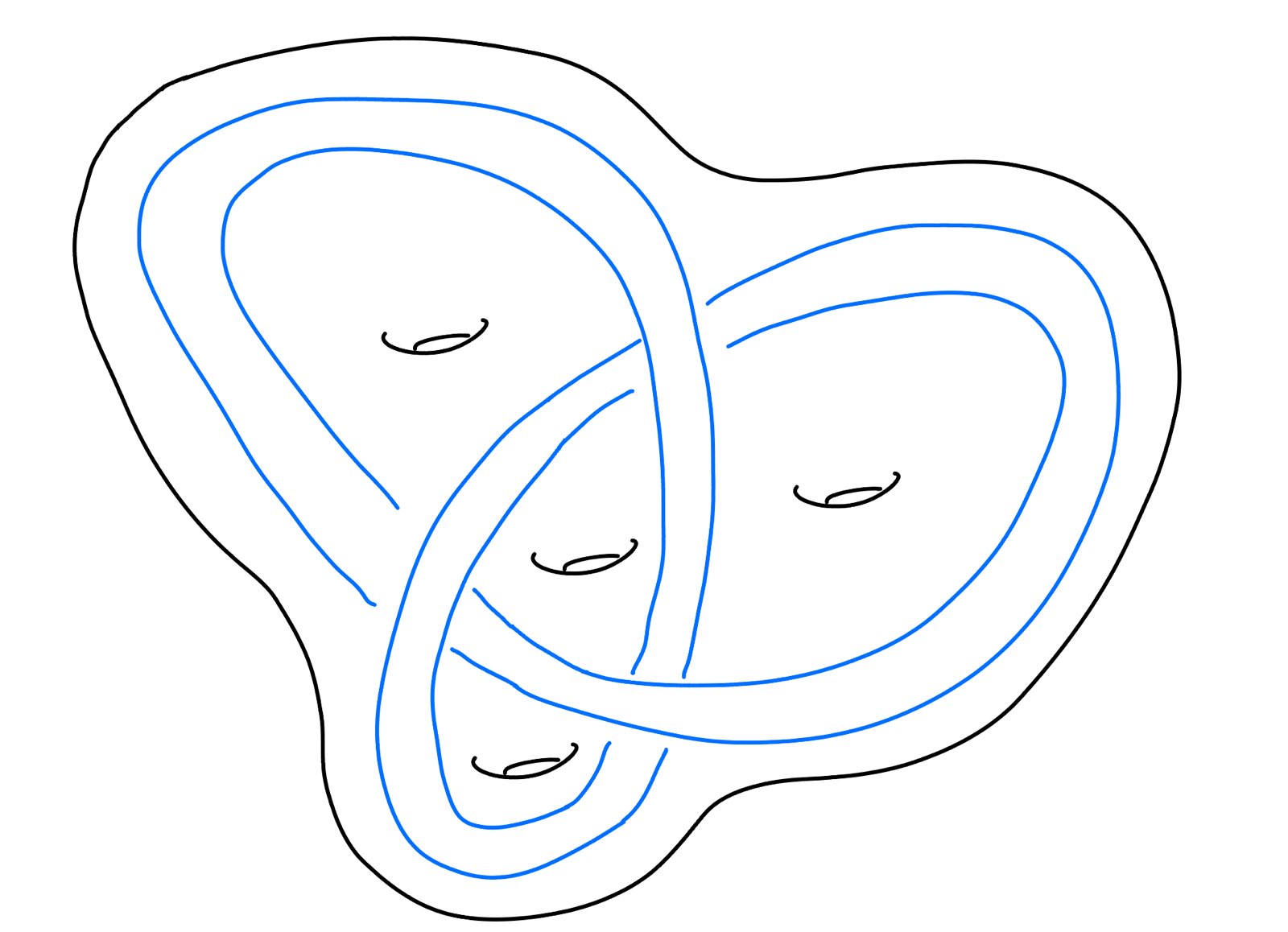}
    \caption{Cobordism from $\partial(T(K))$ (\textcolor{blue}{blue}) to $\Sigma_4$ (black) for the trefoil.}
    \label{trefoil_expl}
\end{figure}


First let us understand $H_g\setminus T(K)$.
Fix a presentation of the knot diagram. 

The data of a knot diagram is the same as that of a connected planar graph, if we replace crossings by vertices. If the number of faces of this planar graph is $g+1$, its tubular neighborhood is diffeomorphic to $H_g$. We can obtain a cobordism $H_g\setminus T(K)\colon\partial(T(K))\to \Sigma_g$ as the complement of the tubular neighborhood of the knot diagram in $H_g$, as depicted in Figure \ref{trefoil_expl} for the example of the trefoil knot. The handlebody $(S^3\setminus H_g)$ is given by applying $g$-many 2-handles (filling the holes) and a 3-handle 
$$\Sigma_g\xrightarrow{\text{2-handles}} S^2\xrightarrow{\text{3-handle}}\emptyset$$ 
gives the full decomposition of $S^3\setminus T(K)$.


This procedure is by no means unique. However, one can already see that it gives rise to a decomposition of $M$ in terms of generators. We have explicit generators matching the 2- and 3- handles ($\epsilon$ and $\nu^\dagger$), and the 3-cobordism $H_g\setminus T(K)$ is non-trivial only at crossings. In other words, it can be obtained by an application of a 1-handle ($\eta$) and a $\beta$ to match the crossing. This last point is the merit of the decomposition of \ref{decomp-knot-comp}.








We now write this decomposition \ref{3mfld-decomp} in terms of 2-morphism generators for certain examples.

Note that the solid torus $S^1\times D^2$, whose boundary is twisted by $S$, is given by:

$$\begin{tz}
  \draw[green] (0,0) rectangle (0.6, -0.6);  
 \end{tz}
\,\,
\Rarrow{\nu}
\,\,
\begin{tz}
        \node[Cap] (A) at (0,0) {};
        \node[Cup] (B) at (0,0) {};
        \node [Cobordism Bottom End 3D] (C) at (0,0) {};
        \selectpart[green, inner sep=1pt] {(C)};
\end{tz}
\,\,
\Rarrow{\epsilon^\dagger}
\,\,
\begin{tz}
        \node[Cap] (A) at (0,0) {};
        \node[Pants, anchor=belt] (B) at (A.center) {};
        \node[Copants, anchor=leftleg] (C) at (B.leftleg) {};
        \node[Cup] (D) at (C.belt) {};
\end{tz}
 $$


\begin{exmp}
    Let $K$ be the 0-framed trefoil. It can be represented as a braid closure as follows:

$$\begin{tz}[scale=0.6, xscale=1]
\draw [red strand] (1,0) to [out=up, in=down] (2,1);
\draw [red strand] (2,0) to [out=up, in=down] (1,1);
\draw [red strand] (1,1) to [out=up, in=down] (2,2);
\draw [red strand] (2,1) to [out=up, in=down] (1,2);
\draw [red strand] (1,2) to [out=up, in=down] (2,3);
\draw [red strand] (2,2) to [out=up, in=down] (1,3);
\draw [red] (0,0) to (0,3);
\draw [red] (3,0) to (3,3);
\draw [red strand] (0,3) to [out=up, in=up, looseness=1.2] (1,3);
\draw [red strand] (2,3) to [out=up, in=up, looseness=1.2] (3,3);
\draw [red strand] (0,0) to [out=down, in=down, looseness=1.2] (1,0);
\draw [red strand] (2,0) to [out=down, in=down, looseness=1.2] (3,0);
\end{tz}$$

    Then, the decomposition of the corresponding 3-manifold is:
    \scalecobordisms{0.6}

 \begin{equation}
\begin{tz}[xscale=2, yscale=2]
\node (1) at (0,-2)
{$\begin{tz}
        \node[Cap] (A) at (0,0) {};
        \node[Cup] (B) at (0,0) {};
        \node [Cobordism Bottom End 3D] (C) at (0,0) {};
\end{tz}$};

\node (2) at (0,0)
{$\begin{tz}
        \node[Cap] (A) at (0,0) {};
        \node[Pants, anchor=belt] (B) at (A.center) {};
        \node[Copants, anchor=leftleg] (C) at (B.leftleg) {};
        \node[Cup] (D) at (C.belt) {};
\end{tz}$};

\node (3) at (1,0)
{$\begin{tz}
  \node[Pants, bot] (P1) at (0,0) {};
  \node[Cap] at (P1.belt) {};

  \node[Cyl, verytall, bot, anchor=top] at (P1.leftleg) {};

  \node[BraidB, anchor=topleft] (B1) at (P1.rightleg) {};

  \node[BraidB, anchor=topright] (B2) at (B1.bottomright) {};

  \node[BraidB, anchor=topright] (B3) at (B2.bottomright) {};

  \node[Pants, bot, anchor=leftleg] (P2) at (B1.topright) {};
  \node[Cap] at (P2.belt) {};

  \node[Cyl, verytall, bot, anchor=top] at (P2.rightleg) {};

  \node[Copants, anchor=rightleg] (C1) at (B3.bottomleft) {};
  \node[Cup] at (C1.belt) {};

  \node[Copants, anchor=leftleg] (C2) at (B3.bottomright) {};
  \node[Cup] at (C2.belt) {};
\end{tz}$};

\node (4) at (2.5,0)
{$\begin{tz}
  \node[Pants, bot] (P1) at (0,0) {};
  \node[Cap] at (P1.belt) {};

  \node[Cyl, supertall, bot, anchor=top] at (P1.leftleg) {};

  \node[Copants, anchor=leftleg] (PB1) at (P1.rightleg) {};
  \node[Pants, bot, anchor=belt] (PB11) at (PB1.belt) {};

  \node[BraidB, anchor=topleft] (B1) at (PB11.leftleg) {};

  \node[Copants, anchor=leftleg] (PB2) at (B1.bottomleft) {};
  \node[Pants, bot, anchor=belt] (PB22) at (PB2.belt) {};

  \node[BraidB, anchor=topleft] (B2) at (PB22.leftleg) {};

  \node[Copants, anchor=leftleg] (PB3) at (B2.bottomleft) {};
  \node[Pants, bot, anchor=belt] (PB33) at (PB3.belt) {};

  \node[BraidB, anchor=topleft] (B3) at (PB33.leftleg) {};

  \node[Pants, bot, anchor=leftleg] (P2) at (PB1.rightleg) {};
  \node[Cap] at (P2.belt) {};

  \node[Cyl, supertall, bot, anchor=top] at (P2.rightleg) {};

  \node[Copants, anchor=rightleg] (C1) at (B3.bottomleft) {};
  \node[Cup] at (C1.belt) {};

  \node[Copants, anchor=leftleg] (C2) at (B3.bottomright) {};
  \node[Cup] at (C2.belt) {};
\end{tz}$};

\node (5) at (4,0)
{$\begin{tz}
  \node[Pants, bot] (P1) at (0,0) {};
  \node[Cap] at (P1.belt) {};

  \node[Cyl, lightlysupertall, bot, anchor=top] at (P1.leftleg) {};

  \node[Copants, anchor=leftleg] (PB1) at (P1.rightleg) {};
  \node[Pants, bot, anchor=belt] (PB11) at (PB1.belt) {};


  \node[Copants, anchor=leftleg] (PB2) at (PB11.leftleg) {};
  \node[Pants, bot, anchor=belt] (PB22) at (PB2.belt) {};


  \node[Copants, anchor=leftleg] (PB3) at (PB22.leftleg) {};
  \node[Pants, bot, anchor=belt] (PB33) at (PB3.belt) {};


  \node[Pants, bot, anchor=leftleg] (P2) at (PB1.rightleg) {};
  \node[Cap] at (P2.belt) {};

  \node[Cyl, lightlysupertall, bot, anchor=top] at (P2.rightleg) {};

  \node[Copants, anchor=rightleg] (C1) at (PB33.leftleg) {};
  \node[Cup] at (C1.belt) {};

  \node[Copants, anchor=leftleg] (C2) at (PB33.rightleg) {};
  \node[Cup] at (C2.belt) {};
\end{tz}$};

\node (6) at (4,-2)
{$\begin{tz}
        \node[Cap] (A) at (0,0) {};
        \node[Cup] (B) at (0,0) {};
        \node [Cobordism Bottom End 3D] (C) at (0,0) {};
\end{tz}$};

\node (7) at (3,-2)
{$\begin{tz}
  \draw[green] (0,0) rectangle (0.6, -0.6);  
 \end{tz}$};

\node (8) at (1,-2)
{$\begin{tz}
  \draw[green] (0,0) rectangle (0.6, -0.6);  
 \end{tz}$};

\begin{scope}[double arrow scope]
    \draw (1) -- node [left] {$\epsilon^\dagger$} (2);
    \draw[] ([xshift=0pt] 2.0) to node [above, inner sep=1pt] {$I_{\varphi^\prime}$} (3);
    \draw (3) -- node [above] {$\eta^3$} (4);
    \draw (4) -- node [above] {$\beta^3$} (5);
    \draw (5) -- node [right] {$\epsilon^4$} (6);
    \draw (8) -- node [above] {$\nu$} (1);
    \draw (6) -- node [above] {$\nu^\dagger$} (7);
\end{scope}
\end{tz}
\end{equation}
\end{exmp}

\begin{remark}
    To apply some of the $\epsilon$'s we need to use $\phileft^\inv$ and associators $\alpha$ or $\alpha^\inv$.
\end{remark}

The case where $M$ is obtained via surgery on a link works entirely analogously. Let $L$ be the link with $l$ components. The only meaningful difference is the use of $\mu^\dagger$ $l-1$ times, followed by $l$ applications of $\epsilon^\dagger$ on the resulting $S^2$s.

\begin{exmp}
    
Let $L$ be the Hopf link (with 0 framing on both components):

$$\begin{tz}[scale=0.6, xscale=1]
\draw [red strand] (1,0) to [out=up, in=down] (2,1);
\draw [red strand] (2,0) to [out=up, in=down] (1,1);
\draw [red strand] (1,1) to [out=up, in=down] (2,2);
\draw [red strand] (2,1) to [out=up, in=down] (1,2);
\draw [red] (0,0) to (0,2);
\draw [red] (3,0) to (3,2);
\draw [red strand] (0,2) to [out=up, in=up, looseness=1.2] (1,2);
\draw [red strand] (2,2) to [out=up, in=up, looseness=1.2] (3,2);
\draw [red strand] (0,0) to [out=down, in=down, looseness=1.2] (1,0);
\draw [red strand] (2,0) to [out=down, in=down, looseness=1.2] (3,0);
\end{tz}$$

If $M$ is obtained by surgery on $L$ it can be decomposed as follows:

\scalecobordisms{0.6}

 \begin{equation}
\begin{tz}[xscale=2, yscale=2]
\node (1) at (1,-2)
{$\begin{tz}
        \node[Cap] (A) at (0,0) {};
        \node[Cup] (B) at (0,0) {};
        \node [Cobordism Bottom End 3D] (C) at (0,0) {};
\end{tz}$};

\node (2) at (0,-0.75)
{$\begin{tz}
        \node[Cap] (A) at (0,0) {};
        \node[Pants, anchor=belt] (B) at (A.center) {};
        \node[Copants, anchor=leftleg] (C) at (B.leftleg) {};
        \node[Cup] (D) at (C.belt) {};
        \node[Cap] (E) at (-1,0) {};
        \node[Pants, anchor=belt] (F) at (E.center) {};
        \node[Copants, anchor=leftleg] (G) at (F.leftleg) {};
        \node[Cup] (H) at (G.belt) {};
\end{tz}$};

\node (3) at (1,0.5)
{$\begin{tz}
  \node[Pants, bot] (P1) at (0,0) {};
  \node[Cap] at (P1.belt) {};

  \node[Cyl, tall, bot, anchor=top] (CylL) at (P1.leftleg) {};

  \node[BraidB, anchor=topleft] (B1) at (P1.rightleg) {};

  \node[BraidB, anchor=topright] (B2) at (B1.bottomright) {};

  \node[Pants, bot, anchor=leftleg] (P2) at (B1.topright) {};
  \node[Cap] at (P2.belt) {};

  \node[Cyl, tall, bot, anchor=top] (CylR) at (P2.rightleg) {};

  \node[Copants, anchor=rightleg] (C1) at (B2.bottomleft) {};
  \node[Cup] at (C1.belt) {};

  \node[Copants, anchor=leftleg] (C2) at (B2.bottomright) {};
  \node[Cup] at (C2.belt) {};
\end{tz}$};

\node (4) at (2.5,0)
{$\begin{tz}
  \node[Pants, bot] (P1) at (0,0) {};
  \node[Cap] at (P1.belt) {};

  \node[Cyl, lightlysupertall, bot, anchor=top] at (P1.leftleg) {};

  \node[Copants, anchor=leftleg] (PB1) at (P1.rightleg) {};
  \node[Pants, bot, anchor=belt] (PB11) at (PB1.belt) {};

  \node[BraidB, anchor=topleft] (B1) at (PB11.leftleg) {};

  \node[Copants, anchor=leftleg] (PB2) at (B1.bottomleft) {};
  \node[Pants, bot, anchor=belt] (PB22) at (PB2.belt) {};

  \node[BraidB, anchor=topleft] (B2) at (PB22.leftleg) {};

  \node[Pants, bot, anchor=leftleg] (P2) at (PB1.rightleg) {};
  \node[Cap] at (P2.belt) {};

  \node[Cyl, lightlysupertall, bot, anchor=top] at (P2.rightleg) {};

  \node[Copants, anchor=rightleg] (C1) at (B2.bottomleft) {};
  \node[Cup] at (C1.belt) {};

  \node[Copants, anchor=leftleg] (C2) at (B2.bottomright) {};
  \node[Cup] at (C2.belt) {};
\end{tz}$};

\node (5) at (4,0)
{$\begin{tz}
  \node[Pants, bot] (P1) at (0,0) {};
  \node[Cap] at (P1.belt) {};

  \node[Cyl, veryverytall, bot, anchor=top] at (P1.leftleg) {};

  \node[Copants, anchor=leftleg] (PB1) at (P1.rightleg) {};
  \node[Pants, bot, anchor=belt] (PB11) at (PB1.belt) {};


  \node[Copants, anchor=leftleg] (PB2) at (PB11.leftleg) {};
  \node[Pants, bot, anchor=belt] (PB22) at (PB2.belt) {};



  \node[Pants, bot, anchor=leftleg] (P2) at (PB1.rightleg) {};
  \node[Cap] at (P2.belt) {};

  \node[Cyl, veryverytall, bot, anchor=top] at (P2.rightleg) {};

  \node[Copants, anchor=rightleg] (C1) at (PB22.leftleg) {};
  \node[Cup] at (C1.belt) {};

  \node[Copants, anchor=leftleg] (C2) at (PB22.rightleg) {};
  \node[Cup] at (C2.belt) {};
\end{tz}$};

\node (6) at (4,-2)
{$\begin{tz}
        \node[Cap] (A) at (0,0) {};
        \node[Cup] (B) at (0,0) {};
        \node [Cobordism Bottom End 3D] (C) at (0,0) {};
\end{tz}$};

\node (7) at (3,-2)
{$\begin{tz}
  \draw[green] (0,0) rectangle (0.6, -0.6);  
 \end{tz}$};

\node (8) at (2,-2)
{$\begin{tz}
  \draw[green] (0,0) rectangle (0.6, -0.6);  
 \end{tz}$};

 \node (9) at (0,-2)
{$\begin{tz}
        \node[Cup] (A) at (0,0) {};
        \node[Cap, bot] (C) at (0,0) {};
        \node[Cap] (B) at (0,-2*\cobheight) {};
        \node[Cup, bot] (D) at (0,-2*\cobheight) {};
    \end{tz}$};

\begin{scope}[double arrow scope]
    \draw (9) -- node [left] {$(\epsilon^\dagger)^2$} (2);
    \draw (2) -- node [left] {$I_{\varphi^\prime}$} (3);
    \draw[] ([xshift=0pt] 4.0) to node [above, inner sep=1pt] {${\beta}^2$} (5);
    \draw (3) -- node [above] {$\eta^2$} (4);
    \draw (1) -- node [above] {$\mu^\dagger$} (9);
    \draw (5) -- node [right] {$\epsilon^3$} (6);
    \draw (8) -- node [above] {$\nu$} (1);
    \draw (6) -- node [above] {$\nu^\dagger$} (7);
\end{scope}
\end{tz}
\end{equation}

\end{exmp}

We can therefore decompose any 3-manifold $M$ as the examples above. By $M^*$ we denote the composite $S^2\to S^2$ obtained by throwing away $\nu$ and $\nu^\dagger$ from the decomposition of $M$.

\subsection{Computing the invariants}

Let $M$ be an oriented 3-manifold that has a surgery presentation in terms of a framed link $L$ with $l$ components.
Denote the Lyubashenko invariant of $M$ \cite{lyubashenko_invariants_1995} by $\mathcal{L}(M)$, and $\mathcal{RT}(M)$ the Reshetikhin-Turaev invariant \cite{reshetikhin_invariants_1991} in the semisimple case.

In order to compute the invariant associated to $M$ by $\Znc$, we make the following assumption.

\begin{lemma}\label{conj-action-strings}
    $I_{\varphi^\prime}$ acts on

$$\begin{tz}
        \node[Cap] (A) at (0,0) {};
        \node[Pants, anchor=belt] (B) at (A.center) {};
        \node[Copants, anchor=leftleg] (C) at (B.leftleg) {};
        \node[Cup] (D) at (C.belt) {};
        \draw[internal string] (B.leftleg)
        to[out=up, in=up, looseness=1.7] (B.rightleg);
        \draw[internal string] (B.leftleg)
        to[out=down, in=down, looseness=1.7] (B.rightleg);
\end{tz}
\quad \dots \quad
\begin{tz}
        \node[Cap] (A) at (0,0) {};
        \node[Pants, anchor=belt] (B) at (A.center) {};
        \node[Copants, anchor=leftleg] (C) at (B.leftleg) {};
        \node[Cup] (D) at (C.belt) {};
        \draw[internal string] (B.leftleg)
        to[out=up, in=up, looseness=1.7] (B.rightleg);
        \draw[internal string] (B.leftleg)
        to[out=down, in=down, looseness=1.7] (B.rightleg);
\end{tz}$$

by having the red strand of the $i$-th torus go through all the (interior of the) surface $\partial\left(S^3\setminus(T(L_i)) \right)$ with appropriate twists $\vartheta$, matching the framing of $L$.

For example, in the case of the +1-framed trefoil we have:

\scalecobordisms{0.6}
\begin{equation}
\begin{tz}[xscale=2, yscale=2]
\node (2) at (0,0)
{$\begin{tz}
        \node[Cap] (A) at (0,0) {};
        \node[Pants, anchor=belt] (B) at (A.center) {};
        \node[Copants, anchor=leftleg] (C) at (B.leftleg) {};
        \node[Cup] (D) at (C.belt) {};
        \draw[internal string] (B.leftleg)
        to[out=up, in=up, looseness=1.7] (B.rightleg);
        \draw[internal string] (B.leftleg)
        to[out=down, in=down, looseness=1.7] (B.rightleg);
\end{tz}$};

\node (3) at (2,0)
{$\begin{tz}
  \node[Pants, bot] (P1) at (0,0) {};
  \node[Cap] at (P1.belt) {};

  \node[Cyl, verytall, bot, anchor=top] (CylL) at (P1.leftleg) {};

  \node[BraidB, anchor=topleft] (B1) at (P1.rightleg) {};

  \node[BraidB, anchor=topright] (B2) at (B1.bottomright) {};

  \node[BraidB, anchor=topright] (B3) at (B2.bottomright) {};

  \node[Pants, bot, anchor=leftleg] (P2) at (B1.topright) {};
  \node[Cap] at (P2.belt) {};

  \node[Cyl, verytall, bot, anchor=top] (CylR) at (P2.rightleg) {};

  \node[Copants, anchor=rightleg] (C1) at (B3.bottomleft) {};
  \node[Cup] at (C1.belt) {};

  \node[Copants, anchor=leftleg] (C2) at (B3.bottomright) {};
  \node[Cup] at (C2.belt) {};


  \draw[internal string] (B1.topleft) to[out=down, in=up] (B1.bottomright);
  \obscureA{B}
    {
\draw[solid, red] (B1.topright) to[out=down, in=up] (B1.bottomleft);
    }

\draw[internal string] (B2.topleft) to[out=down, in=up] (B2.bottomright);
\obscureA{B}{
  \draw[solid, red] (B2.topright) to[out=down, in=up] (B2.bottomleft);
}

\draw[internal string] (B3.topleft) to[out=down, in=up] (B3.bottomright);
\obscureA{B}{
  \draw[solid, red] (B3.topright) to[out=down, in=up] (B3.bottomleft);
}

  \begin{scope}[internal string scope]



    \draw[red strand] (CylL.top) -- (CylL.bottom);
\draw[red strand] (CylR.top) to (CylR.center) node [stiny label] {$\vartheta$} to (CylR.bottom);

    \draw[red strand] (CylL.top)
        to[out=up, in=up, looseness=1.7] (B1.topleft);

    \draw[red strand] (CylR.top)
        to[out=up, in=up, looseness=1.7] (B1.topright);

    \draw[red strand] (CylL.bottom)
    to[out=down, in=down, looseness=1.7] (B3.bottomleft);

    \draw[red strand] (CylR.bottom)
    to[out=down, in=down, looseness=1.7] (B3.bottomright);

  \end{scope}
\end{tz}$};
\begin{scope}[double arrow scope]
    \draw (2) -- node [above] {$\Znc^\prime(I_{\varphi^\prime})$} (3);
\end{scope}
\end{tz}
\end{equation}

\end{lemma}

Intuitively, this is because Dehn surgery sends the meridian of the $S^1\times D^2$ glued to the $i$-th component of $L$ to the curve on $\partial\left(S^3\setminus T(L_i)\right)$ determined by the framing of $L_i$.

\begin{proof}
    This claim essentially follows from the fact that $\Bord$ is a symmetric monoidal 2-category. In particular, every `linked' surface (1-morphism that is the boundary of the tubular neighborhood of the link) is isomorphic to a disjoint union of tori. To see this, we can use elementary moves (2-isomorphisms) like:

\begin{equation}
    \begin{tz}[xscale=1.4, yscale=2]
\node (1) at (0,0)
{
$\begin{tikzpicture}
        \node[Pants] (A) at (0,0) {};
        \node[Cap] at (A.belt) {};
        \selectpart[green, inner sep=1pt]{(A-rightleg)};
        \selectpart[red] {(A-leftleg) (A-rightleg) (A-belt)};
\end{tikzpicture}$
};
\node (2) at (2,0)
{
$\begin{tikzpicture}
        \node[Pants] (A) at (0,0) {};
        \node[Cap] at (A.belt) {};
        \node[BraidA, bot, anchor=topleft] (B) at (A.leftleg) {};
\end{tikzpicture}$
};
\begin{scope}[double arrow scope]
    \draw (1) -- node[above] {${\color{green}\theta}, \color{red}{\beta}$} (2);
\end{scope}
\end{tz},
\end{equation}

 the cusp \ref{cusp}, or 2-isomorphisms involving the symmetric braiding of the symmetric monoidal 2-category $\Bord$ and its compatibility with the tensor unit $\emptyset$. Finally, we use appropriately many $\theta$ to match the framing. In the end, the composite of these 2-isomorphisms is equal to the diffeomorphism $I_{\varphi^\prime}$ since their action on the meridian and longitude of the torus coincide.
\end{proof}


Using Lemma \ref{conj-action-strings}, we can prove the following:

\begin{theorem}\label{invariants-th}
    If $M$ is obtained by Dehn surgery on a knot, we have that:
    $\Znc^\prime(M^*)=(1/\mathscrsfs{D})\cdot\mathcal{L}(M)\cdot\id_{\Znc^\prime(S^2)}$.
    
    In the case where $\C$ is semisimple, $\Znc^\prime(M^*)=\mathcal{RT}(M)\cdot\id_{\Znc^\prime(S^2)}$. 
\end{theorem}

\begin{proof}
    We adopt the convention from Remark \ref{cap-string-diag}. This way, we can draw an internal string diagram $j\colon\unit\to A$ for the cap $\tinycap$. The string diagram of the $S^2$ we start with, is a composite morphism $\unit\to\unit$. As mentioned in the same remark, after applying $\epsilon^\dagger$, we can localize this morphism in a cylinder below $\tinycap$, using the coend relations. Thanks to this, we see that this morphism is not altered when applying the 2-morphisms from the decomposition of $M^*$. This, along with the actions of the generators computed in propositions \ref{eps-eta_bimod}, \ref{epsdag_bimod}, \ref{beta-theta_bimod} and \ref{mudag-nudag_bimod} gives us the desired result.
\end{proof}

\begin{remark}
    We expect that Theorem \ref{invariants-th} holds for 3-manifolds obtained by Dehn surgery on a link. In particular, that 
    $\Znc^\prime(M^*)=(1/\mathscrsfs{D}^l)\cdot\mathcal{L}(M)\cdot\id_{\Znc^\prime(S^2)}$.
\end{remark}

\begin{corol}
    In the semisimple case there is no non-compactness condition on the cobordisms. Therefore, the proof of Theorem \ref{invariants-th} works for any 3-manifold $M$, which in turn implies that the 3d TQFT $Z$ defined in \cite{bartlett_modular_2015} computes the Reshetikhin Turaev invariant: $Z(M)=\mathcal{RT}(M)$.
\end{corol}

Let us see how this works in examples. Note that each $\epsilon^\dagger$ gives rise to a factor of $\mathscrsfs{D}^\inv$, which we do not write. We have also adopted the convention of not drawing the element of $\Znc^\prime(S^2)$ we start with, since, as explained before, it can be localized away from where the 2-morphisms are applied.
The `final' string diagrams in $S^2$ can be thought of as living in a cylinder between $\tinycap$ and $\tinycup$.

\begin{exmp}
    
For the 0-framed trefoil knot we have:

\scalecobordisms{0.6}

 \begin{equation}
\begin{tz}[xscale=2, yscale=2]
\node (1) at (0,-2)
{$\begin{tz}
        \node[Cap] (A) at (0,0) {};
        \node[Cup] (B) at (0,0) {};
        \node [Cobordism Bottom End 3D] (C) at (0,0) {};
\end{tz}$};

\node (2) at (0,0)
{$\begin{tz}
        \node[Cap] (A) at (0,0) {};
        \node[Pants, anchor=belt] (B) at (A.center) {};
        \node[Copants, anchor=leftleg] (C) at (B.leftleg) {};
        \node[Cup] (D) at (C.belt) {};
        \draw[internal string] (B.leftleg)
        to[out=up, in=up, looseness=1.7] (B.rightleg);
        \draw[internal string] (B.leftleg)
        to[out=down, in=down, looseness=1.7] (B.rightleg);
\end{tz}$};

\node (3) at (1,0)
{$\begin{tz}
  \node[Pants, bot] (P1) at (0,0) {};
  \node[Cap] at (P1.belt) {};

  \node[Cyl, verytall, bot, anchor=top] (CylL) at (P1.leftleg) {};

  \node[BraidB, anchor=topleft] (B1) at (P1.rightleg) {};

  \node[BraidB, anchor=topright] (B2) at (B1.bottomright) {};

  \node[BraidB, anchor=topright] (B3) at (B2.bottomright) {};

  \node[Pants, bot, anchor=leftleg] (P2) at (B1.topright) {};
  \node[Cap] at (P2.belt) {};

  \node[Cyl, verytall, bot, anchor=top] (CylR) at (P2.rightleg) {};

  \node[Copants, anchor=rightleg] (C1) at (B3.bottomleft) {};
  \node[Cup] at (C1.belt) {};

  \node[Copants, anchor=leftleg] (C2) at (B3.bottomright) {};
  \node[Cup] at (C2.belt) {};


  \draw[internal string] (B1.topleft) to[out=down, in=up] (B1.bottomright);
  \obscureA{B}
    {
\draw[solid, red] (B1.topright) to[out=down, in=up] (B1.bottomleft);
    }

\draw[internal string] (B2.topleft) to[out=down, in=up] (B2.bottomright);
\obscureA{B}{
  \draw[solid, red] (B2.topright) to[out=down, in=up] (B2.bottomleft);
}

\draw[internal string] (B3.topleft) to[out=down, in=up] (B3.bottomright);
\obscureA{B}{
  \draw[solid, red] (B3.topright) to[out=down, in=up] (B3.bottomleft);
}

  \begin{scope}[internal string scope]



    \draw[red strand] (CylL.top) -- (CylL.bottom);
    \draw[red strand] (CylR.top) -- (CylR.bottom);
    
    \draw[red strand] (CylL.top)
        to[out=up, in=up, looseness=1.7] (B1.topleft);

    \draw[red strand] (CylR.top)
        to[out=up, in=up, looseness=1.7] (B1.topright);

    \draw[red strand] (CylL.bottom)
    to[out=down, in=down, looseness=1.7] (B3.bottomleft);

    \draw[red strand] (CylR.bottom)
    to[out=down, in=down, looseness=1.7] (B3.bottomright);

  \end{scope}
\end{tz}$};

\node (4) at (2.5,0)
{$\begin{tz}
  \node[Pants, bot] (P1) at (0,0) {};
  \node[Cap] at (P1.belt) {};

  \node[Cyl, supertall, bot, anchor=top] (CylL) at (P1.leftleg) {};

  \node[Copants, anchor=leftleg] (PB1) at (P1.rightleg) {};
  \node[Pants, bot, anchor=belt] (PB11) at (PB1.belt) {};

  \node[BraidB, anchor=topleft] (B1) at (PB11.leftleg) {};

  \node[Copants, anchor=leftleg] (PB2) at (B1.bottomleft) {};
  \node[Pants, bot, anchor=belt] (PB22) at (PB2.belt) {};

  \node[BraidB, anchor=topleft] (B2) at (PB22.leftleg) {};

  \node[Copants, anchor=leftleg] (PB3) at (B2.bottomleft) {};
  \node[Pants, bot, anchor=belt] (PB33) at (PB3.belt) {};

  \node[BraidB, anchor=topleft] (B3) at (PB33.leftleg) {};

  \node[Pants, bot, anchor=leftleg] (P2) at (PB1.rightleg) {};
  \node[Cap] at (P2.belt) {};

  \node[Cyl, supertall, bot, anchor=top] (CylR) at (P2.rightleg) {};

  \node[Copants, anchor=rightleg] (C1) at (B3.bottomleft) {};
  \node[Cup] at (C1.belt) {};

  \node[Copants, anchor=leftleg] (C2) at (B3.bottomright) {};
  \node[Cup] at (C2.belt) {};


  \draw[internal string] (B1.topleft) to[out=down, in=up] (B1.bottomright);
  \obscureA{B}
    {
\draw[solid, red] (B1.topright) to[out=down, in=up] (B1.bottomleft);
    }

\draw[internal string] (B2.topleft) to[out=down, in=up] (B2.bottomright);
\obscureA{B}{
  \draw[solid, red] (B2.topright) to[out=down, in=up] (B2.bottomleft);
}

\draw[internal string] (B3.topleft) to[out=down, in=up] (B3.bottomright);
\obscureA{B}{
  \draw[solid, red] (B3.topright) to[out=down, in=up] (B3.bottomleft);
}

  \begin{scope}[internal string scope]
         \draw (P1.rightleg) to[out=down, in=up] (PB1-belt.in-leftthird) to[out=down, in=up] (B1.topleft);
         \draw (P2.leftleg) to[out=down, in=up] (PB1-belt.in-rightthird) to[out=down, in=up] (B1.topright);
         \draw (B1.bottomleft) to[out=down, in=up] (PB2-belt.in-leftthird) to[out=down, in=up] (B2.topleft);
         \draw (B1.bottomright) to[out=down, in=up] (PB2-belt.in-rightthird) to[out=down, in=up] (B2.topright);
         \draw (B2.bottomleft) to[out=down, in=up] (PB3-belt.in-leftthird) to[out=down, in=up] (B3.topleft);
         \draw (B2.bottomright) to[out=down, in=up] (PB3-belt.in-rightthird) to[out=down, in=up] (B3.topright);
    \draw[red strand] (CylL.top) -- (CylL.bottom);
    \draw[red strand] (CylR.top) -- (CylR.bottom);
    
    \draw[red strand] (CylL.top)
        to[out=up, in=up, looseness=1.7] (PB1.leftleg);

    \draw[red strand] (CylR.top)
        to[out=up, in=up, looseness=1.7] (PB1.rightleg);

    \draw[red strand] (CylL.bottom)
    to[out=down, in=down, looseness=1.7] (B3.bottomleft);

    \draw[red strand] (CylR.bottom)
    to[out=down, in=down, looseness=1.7] (B3.bottomright);

  \end{scope}
\end{tz}$};

\node (5) at (4,0)
{$\begin{tz}
  \node[Pants, bot] (P1) at (0,0) {};
  \node[Cap] at (P1.belt) {};

  \node[Cyl, veryverytall, bot, height scale=1.725, anchor=top] (CylL) at (P1.leftleg) {};

  \node[Copants, anchor=leftleg] (PB1) at (P1.rightleg) {};
  \node[Pants, bot, height scale=1.3, anchor=belt] (PB11) at (PB1.belt) {};


  \node[Copants, anchor=leftleg] (PB2) at (PB11.leftleg) {};
  \node[Pants, bot, height scale=1.3, anchor=belt] (PB22) at (PB2.belt) {};


  \node[Copants, anchor=leftleg] (PB3) at (PB22.leftleg) {};
  \node[Pants, bot, height scale=1.3, anchor=belt] (PB33) at (PB3.belt) {};


  \node[Pants, bot, anchor=leftleg] (P2) at (PB1.rightleg) {};
  \node[Cap] at (P2.belt) {};

  \node[Cyl, veryverytall, bot, height scale=1.725, anchor=top] (CylR) at (P2.rightleg) {};

  \node[Copants, anchor=rightleg] (C1) at (PB33.leftleg) {};
  \node[Cup] at (C1.belt) {};

  \node[Copants, anchor=leftleg] (C2) at (PB33.rightleg) {};
  \node[Cup] at (C2.belt) {};

  \begin{scope}[internal string scope]
         \draw (P2.leftleg) to[out=down, in=up] (PB1-belt.in-rightthird) to [out=down, in=up, out looseness=1.7, in looseness=0.7] (PB11.leftleg);
         \draw (P1.rightleg) to[out=down, in=up] (PB1-belt.in-leftthird) to +(0,-0.2) to [out=down, in=up, out looseness=1.7, in looseness=1.2] (PB11.rightleg);
         \draw (PB11.rightleg) to[out=down, in=up] (PB2-belt.in-rightthird) to [out=down, in=up, out looseness=1.7, in looseness=0.7] (PB22.leftleg);
         \draw (PB11.leftleg) to[out=down, in=up] (PB2-belt.in-leftthird) to +(0,-0.2) to [out=down, in=up, out looseness=1.7, in looseness=1.2] (PB22.rightleg);
         \draw (PB22.rightleg) to[out=down, in=up] (PB3-belt.in-rightthird) to [out=down, in=up, out looseness=1.7, in looseness=0.7] (PB33.leftleg);
         \draw (PB22.leftleg) to[out=down, in=up] (PB3-belt.in-leftthird) to +(0,-0.2) to [out=down, in=up, out looseness=1.7, in looseness=1.2] (PB33.rightleg);
    \draw[red strand] (CylL.top) -- (CylL.bottom);
    \draw[red strand] (CylR.top) -- (CylR.bottom);
    
    \draw[red strand] (CylL.top)
        to[out=up, in=up, looseness=1.7] (PB1.leftleg);

    \draw[red strand] (CylR.top)
        to[out=up, in=up, looseness=1.7] (PB1.rightleg);

    \draw[red strand] (CylL.bottom)
    to[out=down, in=down, looseness=1.7] (PB33.leftleg);

    \draw[red strand] (CylR.bottom)
    to[out=down, in=down, looseness=1.7] (PB33.rightleg);

  \end{scope}
\end{tz}$};

\node (6) at (3.25,-2.5)
{$\begin{tz}
        \node[Cap, scale=5] (A) at (0,0) {};
        \node[Cup, scale=5] (B) at (0,0) {};
        \begin{scope}[internal string scope]
            \draw [red strand] (-0.17,-0.25) to [out=up, in=down] (0.17,0.08);
\draw [red strand] (0.17,-0.25) to [out=up, in=down] (-0.17,0.08);
\draw [red strand] (-0.17,0.08) to [out=up, in=down] (0.17,0.42);
\draw [red strand] (0.17,0.08) to [out=up, in=down] (-0.17,0.42);
\draw [red strand] (-0.17,0.42) to [out=up, in=down] (0.17,0.75);
\draw [red strand] (0.17,0.42) to [out=up, in=down] (-0.17,0.75);
\draw [red strand] (-0.5,-0.25) to (-0.5,0.75);
\draw [red strand] (0.5,-0.25) to (0.5,0.75);
\draw [red strand] (-0.5,0.75) to [out=up, in=up, looseness=1.2] (-0.17,0.75);
\draw [red strand] (0.17,0.75) to [out=up, in=up, looseness=1.2] (0.5,0.75);
\draw [red strand] (-0.5,-0.25) to [out=down, in=down, looseness=1.2] (-0.17,-0.25);
\draw [red strand] (0.17,-0.25) to [out=down, in=down, looseness=1.2] (0.5,-0.25);
        \end{scope}
\end{tz}$};

\begin{scope}[double arrow scope]
    \draw (1) -- node [left] {$\Znc^\prime(\epsilon^\dagger)$} (2);
    \draw (2) -- node [above] {$\sim$} (3);
    \draw[] ([xshift=0pt] 4.0) to node [above, inner sep=1pt] {$\Znc^\prime(\beta)^3$} (5);
    \draw (3) -- node [above] {$\Znc^\prime(\eta)^3$} (4);
    \draw (5) -- node [right] {$\Znc^\prime(\epsilon)^4$} (6);
\end{scope}
\end{tz}
\end{equation}
\end{exmp}

\begin{exmp}
    For the 0-framed Hopf link we have:

\scalecobordisms{0.6}

 \begin{equation}
\begin{tz}[xscale=2, yscale=2]
\node (1) at (1,-2)
{$\begin{tz}
        \node[Cap] (A) at (0,0) {};
        \node[Cup] (B) at (0,0) {};
        \node [Cobordism Bottom End 3D] (C) at (0,0) {};
\end{tz}$};

\node (2) at (0,-0.75)
{$\begin{tz}
        \node[Cap] (A) at (0,0) {};
        \node[Pants, anchor=belt] (B) at (A.center) {};
        \node[Copants, anchor=leftleg] (C) at (B.leftleg) {};
        \node[Cup] (D) at (C.belt) {};
        \draw[internal string] (B.leftleg)
        to[out=up, in=up, looseness=1.7] (B.rightleg);
        \draw[internal string] (B.leftleg)
        to[out=down, in=down, looseness=1.7] (B.rightleg);
        \node[Cap] (E) at (-1,0) {};
        \node[Pants, anchor=belt] (F) at (E.center) {};
        \node[Copants, anchor=leftleg] (G) at (F.leftleg) {};
        \node[Cup] (H) at (G.belt) {};
        \draw[internal string] (F.leftleg)
        to[out=up, in=up, looseness=1.7] (F.rightleg);
        \draw[internal string] (F.leftleg)
        to[out=down, in=down, looseness=1.7] (F.rightleg);
\end{tz}$};

\node (3) at (1,0.5)
{$\begin{tz}
  \node[Pants, bot] (P1) at (0,0) {};
  \node[Cap] at (P1.belt) {};

  \node[Cyl, tall, bot, anchor=top] (CylL) at (P1.leftleg) {};

  \node[BraidB, anchor=topleft] (B1) at (P1.rightleg) {};

  \node[BraidB, anchor=topright] (B2) at (B1.bottomright) {};

  \node[Pants, bot, anchor=leftleg] (P2) at (B1.topright) {};
  \node[Cap] at (P2.belt) {};

  \node[Cyl, tall, bot, anchor=top] (CylR) at (P2.rightleg) {};

  \node[Copants, anchor=rightleg] (C1) at (B2.bottomleft) {};
  \node[Cup] at (C1.belt) {};

  \node[Copants, anchor=leftleg] (C2) at (B2.bottomright) {};
  \node[Cup] at (C2.belt) {};

  \draw[internal string] (B1.topleft) to[out=down, in=up] (B1.bottomright);
  \obscureA{B1}
    {
      \draw[solid, red] (B1.topright) to[out=down, in=up] (B1.bottomleft);
    }

  \draw[internal string] (B2.topleft) to[out=down, in=up] (B2.bottomright);
  \obscureA{B2}
    {
      \draw[solid, red] (B2.topright) to[out=down, in=up] (B2.bottomleft);
    }

  \draw[red strand] (CylL.top) -- (CylL.bottom);
  \draw[red strand] (CylR.top) -- (CylR.bottom);

  \draw[red strand] (CylL.top)
      to[out=up, in=up, looseness=1.7] (B1.topleft);

  \draw[red strand] (CylR.top)
      to[out=up, in=up, looseness=1.7] (B1.topright);

  \draw[red strand] (CylL.bottom)
      to[out=down, in=down, looseness=1.7] (B2.bottomleft);

  \draw[red strand] (CylR.bottom)
      to[out=down, in=down, looseness=1.7] (B2.bottomright);
\end{tz}$};

\node (4) at (2.5,0)
{$\begin{tz}
  \node[Pants, bot] (P1) at (0,0) {};
  \node[Cap] at (P1.belt) {};

  \node[Cyl, lightlysupertall, bot, anchor=top] (CylL) at (P1.leftleg) {};

  \node[Copants, anchor=leftleg] (PB1) at (P1.rightleg) {};
  \node[Pants, bot, anchor=belt] (PB11) at (PB1.belt) {};

  \node[BraidB, anchor=topleft] (B1) at (PB11.leftleg) {};

  \node[Copants, anchor=leftleg] (PB2) at (B1.bottomleft) {};
  \node[Pants, bot, anchor=belt] (PB22) at (PB2.belt) {};

  \node[BraidB, anchor=topleft] (B2) at (PB22.leftleg) {};

  \node[Pants, bot, anchor=leftleg] (P2) at (PB1.rightleg) {};
  \node[Cap] at (P2.belt) {};

  \node[Cyl, lightlysupertall, bot, anchor=top] (CylR) at (P2.rightleg) {};

  \node[Copants, anchor=rightleg] (C1) at (B2.bottomleft) {};
  \node[Cup] at (C1.belt) {};

  \node[Copants, anchor=leftleg] (C2) at (B2.bottomright) {};
  \node[Cup] at (C2.belt) {};


  \draw[internal string] (B1.topleft) to[out=down, in=up] (B1.bottomright);
  \obscureA{B}
    {
\draw[solid, red] (B1.topright) to[out=down, in=up] (B1.bottomleft);
    }

\draw[internal string] (B2.topleft) to[out=down, in=up] (B2.bottomright);
\obscureA{B}{
  \draw[solid, red] (B2.topright) to[out=down, in=up] (B2.bottomleft);
}

  \begin{scope}[internal string scope]
         \draw (P1.rightleg) to[out=down, in=up] (PB1-belt.in-leftthird) to[out=down, in=up] (B1.topleft);
         \draw (P2.leftleg) to[out=down, in=up] (PB1-belt.in-rightthird) to[out=down, in=up] (B1.topright);
         \draw (B1.bottomleft) to[out=down, in=up] (PB2-belt.in-leftthird) to[out=down, in=up] (B2.topleft);
         \draw (B1.bottomright) to[out=down, in=up] (PB2-belt.in-rightthird) to[out=down, in=up] (B2.topright);
    \draw[red strand] (CylL.top) -- (CylL.bottom);
    \draw[red strand] (CylR.top) -- (CylR.bottom);
    
    \draw[red strand] (CylL.top)
        to[out=up, in=up, looseness=1.7] (PB1.leftleg);

    \draw[red strand] (CylR.top)
        to[out=up, in=up, looseness=1.7] (PB1.rightleg);

    \draw[red strand] (CylL.bottom)
    to[out=down, in=down, looseness=1.7] (B2.bottomleft);

    \draw[red strand] (CylR.bottom)
    to[out=down, in=down, looseness=1.7] (B2.bottomright);

  \end{scope}
  
\end{tz}$};

\node (5) at (4,0)
{$\begin{tz}
  \node[Pants, bot] (P1) at (0,0) {};
  \node[Cap] at (P1.belt) {};

  \node[Cyl, verytall, bot, height scale=1.535, anchor=top] (CylL) at (P1.leftleg) {};

  \node[Copants, anchor=leftleg] (PB1) at (P1.rightleg) {};
  \node[Pants, bot, height scale=1.3, anchor=belt] (PB11) at (PB1.belt) {};


  \node[Copants, anchor=leftleg] (PB2) at (PB11.leftleg) {};
  \node[Pants, bot, height scale=1.3, anchor=belt] (PB22) at (PB2.belt) {};



  \node[Pants, bot, anchor=leftleg] (P2) at (PB1.rightleg) {};
  \node[Cap] at (P2.belt) {};

  \node[Cyl, verytall, bot, height scale=1.535, anchor=top] (CylR) at (P2.rightleg) {};

  \node[Copants, anchor=rightleg] (C1) at (PB22.leftleg) {};
  \node[Cup] at (C1.belt) {};

  \node[Copants, anchor=leftleg] (C2) at (PB22.rightleg) {};
  \node[Cup] at (C2.belt) {};

  \begin{scope}[internal string scope]
         \draw (P2.leftleg) to[out=down, in=up] (PB1-belt.in-rightthird) to [out=down, in=up, out looseness=1.7, in looseness=0.7] (PB11.leftleg);
         \draw (P1.rightleg) to[out=down, in=up] (PB1-belt.in-leftthird) to +(0,-0.2) to [out=down, in=up, out looseness=1.7, in looseness=1.2] (PB11.rightleg);
         \draw (PB11.rightleg) to[out=down, in=up] (PB2-belt.in-rightthird) to [out=down, in=up, out looseness=1.7, in looseness=0.7] (PB22.leftleg);
         \draw (PB11.leftleg) to[out=down, in=up] (PB2-belt.in-leftthird) to +(0,-0.2) to [out=down, in=up, out looseness=1.7, in looseness=1.2] (PB22.rightleg);
    \draw[red strand] (CylL.top) -- (CylL.bottom);
    \draw[red strand] (CylR.top) -- (CylR.bottom);
    
    \draw[red strand] (CylL.top)
        to[out=up, in=up, looseness=1.7] (PB1.leftleg);

    \draw[red strand] (CylR.top)
        to[out=up, in=up, looseness=1.7] (PB1.rightleg);

    \draw[red strand] (CylL.bottom)
    to[out=down, in=down, looseness=1.7] (PB22.leftleg);

    \draw[red strand] (CylR.bottom)
    to[out=down, in=down, looseness=1.7] (PB22.rightleg);

  \end{scope}
\end{tz}$};

\node (6) at (3.25,-2.25)
{$\begin{tz}
        \node[Cap, scale=5] (A) at (0,0) {};
        \node[Cup, scale=5] (B) at (0,0) {};
        \begin{scope}[internal string scope]
            \draw [red strand] (-0.17,-0.25) to [out=up, in=down] (0.17,0.08);
\draw [red strand] (0.17,-0.25) to [out=up, in=down] (-0.17,0.08);
\draw [red strand] (-0.17,0.08) to [out=up, in=down] (0.17,0.42);
\draw [red strand] (0.17,0.08) to [out=up, in=down] (-0.17,0.42);
\draw [red strand] (-0.5,-0.25) to (-0.5,0.42);
\draw [red strand] (0.5,-0.25) to (0.5,0.42);
\draw [red strand] (-0.5,0.42) to [out=up, in=up, looseness=1.2] (-0.17,0.42);
\draw [red strand] (0.17,0.42) to [out=up, in=up, looseness=1.2] (0.5,0.42);
\draw [red strand] (-0.5,-0.25) to [out=down, in=down, looseness=1.2] (-0.17,-0.25);
\draw [red strand] (0.17,-0.25) to [out=down, in=down, looseness=1.2] (0.5,-0.25);
        \end{scope}
\end{tz}$};

 \node (9) at (0,-2)
{$\begin{tz}
        \node[Cup] (A) at (0,0) {};
        \node[Cap, bot] (C) at (0,0) {};
        \node[Cap] (B) at (0,-2*\cobheight) {};
        \node[Cup, bot] (D) at (0,-2*\cobheight) {};
    \end{tz}$};

\begin{scope}[double arrow scope]
    \draw (9) -- node [left] {$\Znc^\prime(\epsilon^\dagger)^2$} (2);
    \draw (2) -- node [left] {$\sim$} (3);
    \draw[] ([xshift=0pt] 4.0) to node [above, inner sep=1pt] {$\Znc^\prime({\beta})^2$} (5);
    \draw (3) -- node [above] {$\Znc^\prime(\eta)^2$} (4);
    \draw (1) -- node [above] {$\Znc^\prime(\mu^\dagger)$} (9);
    \draw (5) -- node [right] {$\Znc^\prime(\epsilon)^3$} (6);
\end{scope}
\end{tz}
\end{equation}

\end{exmp}

\begin{remark}
    Our recipe for decomposing $M$ in terms of 2-morphism generators works for the anomaly-free/oriented case. It is not clear how to extend this for $\Bordsig$ at this point, as composition is not fully worked out.
    However, we make the following observation:
    Our recipe computes the correct invariant when the signature $\sigma(L)=0$.
    If $\sigma(L)\neq 0$, we can obtain an $L^\prime$ with $\sigma(L^\prime)=0$ that gives rise to the same 3-manifold $M$ via surgery. This can be done with a series of Kirby 1-moves (blow-up/down). This essentially adds $|\sigma(L)|$-many $\pm$-framed unknots to $L$, unlinked to the rest of $L$. Applying our recipe on $L^\prime$, these framed unknots correspond to anomaly composites \ref{Anom}. This way, we see that $\Znc^\prime(M)$ is multiplied by $\xi^{\sigma(L)}$, and therefore again retrieve the correct formula for the invariant.
\end{remark}

\newpage

\section{Modified trace}\label{mod-trace}


The technology of modified traces developed in recent years \cite{geer_modified_2009,geer_generalized_2011,geer_traces_2012,geer_ambidextrous_2013,beliakova_modified_2020,geer_m-traces_2022,shibata_modified_2021} has been used to renormalize invariants that proved to be somewhat degenerate. Modified traces were used in the construction of non-semisimple TQFTs in \cite{de_renzi_3-dimensional_2022}, which are closely related to our construction in Section \ref{construction}. In this chapter, we examine how the TQFT $\Znc$ relates to modified traces. In particular, we will explain how to extract such a modified trace from $\Znc$.

We start with some basic definitions.


\begin{definition}\label{non-degen_trace}
    For $\C$ a finite abelian category over a field $k$, a \textit{trace} on a full subcategory $\mathcal{P}\subset\C$, is a family of $k$-linear maps $\{t_P\colon \Hom_\C(P,P)\to k\}_{P\in \mathcal{P}}$, satisfying the following conditions:

\begin{enumerate}
    \item (Cyclicity) For all $P,Q\in\cP$ and $f\colon P\to Q$, $g\colon Q\to P$, it is true that
    $$t_P(g\circ f)=t_Q(f\circ g).$$
    \item (Non-degeneracy) For all $P\in \cP$ and $X\in \C$, the pairing 
    $$t_P(-\circ -)\colon \Hom_\C(P,X)\times\Hom_\C(X,P)\to k,$$ is non-degenerate.
\end{enumerate}
\end{definition}

\begin{remark}
    The cyclicity condition is equivalent to saying that the family 
    
    $\{t_P\colon \Hom_\C(P,P)\to k\}_{P\in \mathcal{P}}$ is dinatural.
\end{remark}

If $\C$ is rigid monoidal, pivotal, and $\mathcal{P}=\text{Proj}(\C)$, the tensor ideal of projectives, we can define a \textit{modified} trace. 

First, for $X,Y\in \C$ and $f\in \text{End}(X\otimes Y)$, we define the left and right partial traces of $f$ as: 
$$\text{tr}_L(f)= (\text{ev}_X\otimes \id_Y)\circ(\id_{X^\lor}\otimes f)\circ(\text{coev}^R_X\otimes \id_Y),$$
$$\text{tr}_R(f)= (\id_Y \otimes\text{ev}^R_X)\circ(f\otimes \id_{X^\lor})\circ(\id_Y\otimes\text{coev}_X)$$

\begin{definition}

Now, the trace $t_P$ will be called \textit{modified} if it additionally satisfies the following conditions:

\begin{enumerate}
    \item (Right partial trace) For all $P\in \cP$, $X\in \C$ and $f\in \text{End}(P\otimes X)$,
    $$t_{P\otimes X}(f)=t_X(\text{tr}_R(f));$$
    \item (Left partial trace) For all $P\in \cP$, $X\in \C$ and $f\in \text{End}(X\otimes P)$,
    $$t_{X\otimes P}(f)=t_X(\text{tr}_L(f)).$$
\end{enumerate}
    
\end{definition}

Note that when $\C$ is ribbon, the conditions (1) and (2) above are equivalent \cite{geer_generalized_2011}.

\begin{remark}
    Modified traces can be defined in the non-pivotal setting \cite{shibata_modified_2023}.
\end{remark}

To see how modified traces appear through $\Znc$, it will be beneficial to think of the family $t_P$ as a $k$-linear map 
$$\int^{P\in\mathcal{P}} \hspace{-15pt} \Hom_\C(P,P)\to k,$$

as in our setting, the space $\displaystyle\int^{P\in\cP} \hspace{-20pt} \Hom_\C(P,P)$ is canonically isomorphic to $\mathcal{Z}(T^2)$. 
Moreover, by evaluating the non-compact TQFT $\mathcal{Z}$ on the solid torus, we obtain a  map $$\mathcal{Z}(\nu^\dagger\circ\epsilon)\colon\mathcal{Z}(T^2)\to k.$$

Arguing in terms of the compactified 2d TQFT $\mathcal{Z}(S^1\times -)$ and Lurie's non-compact cobordism hypothesis \cite{lurie_classification_2009}, this map is a trace map, part of the `Calabi-Yau' structure on $\mathcal{Z}(S^1)$.

What we prove in this chapter is that the map $\displaystyle\int^{P\in\cP} \hspace{-20pt} \Hom_\C(P,P)\to k$ we obtain as the composite of the above, is a modified trace. A central, non-trivial component of this statement concerns the origin of the canonical isomorphism 
$$\mathcal{Z}(T^2)\cong \displaystyle\int^{P\in\cP} \hspace{-20pt} \Hom_\C(P,P).$$

\subsection{1-dualizability in bicategories}

In order to explain why there is a canonical isomorphism $\mathcal{Z}(T^2)\cong \displaystyle\int^{P\in\cP} \hspace{-20pt} \Hom_\C(P,P)$, we will be concerned with the theory of 1-dualizable objects in bicategories. For this, we follow closely the exposition in \cite{pstrc_agowski_dualizable_2022}.

Let $\cM$ be a monoidal bicategory with monoidal unit $I$.

\begin{definition}
    A \textit{dual pair} in $\cM$ is a tuple ($L, R, e, c, \sigma, \tau$), where $L,R\in \cM$ are objects, $e\colon L\otimes R\to I$, $c\colon I\to R\otimes L$ are 1-morphisms and $\sigma, \tau$ are isomorphisms witnessing the `snake' identities:

\[\begin{tikzcd}
	&& {L\otimes(R\otimes L)} & {(L\otimes R)\otimes L} \\
	& {L\otimes I} &&& {I\otimes L} \\
	L &&&&& L
	\arrow[""{name=0, anchor=center, inner sep=0}, "\sim", from=1-3, to=1-4]
	\arrow["{e\otimes \id_L}", from=1-4, to=2-5]
	\arrow["{\id_L\otimes c}", from=2-2, to=1-3]
	\arrow["\sim", from=2-5, to=3-6]
	\arrow["\sim", from=3-1, to=2-2]
	\arrow[""{name=1, anchor=center, inner sep=0}, "{\id_L}"', from=3-1, to=3-6]
	\arrow["\sigma", shorten <=9pt, shorten >=9pt, Rightarrow, from=0, to=1]
\end{tikzcd}\]

\[\begin{tikzcd}
	&& {(R\otimes L)\otimes R} & {R\otimes(L\otimes R)} \\
	& {I\otimes R} &&& {R\otimes I} \\
	R &&&&& R
	\arrow[""{name=0, anchor=center, inner sep=0}, "\sim", from=1-3, to=1-4]
	\arrow["{\id_R\otimes e }", from=1-4, to=2-5]
	\arrow["{c\otimes\id_R}", from=2-2, to=1-3]
	\arrow["\sim", from=2-5, to=3-6]
	\arrow["\sim", from=3-1, to=2-2]
	\arrow[""{name=1, anchor=center, inner sep=0}, "{\id_R}"', from=3-1, to=3-6]
	\arrow["\tau", shorten <=9pt, shorten >=9pt, Rightarrow, from=0, to=1]
\end{tikzcd}\]

\end{definition}

Whenever there is no ambiguity, we will denote a dual pair $(L, R, e, c, \sigma, \tau)$ in $\mathcal{M}$ by $\langle L,R\rangle$.

Dual pairs are organized into a bicategory, $\text{DualPair}(\mathcal{M})$ - see \cite[Definition 3.8, Notation 3.9]{pstrc_agowski_dualizable_2022} for a complete description.
We are interested in the data of a 1-morphism in this bicategory, and this is what we explain now.

A morphism $(s, t, \gamma, \delta)$ between two dual pairs $(L, R, e, c, \sigma, \tau)$ and $(L^\prime, R^\prime, e^\prime, c^\prime, \sigma^\prime, \tau^\prime)$, consists of the data of:

\begin{itemize}
    \item Two 1-morphisms in $\cM$, $s\colon L\to L^\prime$ and $t\colon R\to R^\prime$;
    \item Two invertible 2-morphisms $\gamma, \delta$:
    

\[\begin{tikzcd}
	I && {R^\prime\otimes L^\prime} & {L^\prime\otimes R^\prime} && I \\
	& \delta &&& \gamma \\
	I && {R\otimes L} & {L\otimes R} && I
	\arrow["{c^\prime}", from=1-1, to=1-3]
	\arrow["{e^\prime}", from=1-4, to=1-6]
	\arrow["{\id_I}", from=3-1, to=1-1]
	\arrow["c"', from=3-1, to=3-3]
	\arrow["{t\otimes s}"', from=3-3, to=1-3]
	\arrow["{s\otimes t}", from=3-4, to=1-4]
	\arrow["e", from=3-4, to=3-6]
	\arrow["{\id_I}"', from=3-6, to=1-6]
\end{tikzcd}\]
\end{itemize}

Satisfying the following coherence:

\begin{equation}\label{cusp_coherence}
\begin{tikzcd}[cramped, sep=tiny]
	&& {L^\prime\otimes R^\prime\otimes L^\prime} &&&&&& {L^\prime\otimes R^\prime\otimes L^\prime} \\
	{L^\prime} &&&& {L^\prime} && {L^\prime} && {\sigma^\prime} && {L^\prime} \\
	& {s\otimes \delta} && {\gamma\otimes s} && {=} \\
	&& {L\otimes R\otimes L} &&&&&& {=} \\
	L && \sigma && L && L &&&& L
	\arrow["{e^\prime\otimes \id_{L^\prime}}", from=1-3, to=2-5]
	\arrow[from=1-9, to=2-11]
	\arrow["{\id_{L^\prime}\otimes c^\prime}", from=2-1, to=1-3]
	\arrow[from=2-7, to=1-9]
	\arrow["{\id_{L^\prime}}"', curve={height=24pt}, from=2-7, to=2-11]
	\arrow["{s\otimes t\otimes s}"{description}, from=4-3, to=1-3]
	\arrow["{e\otimes\id_L}"', from=4-3, to=5-5]
	\arrow["s", from=5-1, to=2-1]
	\arrow["{\id_L\otimes c}"', from=5-1, to=4-3]
	\arrow["{\id_L}"', curve={height=24pt}, from=5-1, to=5-5]
	\arrow["s"', from=5-5, to=2-5]
	\arrow["s", from=5-7, to=2-7]
	\arrow["{\id_L}"', from=5-7, to=5-11]
	\arrow["s"', from=5-11, to=2-11]
\end{tikzcd}
\end{equation}

\begin{equation}\label{cusp_coherence2}
\begin{tikzcd}[cramped, sep=tiny]
	&& {R^\prime\otimes L^\prime\otimes R^\prime} &&&&&& {R^\prime\otimes L^\prime\otimes R^\prime} \\
	{R^\prime} &&&& {R^\prime} && {R^\prime} && {\tau^\prime} && {R^\prime} \\
	& {\delta\otimes t} && {t\otimes \gamma} && {=} \\
	&& {R\otimes L\otimes R} &&&&&& {=} \\
	R && \tau && R && R &&&& R
	\arrow["{\id_{R^\prime}\otimes e^\prime}", from=1-3, to=2-5]
	\arrow["{\id_{R^\prime}\otimes e^\prime}", from=1-9, to=2-11]
	\arrow["{c^\prime\otimes \id_{R^\prime}}", from=2-1, to=1-3]
	\arrow["{c^\prime\otimes \id_{R^\prime}}", from=2-7, to=1-9]
	\arrow["{\id_{L^\prime}}"', curve={height=24pt}, from=2-7, to=2-11]
	\arrow["{t\otimes s\otimes t}"{description}, from=4-3, to=1-3]
	\arrow["{\id_R\otimes e}"', from=4-3, to=5-5]
	\arrow["t", from=5-1, to=2-1]
	\arrow["{c \otimes \id_R}"', from=5-1, to=4-3]
	\arrow["{\id_R}"', curve={height=24pt}, from=5-1, to=5-5]
	\arrow["t"', from=5-5, to=2-5]
	\arrow["t", from=5-7, to=2-7]
	\arrow["{\id_R}"', from=5-7, to=5-11]
	\arrow["t"', from=5-11, to=2-11]
\end{tikzcd}
\end{equation}

Clearly, an object is part of a dual pair if and only if it is 1-dualizable \cite[Proposition 3.6]{pstrc_agowski_dualizable_2022}.
However, it turns out that if we naively consider the forgetful 2-functor from the bicategory of dual pairs onto the 2-groupoid of 1-dualizable objects, this needn't be an equivalence. In order to fix this, we need to consider a class of dual pairs that satisfy additional coherence equations, the so-called \textit{swallowtail} axioms \cite[Definition 3.11]{pstrc_agowski_dualizable_2022}.
We call the bicategory of such dual pairs the \textit{bicategory of coherent dual pairs} in $\mathcal{M}$, which we denote by $\text{CohDualPair}(\mathcal{M})$. Since we are only concerned with coherent dual pairs, any mention of `dual pairs' after this point implicitly refers to coherent. 

We have the following results:


\begin{itemize}
    \item $\text{CohDualPair}(\mathcal{M})$ is a 2-groupoid \cite[Proposition 3.10]{pstrc_agowski_dualizable_2022};
    \item The data $(L, R, e, c, \sigma)$ can be uniquely completed to a dual pair \cite[Theorem 3.14]{pstrc_agowski_dualizable_2022};
    \item If $\left(\mathcal{M}^{\text{d}}\right)^{\sim}$ is the 2-groupoid of dualizable objects in $\mathcal{M}$, the forgetful 2-functor 
    \begin{equation}\label{DDat_equiv}
       F\colon\text{CohDualPair}(\mathcal{M})\to \left(\mathcal{M}^{\text{d}}\right)^{\sim}
    \end{equation}
    is an equivalence \cite[Theorem 3.16]{pstrc_agowski_dualizable_2022}.
\end{itemize}



We will denote an object of $\text{CohDualPair}(\mathcal{M})$ by $(L, R, e, c, \sigma)$, something enabled by the second result above.


\subsection{The canonical isomorphism}\label{canonical-iso}
In order to explain how to obtain the isomorphism $\mathcal{Z}(T^2)\cong \displaystyle\int^{P\in\cP} \hspace{-20pt} \Hom_\C(P,P)$, we will define two objects of $\text{CohDualPair}(\text{Bimod}_k)$. A 1-morphism between these objects gives rise to such an isomorphism. Through further analysis, we will see that an appropriately defined space of such 1-morphisms is contractible, and therefore a choice of such is unique, up to a unique 2-isomorphism.

We begin by defining the two objects of $\text{CohDualPair}(\text{Bimod}_k)$) in question. They both are associated to fully extended, non-compact 2d TQFTs $\mathcal{Z}_0$ and $\Znc(-\times S^1)$, so we denote them as such.
We emphasize that we do not prove this claim for $\mathcal{Z}_0$, and we simply present them as objects of $\text{CohDualPair}(\text{Bimod}_k)$.




\scalecobordisms{0.6}

\begin{table}[h!]
\centering
 \begin{tabular}{||c || c | c||} 
 \hline
 $\text{Bord}^{nc}_{2,1,0}$ & $\mathcal{Z}_0\colon \text{Bord}^{nc}_{2,1,0}\to \text{Bimod}_k$& $\Znc^\prime(-\times S^1)\colon \text{Bord}^{nc}_{2,1,0}\to \text{Bimod}_k$ \\ [0.5ex] 
 \hline\hline
 $pt_+$ & $\cP$ & $\cP$\\ [0.5ex] 
$pt_-$ & $\cP^{op}$ & $\cP$ \\ [0.5ex] 
$\begin{tz}[xscale=0.8,yscale=0.7]
\draw (0,0) to (0,0.5);
\draw (1,0) to (1,0.5);
\draw [black strand] (0,0.5) to [out=up, in=up, looseness=1.2] (1,0.5);
     \end{tz}$ & $\Hom_\C(-,-)\colon\cP\boxtimes\cP^{op}\to \Vecfd$ & $\Hom_\C(-\otimes -,\unit)^*\colon\cP\boxtimes\cP\to \Vecfd$ \\ [1ex] 
$\begin{tz}[xscale=0.8,yscale=0.7]
\draw (0,0) to (0,0.5);
\draw (1,0) to (1,0.5);
\draw [black strand] (0,0) to [out=down, in=down, looseness=1.2] (1,0);
     \end{tz}$ & $\Hom_\C(-,-)\colon\cP^{op}\boxtimes\cP\to \Vecfd$ & $\Hom_\C(-\otimes -,\unit)\colon\cP^{op}\boxtimes\cP^{op}\to \Vecfd$ \\ [1ex]
$
\begin{tz}[xscale=0.8,yscale=0.7]
\draw (0,0) to (0,0.5);
\draw (1,0) to (1,0.5);
\draw (2,0) to (2,0.5);
\draw [black strand] (0,0.5) to [out=up, in=up, looseness=1.2] (1,0.5);
\draw [black strand] (1,0) to [out=down, in=down, looseness=1.2] (2,0);
     \end{tz}
\xrightarrow{\text{cusp}}
\begin{tz}[xscale=0.8,yscale=0.7]
\draw (0,0) to (0,1);\end{tz}$ & Co-Yoneda & $\alpha_1$ \\ [1ex]
 \hline
 \end{tabular}
  \caption{Dualizability data for $\Znc_0(pt)$ and $\Znc^\prime(pt\times S^1)$.}
    \label{1-dual_data_table}
\end{table}



The isomorphism $\alpha_1$ is given by the the image of the cusp isomorphism under the non-compact TQFT $\Znc$.

In $\Bordncsig$, the cusp isomorphism is given by the following composite of 2-morphisms:

\begin{align}\label{cusp}
 \text{cusp} \quad&:=\quad
\begin{tz}
 \node[Pants, bot] (A) at (0,0) {};
 \node[Cyl, bot, anchor=top] (B) at (A.leftleg) {};
 \node[Copants, anchor=leftleg] (C) at (A.rightleg) {};
 \node[Cyl, top, bot, anchor=bottom] (D) at (C.rightleg) {};
 \node[Cap] (E) at (A.belt) {};
 \node[Cup] (F) at (C.belt) {};
 \selectpart[green] {(A-belt) (B-bottom) (C-belt) (D-top)};
\end{tz}
\,\, \Rarrow{\phileft} \,\,
\begin{tz}
        \node[Copants, top, bot] (A) at (0,0) {};
        \node[Pants, bot, anchor=belt] (B) at (A.belt) {};
        \node[Cap] (C) at (A.leftleg) {};
        \node[Cup] (D) at (B.rightleg) {};
        \selectpart[green] {(C-top) (A-rightleg) (A-leftleg) (A-belt)}
\end{tz}
\,\, \Rarrow{\breve{\lambda}} \,\,
\begin{tz}
        \node[Pants,  top, bot] (B) at (0,0) {};
        \node[Cup] (D) at (B.rightleg) {};
\end{tz}
\,\, \Rarrow{\rho} \,\,
\begin{tz}
 \node[Cyl, tall, top, bot] (A) at (0,0) {};
\end{tz}
\end{align}

As explained in \cite{bartlett_extended_2014}, a formula for the left co-unitor $\breve{\lambda}$ can be given via `calculus of mates', by the following composite of generators:

\begin{align}
 \breve{\lambda} \quad&:=\quad
\begin{tz}
        \node[Copants, bot] (A) at (0,0) {};
        \node[Cyl, top, bot, anchor=bottom] (B) at (A.rightleg) {};
        \node[Cap] (C) at (A.leftleg) {};
        \selectpart[green] {(B-top) (B-bottom)}
\end{tz}
\,\, \Rarrow{\lambda^\inv} \,\,
\begin{tz}
        \node[Copants, bot] (A) at (0,0) {};
        \node[Cyl, tall, bot, anchor=bottom] (B) at (A.rightleg) {};
        \node[Cap] (C) at (A.leftleg) {};
        \node[Pants, top, bot, anchor=rightleg] (D) at (B.top) {};
        \node[Cup] (E) at (D.leftleg) {};
        \selectpart[green] {(C) (D-leftleg)}
\end{tz}
\,\, \Rarrow{\mu} \,\,
\begin{tz} 
        \node[Pants, top, bot] (A) at (0,0) {};
        \node[Copants, bot, anchor=leftleg] at (A.leftleg) {};
\end{tz}
\,\, \Rarrow{\epsilon} \,\,
\begin{tz}
         \node[Cyl, tall, top, bot] (A) at (0,0) {};
\end{tz}
\end{align}

\scalecobordisms{0.3}

\begin{remark}\label{snake-nakayama}
    In Lincat, $\Znc(\begin{tz}
 \node[Pants, bot] (A) at (0,0) {};
 \node[Cyl, bot, anchor=top] (B) at (A.leftleg) {};
 \node[Copants, anchor=leftleg] (C) at (A.rightleg) {};
 \node[Cyl, top, bot, anchor=bottom] (D) at (C.rightleg) {};
 \node[Cap] (E) at (A.belt) {};
 \node[Cup] (F) at (C.belt) {};
\end{tz})(-)=
\displaystyle\int^{X\in\C}\hspace{-20pt}\Hom_\C(-\otimes X^\lor,1)^*\otimes X$, which is isomorphic to the Nakayama functor $\mathcal{N}$ \cite{fuchs_eilenberg-watts_2020}. It is known that trivializations of the Nakayama functor are in 1-1 correspondence with modified traces \cite{schweigert_trace_2023,shibata_modified_2023}.
Since the cusp isomorphism provides such a trivialization, it should not come as a surprise that it is related to the definition of the modified trace. We emphasize, however, that our method for obtaining a modified trace is independent of those in the aforementioned works.
\end{remark}

Specializing the discussion of the previous section, a morphism between two such dual pairs consists of the following data:

\begin{itemize}
    \item A functor $\cP\xrightarrow{s_0}\cP$;
    \item A functor $\cP^{op}\xrightarrow{t_0}\cP$;
    \item A natural isomorphism with components $\Hom_\C(Y,P)\xrightarrow{\gamma_0}\Hom_\C(P\otimes Y^\lor,\unit)^*$;
    \item A natural isomorphism with components ${\Hom_\C(Q,Y)\xrightarrow{\delta_0}\Hom_\C(Y^\lor\otimes Q,\unit)}$.
\end{itemize}

The ingredients to obtain the isomorphism $c_t\colon{\displaystyle\int^{P\in\cP} \hspace{-20pt} \Hom_\C(P,P)\xrightarrow{\sim}\Hom_\C(\coend,1)^*\left(=\Znc(T^2)\right)}$ are the following:

\begin{itemize}
    \item The canonical isomorphism $\mathcal{Z}_0(S^1)\cong \displaystyle\int^{P\in\cP} \hspace{-20pt} \Hom_\C(P,P)$, given by the cusp isomorphism under $\mathcal{Z}_0$ (co-Yoneda);
    \item An isomorphism $\mathcal{Z}_0(S^1)\cong \Znc^\prime(S^1\!\times S^1)$ coming from a 1-morphism ${\langle \cP,\cP^{op}\rangle\to \langle \cP,\cP\rangle}$ in $\text{CohDualPair}(\text{Bimod}_k)$;
    \item The identification $\Znc(T^2)\cong\Znc^\prime(T^2)$, as explained in Section \ref{Lincat-Bimod}.
\end{itemize}

Technically, $\gamma_0$ and $\delta_0$ cannot give us an isomorphism $\mathcal{Z}_0(S^1)\cong \Znc^\prime(S^1\!\times S^1)$, as they are part of a morphism between right dualizability data for $\cP$, and therefore we can't quite compose the two bimodules to get $\Znc(S^1\!\times S^1)$. However, since $\text{Bimod}_k$ is symmetric monoidal, we obtain the left dualizability data of $\cP$ from the right. So, to be more precise, to obtain the isomorphism in question, we also need part of the data of a 1-morphism $\langle \cP^{op},\cP\rangle\to \langle \cP,\cP\rangle$. In particular, a natual isomorphism 

\begin{equation}
    \gamma_1\colon\Hom_\C(Y,P)\to \Hom_\C(Y^\lor\otimes P,\unit)^*
\end{equation}

(notice the flipped entries in the target compared to $\gamma_0$).


To summarise, the isomorphism $c_t$ is given by the following composite:

\[\begin{tikzcd}
	{\displaystyle\int^{P\in\cP} \hspace{-20pt} \Hom_\C(P,P)} && {\displaystyle\int^{P,Q\in\cP} \hspace{-20pt} \Hom_\C(P,Q)\otimes\Hom_\C(Q,P)} \\
	\\
	{\Hom_\C(\coend,1)^*} && {\displaystyle\int^{P,Q\in\cP} \hspace{-20pt} \Hom_\C(P^\lor\otimes Q,\unit)^*\otimes\Hom_\C(P^\lor\otimes Q,\unit)}
	\arrow["\sim", from=1-1, to=1-3]
	\arrow["c_t"', from=1-1, to=3-1]
	\arrow["{\gamma_1\otimes\delta_0}", from=1-3, to=3-3]
	\arrow["\sim", from=3-3, to=3-1]
\end{tikzcd}\]

At this point, $c_t$ appears to depend on the choice of $\gamma_1$ and $\delta_0$. We dedicate the rest of this section on the following:

\begin{theorem}\label{iso_exist_canon}
   The isomoprhism $c_t$ exists and is canonical.
\end{theorem}

 To do this, we will show that there exists such a 1-morphism ${\langle \cP,\cP^{op}\rangle\to \langle \cP,\cP\rangle}$ of dual pairs, and that it is unique up to a contractible choice.


We start with the former.


Define a non-degenerate pairing, dinatural in $P\in\cP$ 

\begin{equation}
    t_P\colon\Hom(\unit,P)\otimes \Hom(P,\unit)\to k.
\end{equation}

From the results discussed in Chapter \ref{finiteproj}, to define $t_P$ it is sufficient to specify $t_{P_\unit}$, which we define by:

\begin{equation*}
    \setlength{\arraycolsep}{0pt}
\renewcommand{\arraystretch}{1.2}
  \begin{array}{ c c c }
    t_{P_\unit}\colon\Hom(\unit,P_\unit)\otimes \Hom(P_\unit,\unit) & {} \to {} & k \\
    \left(\eta_\unit,  \varepsilon_\unit\right)          & {} \mapsto {} & \mathscrsfs{D}^\inv
  \end{array}
\end{equation*}

The associated copairing 

\begin{equation}
    \Omega_P\colon k\to\Hom(P,\unit)\otimes\Hom(\unit,P),
\end{equation}

is determined by:

\begin{equation*}
    \setlength{\arraycolsep}{0pt}
\renewcommand{\arraystretch}{1.2}
  \begin{array}{ c c c c }
    \Omega_{P_\unit}\colon & {} k & {} \to {} & \Hom(P_\unit,\unit)\otimes\Hom(\unit,P_\unit)  \\
            & {}  1 & {} \mapsto {} & \mathscrsfs{D}\left(\varepsilon_\unit, \eta_\unit\right)
  \end{array}
\end{equation*}

\begin{remark}
    The data of a dinatural non-degenerate pairing follows from that of a non-degenerate trace (Definition \ref{non-degen_trace}) on the category. Dinaturality follows from the cyclicity condition. 
\end{remark}

The data of $t_P$ is equivalent to that of a natural isomorphism:

 \begin{equation}
     t^1_P\colon \Hom(\unit,P)\xrightarrow{\sim} \Hom(P,\unit)^*
 \end{equation}

Using $\Omega$ we have a formula for the inverse $\Omega^1:=(t^1)^\inv$ given by the following composite:

\begin{equation}\label{omega1def}
    \Hom(P,\unit)^*\xrightarrow{\Omega_P\otimes\id}
\Hom(\unit,P)\Hom(P,\unit)\otimes\Hom(P,\unit)^*
\xrightarrow{\id\otimes\text{ev}_{vect}}
\Hom(\unit,P)
\end{equation}

A useful fact we will use later on is the following:

\begin{lemma}\label{altmu}
    $\Znc(\mu)=EV_\unit\circ(\Omega^1\otimes\id_\unit)$
\end{lemma}

\begin{proof}
   We remind that $EV$ was defined in \ref{evaluation-counit} and $\Znc(\mu)$ at \ref{mudef}. The claim follows easily by unpacking the definitions.
\end{proof}

\begin{remark}
    In fact, the data of a natural isomorphism $t^1$ (and therefore also of $t$) is equivalent to an isomoprhism $\unit\xrightarrow{\sim} \mathcal{N}(\unit)$, where $\mathcal{N}$ is the Nakayama functor. As explained in \cite{shibata_modified_2023}, such isomorphisms are equivalent to trivializations of $\mathcal{N}$.
\end{remark}

We define the morphism between the two dual pairs to be 

\begin{equation}\label{morph_dual_pairs}
    (\id_\cP, (-)^\lor, t^1\circ(-)^\flat, (-)^\sharp).
\end{equation}

More explicitly,

\begin{itemize}
    \item $\cP\xrightarrow{\id_{\cP}}\cP$;
    \item $\cP^{op}\xrightarrow{(-)^\lor}\cP$;
    \item $\Hom_\C(-,-)\xrightarrow{t^1\circ(-)^\flat}\Hom_\C(-\otimes -,\unit)^*$;
    \item $\Hom_\C(-,-)\xrightarrow{(-)^\sharp}\Hom_\C(-\otimes -,\unit)$.
\end{itemize}


\begin{prop}
    The data $(\id_\cP, (-)^\lor, t^1\circ(-)^\flat, (-)^\sharp)$ is a well defined morphism of dual pairs.
\end{prop}


\begin{proof}

To prove this, we check the compatibility requirement \ref{cusp_coherence}. We remind that we do not need to check \ref{cusp_coherence2}, as there is a unique choice of $\tau$ satisfying it. 

The coherence \ref{cusp_coherence} in this case becomes:

\[\begin{tikzcd}[cramped, sep=tiny]
	&& {\C\boxtimes \C\boxtimes \C} &&&&&& {\C\boxtimes \C\boxtimes \C} \\
	\C &&&& \C && \C && {\alpha_1} && \C \\
	& {\id_\C\otimes (-)^\sharp} && {\left(t^1\circ(-)^\flat\right)\otimes \id_\C} && {=} \\
	&& {\C\boxtimes\C^{op}\boxtimes \C} &&&&&& {=} \\
	\C && {\text{Co-yoneda}} && \C && \C &&&& \C
	\arrow["{(2)}", from=1-3, to=2-5]
	\arrow["{(2)}", from=1-9, to=2-11]
	\arrow["{(1)}", from=2-1, to=1-3]
	\arrow["{(1)}", from=2-7, to=1-9]
	\arrow["{\Hom_\C(-,-)}"', curve={height=24pt}, from=2-7, to=2-11]
	\arrow["{\id_\C\otimes (-)^\lor\otimes \id_\C}"{description}, from=4-3, to=1-3]
	\arrow["{(3)}"', from=4-3, to=5-5]
	\arrow["{\id_\C}", from=5-1, to=2-1]
	\arrow["{(3)}"', from=5-1, to=4-3]
	\arrow["{\Hom_\C(-,-)}"', curve={height=24pt}, from=5-1, to=5-5]
	\arrow["{\id_\C}"', from=5-5, to=2-5]
	\arrow["{\id_\C}", from=5-7, to=2-7]
	\arrow["{\Hom_\C(-,-)}"', from=5-7, to=5-11]
	\arrow["{\id_\C}"', from=5-11, to=2-11]
\end{tikzcd}\]

where the arrows (1),(2),(3) are the bimodules:

\begin{enumerate}
    \item $=\Hom_\C(-,-)\otimes\Hom_\C(-\otimes -,\unit)$,
    \item $=\Hom_\C(-\otimes -,\unit)^*\otimes\Hom_\C(-,-)$,
    \item $=\Hom_\C(-,-)\otimes\Hom_\C(-,-)$.
\end{enumerate}

In other words, working in components $P,Q\in\cP$ 
($\Znc_0(\begin{tz}
 \node[Pants, bot] (A) at (0,0) {};
 \node[Cyl, bot, anchor=top] (B) at (A.leftleg) {};
 \node[Copants, anchor=leftleg] (C) at (A.rightleg) {};
 \node[Cyl, top, bot, anchor=bottom] (D) at (C.rightleg) {};
 \node[Cap] (E) at (A.belt) {};
 \node[Cup] (F) at (C.belt) {};
\end{tz})_P^Q$ 
and
$\Znc^\prime(\begin{tz}
 \node[Pants, bot] (A) at (0,0) {};
 \node[Cyl, bot, anchor=top] (B) at (A.leftleg) {};
 \node[Copants, anchor=leftleg] (C) at (A.rightleg) {};
 \node[Cyl, top, bot, anchor=bottom] (D) at (C.rightleg) {};
 \node[Cap] (E) at (A.belt) {};
 \node[Cup] (F) at (C.belt) {};
\end{tz})_P^Q$), we want to check the commutativity of the following diagram:

\[\begin{tikzcd}
	{\displaystyle\int^{Y\in\cP}\!\!\!\!\!\!\!\!\!\!\Hom_\C(Y,P)\otimes\Hom_\C(Q,Y)} && {\Hom_\C(Q,P) } \\
	\\
	{\displaystyle\int^{Y\in\cP}\!\!\!\!\!\!\!\!\!\!\Hom_\C\left(P\otimes Y^\lor,\unit\right)^*}\otimes\Hom_\C\left(Y^\lor\otimes Q,\unit\right)
	\arrow["{\text{co-Yoneda}}", from=1-1, to=1-3]
	\arrow["{t^1_{P\otimes Y^\lor}\circ(-)^\flat\otimes(-)^\sharp }"', from=1-1, to=3-1]
	\arrow["{\alpha_1}"', from=3-1, to=1-3]
\end{tikzcd}\]

\scalecobordisms{0.3}

We first compute $\alpha_1$. Note that by \ref{copower_def} it is true that:

\begin{equation}\label{int-hom}
    \displaystyle\int^{Y\in\cP}\!\!\!\!\!\!\!\!\!\!\Hom_\C(P\otimes Y^\lor,\unit)^*\otimes\Hom_\C(Q,Y)\cong 
\Hom_\C(Q,\displaystyle\int^{Y\in\cP}\!\!\!\!\!\!\!\!\!\!\Hom_\C(P\otimes Y^\lor,\unit)^*\otimes Y)
\end{equation}



This, together with an application of $(-)^\flat$ identifies $\Znc^\prime(\begin{tz}
 \node[Pants, bot] (A) at (0,0) {};
 \node[Cyl, bot, anchor=top] (B) at (A.leftleg) {};
 \node[Copants, anchor=leftleg] (C) at (A.rightleg) {};
 \node[Cyl, top, bot, anchor=bottom] (D) at (C.rightleg) {};
 \node[Cap] (E) at (A.belt) {};
 \node[Cup] (F) at (C.belt) {};
\end{tz})$ with $(\Znc(\begin{tz}
 \node[Pants, bot] (A) at (0,0) {};
 \node[Cyl, bot, anchor=top] (B) at (A.leftleg) {};
 \node[Copants, anchor=leftleg] (C) at (A.rightleg) {};
 \node[Cyl, top, bot, anchor=bottom] (D) at (C.rightleg) {};
 \node[Cap] (E) at (A.belt) {};
 \node[Cup] (F) at (C.belt) {};
\end{tz}))_*$: 

$$
\displaystyle\int^{Y\in\cP}\!\!\!\!\!\!\!\!\!\!\Hom_\C\left(P\otimes Y^\lor,\unit\right)^*\otimes\Hom_\C\left(Y^\lor\otimes Q,\unit\right)
\cong 
\Hom_\C(Q,\displaystyle\int^{Y\in\cP}\!\!\!\!\!\!\!\!\!\!\Hom_\C(P\otimes Y^\lor,\unit)^*\otimes Y).
$$



So we have:

\scalecobordisms{0.5}

\begin{equation}
\begin{tz}[xscale=1.4, yscale=1.5]

\node (1) at (-2,0)
{
$\begin{tikzpicture}
        \node[Pants, bot] (A) at (0,0) {};
 \node[Cyl, bot, anchor=top] (B) at (A.leftleg) {};
 \node[Copants, anchor=leftleg] (C) at (A.rightleg) {};
 \node[Cyl, top, bot, anchor=bottom] (D) at (C.rightleg) {};
 \node[Cap] (E) at (A.belt) {};
 \node[Cup] (F) at (C.belt) {};
 \selectpart[green] {(A-belt) (B-bottom) (C-belt) (D-top)};
\end{tikzpicture}$
};

\node (2) at (-2,-1.5)
{
$\begin{tikzpicture}
        \node[Copants, top, bot] (A) at (0,0) {};
        \node[Pants, bot, anchor=belt] (B) at (A.belt) {};
        \node[Cap] (C) at (A.leftleg) {};
        \node[Cup] (D) at (B.rightleg) {};
        \selectpart[green] {(A-belt) (B-leftleg) (D)}
\end{tikzpicture}$
};

\node (3) at (-2,-3)
{
$\begin{tikzpicture}
        \node[Copants, bot] (A) at (0,0) {};
        \node[Cyl, top, bot, anchor=bottom] (B) at (A.rightleg) {};
        \node[Cap] (C) at (A.leftleg) {};
        \selectpart[green] {(B-top) (B-bottom)}
\end{tikzpicture}$
};

\node (4) at (-2,-4.5)
{
$\begin{tikzpicture}
        \node[Copants, bot] (A) at (0,0) {};
        \node[Cyl, tall, bot, anchor=bottom] (B) at (A.rightleg) {};
        \node[Cap] (C) at (A.leftleg) {};
        \node[Pants, top, bot, anchor=rightleg] (D) at (B.top) {};
        \node[Cup] (E) at (D.leftleg) {};
        \selectpart[green] {(C) (D-leftleg)}
\end{tikzpicture}$
};

\node (5) at (-2,-6)
{
$\begin{tikzpicture}
        \node[Pants, top, bot] (A) at (0,0) {};
        \node[Copants, bot, anchor=leftleg] at (A.leftleg) {};
\end{tikzpicture}$
};

\node (6) at (-2,-7.5)
{
$\begin{tikzpicture}
        \node[Cyl, tall, top, bot] (A) at (0,0) {};
\end{tikzpicture}$
};

\node (7) at (2,0)
{
${\Hom_\C(Q,\displaystyle\int^{Y\in\cP}\!\!\!\!\!\!\!\!\!\!\Hom_\C(P\otimes Y^\lor,\unit)^*\otimes Y)}$
};

\node (8) at (2,-1.5)
{
$\Hom_\C(Q,\displaystyle\int^{Y\in\cP}\!\!\!\!\!\!\!\!\!\!\Hom_\C(P\otimes Y^\lor,\unit)^*\otimes Y)$
};

\node (9) at (2,-3)
{
$\Hom_\C(Q,\displaystyle\int^{Y\in\cP}\!\!\!\!\!\!\!\!\!\!\Hom_\C(P\otimes Y^\lor,\unit)^*\otimes Y)$
};

\node (10) at (2,-4.5)
{
$\Hom_\C(Q,\displaystyle\int^{Y\in\cP}\!\!\!\!\!\!\!\!\!\!\Hom_\C(P\otimes Y^\lor,\unit)^*\otimes\unit\otimes Y)$
};

\node (11) at (2,-6)
{
$\Hom_\C(Q,\displaystyle\int^{Y\in\cP}\!\!\!\!\!\!\!\!\!\!\!\!\!P\otimes Y^\lor\otimes Y)$
};

\node (12) at (2,-7.5)
{
$\Hom_\C\left(Q,P\right)$
};

\begin{scope}[double arrow scope]
    \draw (1) -- node[left] {$\phi_l$} (2);
    \draw (2) -- node[left] {$\rho$} (3);
    \draw (3) -- node[left] {$\lambda^\inv$} (4);
    \draw (4) -- node[left] {$\mu$} (5);
    \draw (5) -- node[left] {$\epsilon$} (6);
    \draw (7) -- node[left] {$\Znc(\phi_l)_*$} (8);
    \draw (8) -- node[left] {$\Znc(\rho)_*$} (9);
    \draw (9) -- node[left] {$\Znc(\lambda^\inv)_*$} (10);
    \draw (10) -- node[left] {$\Znc(\mu)_*$} (11);
    \draw (11) -- node[left] {$\Znc(\epsilon)_*$} (12);
\end{scope}
\end{tz}
\end{equation}

But as we know, $\Znc(\phileft)=\id$ and since we have been suppressing the unit (${P\otimes \unit\cong P}$), the right unitor $\Znc(\rho)$ is also the identity.

Therefore, the diagram whose commutativity we want to check becomes:

\[\begin{tikzcd}
	{\displaystyle\int^{Y\in\cP}\!\!\!\!\!\!\!\!\!\!\Hom_\C(Y,P)\otimes\Hom_\C(Q,Y)} & {\Hom_\C\left(Q,P\right)} \\
	{\displaystyle\int^{Y\in\cP}\!\!\!\!\!\!\!\!\!\!\Hom_\C\left(P\otimes Y^\lor,\unit\right)^*\otimes\Hom_\C\left(Y^\lor\otimes Q,\unit\right)} & {} \\
	{\Hom_\C(Q,\displaystyle\int^{Y\in\cP}\!\!\!\!\!\!\!\!\!\!\Hom_\C(P\otimes Y^\lor,\unit)^*\otimes Y)} \\
	{\Hom_\C(Q,\displaystyle\int^{Y\in\cP}\!\!\!\!\!\!\!\!\!\!\Hom_\C(P\otimes Y^\lor,\unit)^*\otimes\unit\otimes Y)} & {} \\
	{\Hom_\C(Q,\displaystyle\int^{Y\in\cP}\!\!\!\!\!\!\!\!\!\!\!\!\!P\otimes Y^\lor\otimes Y)} & {\Hom_\C\left(Q,P\right)}
	\arrow["{\text{co-Yoneda}}", from=1-1, to=1-2]
	\arrow["{t^1_{P\otimes Y^\lor}\circ(-)^\flat\otimes(-)^\sharp}"', from=1-1, to=2-1]
	\arrow["\id\otimes({\ref{int-hom}\circ(-)^\flat)}"', from=2-1, to=3-1]
	\arrow[no head, from=2-2, to=4-2]
	\arrow[shift left=3, no head, from=2-2, to=4-2]
	\arrow["{\Znc(\lambda^\inv)_*}"', from=3-1, to=4-1]
	\arrow["{\left(EV_\unit\circ(\Omega^1_{P\otimes Y^\lor}\otimes\id_\unit)\otimes\id_Y\right)_*}"', from=4-1, to=5-1]
	\arrow["{(\id_P\otimes\text{ev}_Y)_*}"', from=5-1, to=5-2]
\end{tikzcd}\]

This is indeed commutative, since

\begin{itemize}
    \item $(-)^\sharp$ cancels with $(-)^\flat$ applied right after;
    \item $t^1_{P\otimes Y^\lor}$ cancels with $\Omega^1_{P\otimes Y^\lor}$;
    \item The isomorphism \ref{int-hom} followed by $EV_1$ is the same as: 
    
    $${\displaystyle\int^{Y\in\cP}\!\!\!\!\!\!\!\!\!\!\Hom_\C(\unit,P\otimes Y^\lor)\otimes\Hom_\C(Q,\unit\otimes Y)\xrightarrow{\text{composition}}\displaystyle\int^{Y\in\cP}\!\!\!\!\!\!\!\!\!\!\Hom_\C(Q,P\otimes Y^\lor\otimes Y)};$$
    \item $(-)^\flat$ cancels with $\text{ev}_Y$;
    \item The co-Yoneda isomorphism is given by composition.
\end{itemize}

\end{proof}

As mentioned above, $\gamma_1=t^1\circ(-)^\natural$ is part of another morphism of dualizability data, ${\langle \mathcal{P}^{op},\mathcal{P}\rangle\to \langle \mathcal{P}, \mathcal{P}\rangle}$. The definition of this 1-morphism (as well as the proof of this fact) is entirely analogous to the above, and we omit it.

With this we have proved the existence of $c_t$. We proceed to explain why and in what sense the above choice of 1-morphism is canonical.

Let $\mathfrak{G}_1:=\Hom( \langle \mathcal{P},\mathcal{P}^{op}\rangle, \langle \mathcal{P}, \mathcal{P}\rangle)$ in $\text{CohDualPair}(\text{Bimod}_k)$ 
and $\mathfrak{G}_2:=\Hom(\mathcal{P},\mathcal{P})$ in $\left(\text{Bimod}_k^d\right)^\sim$.

Then, the equivalence \ref{DDat_equiv} gives an equivalence of 1-groupoids $F\colon \mathfrak{G}_1\to\mathfrak{G}_2$.

We are not interested in the entire category of morphisms $\mathfrak{G}_1$, but rather 1-morphisms lying over $\id_\mathcal{P}$. So, the correct notion to look at is that of the homotopy fiber $hoF^{-1}(\id_\mathcal{P})$.

\begin{definition}
    Let $\Phi\colon\mathcal{G}_1\to\mathcal{G}_2$ be a map of groupoids and $x\in\mathcal{G}_2$ an object. The \textit{homotopy fiber} $ho\Phi^{-1}(x)$ is the groupoid whose:

    \begin{itemize}
        \item Objects are pairs $(z,g)$ of an object $z\in\mathcal{G}_1$ and an isomorphism ${g\colon\Phi(z)\xrightarrow{\sim} x}$;
        \item Morphisms $(z_1,g_1)\to(z_2,g_2)$ are given by morphisms $h\colon z_1\to z_2$ in $\mathcal{G}_1$, such that $g_2\circ\Phi(h)=g_1$.
    \end{itemize}
    
\end{definition}

We are also interested in the strict fiber $stF^{-1}(\id_\mathcal{P})$.

\begin{definition}
    Let $\Phi\colon\mathcal{G}_1\to\mathcal{G}_2$ be a map of groupoids and $x\in\mathcal{G}_2$ an object. The \textit{strict fiber} $st\Phi^{-1}(x)$ is the groupoid whose:

    \begin{itemize}
        \item Objects are objects $z\in\mathcal{G}_1$ such that $\Phi(z) = x$;
        \item Morphisms $z_1\to z_2$ are given by morphisms $h\colon z_1\to z_2$ in $\mathcal{G}_1$, such that $\Phi(h)=\id_x$.
    \end{itemize}
    
\end{definition}


There is a canonical inclusion functor: 

\begin{equation*}
    \setlength{\arraycolsep}{0pt}
\renewcommand{\arraystretch}{1.2}
  \begin{array}{ c c c }
    \iota\colon st\Phi^{-1}(x) & {} \hookrightarrow {} & ho\Phi^{-1}(x) \\
    z           & {} \mapsto {} & (z,\id_x)
  \end{array}
\end{equation*}





\begin{lemma}\label{contractibility_hofib}
    Let $\Phi\colon \mathcal{G}_1\to\mathcal{G}_2$ be an equivalence of groupoids and $x\in\mathcal{G}_1$ an object. Then the homotopy fiber $ho\Phi^{-1}(x)$ is contractible.
\end{lemma}

\begin{proof}
    Let $(z_1,g_1), (z_2,g_2)\in ho\Phi^{-1}(x)$. Then, we have the following:

    \begin{equation*}
    \setlength{\arraycolsep}{0pt}
\renewcommand{\arraystretch}{1.2}
  \begin{array}{ c c c c c }
    \Hom_{\mathcal{G}_2}(x,x) & {} \cong {} & \Hom_{\mathcal{G}_2}(\Phi(z_1),\Phi(z_2)) & \cong & \Hom_{\mathcal{G}_1}(z_1,z_2) \\
    \id_{x}           & {} \longmapsto {} & g_2^{-1}\circ g_1 & \longmapsto & h
  \end{array}
\end{equation*}

In other words, there is a unique $h\in \Hom_{ho\Phi^{-1}(x)}(z_1,z_2)$.
\end{proof}

\begin{remark}
    The proof of Lemma \ref{contractibility_hofib} also works if we replace $ho\Phi^{-1}(x)$ by $st\Phi^{-1}(x)$, provided that the latter is not empty. This is true for $stF^{-1}(\id_\mathcal{P})$, thanks to the existence of the 1-morphism \ref{morph_dual_pairs}.
\end{remark}

\begin{corol}
    Since $F\colon \mathfrak{G}_1\to\mathfrak{G}_2$ is an equivalence, the homotopy fiber $hoF^{-1}(\id_\mathcal{P})$ is contractible.
\end{corol}

We will also show that $F\colon \mathfrak{G}_1\to\mathfrak{G}_2$ is an isofibration.

\begin{definition}
    Let $C,D$ be categories. An \textit{isofibration} is a functor $\Psi\colon C\to D$ such that for any object $y\in D$, and any isomorphism $g\colon\Psi(y)\xrightarrow{\sim} x$, there exists an isomorphism $h\colon y\xrightarrow{\sim} y^\prime$, such that $\Psi(h)=g$.
\end{definition}

\begin{lemma}
    $F\colon \mathfrak{G}_1\to\mathfrak{G}_2$ is an isofibration.
\end{lemma}

\begin{proof}
    As already mentioned, $stF^{-1}(\id_\mathcal{P})$ is not empty. Fix an object $z_0\in stF^{-1}(\id_\mathcal{P})$. Since $hoF^{-1}(\id_\mathcal{P})$ is contractible, for every object $(z,g)\in hoF^{-1}(\id_\mathcal{P})$ there exists a unique morphism ${h_0\colon (z,g)\to \iota(z_0)}$. In other words, $F$ is an isofibration.
\end{proof}

The following result is the reason why $F$ being an isofibration is important.

\begin{prop}
    If $F\colon \mathfrak{G}_1\to\mathfrak{G}_2$ is an isofibration, then the canonical inclusion 
    $\iota\colon stF^{-1}(\id_\mathcal{P})\hookrightarrow hoF^{-1}(\id_\mathcal{P})$ is an equivalence.
\end{prop}

\begin{proof}
    \begin{itemize}
        \item $F$ being an isofibration gives essential surjectivity.
        \item The fact that $\iota$ is fully faithful, i.e 
        $$\Hom_{stF^{-1}(\id_\mathcal{P})}(z_1,z_2)\cong \Hom_{hoF^{-1}(\id_\mathcal{P})}((z_1,\id),(z_2,\id)),$$ 
        is true by definition.
    \end{itemize}
\end{proof}


In other words, the 1-morphism of dualizability data is canonical, in the sense that it is unique, up to a unique 2-isomorphism. By extension, the choice of $c_t$ is canonical.

\subsection{The composite is a modified trace}

In this section, we prove the main result of this chapter. 

\begin{theorem}\label{mod-tr-thm}
    The map 
    
\begin{equation}
    t^\prime\colon \displaystyle\int^{P\in\cP} \hspace{-20pt} \Hom_\C(P,P)\to k,
\end{equation}
    
   given by $t^\prime:=\Znc(\nu^\dagger\circ\epsilon)\circ c_t$ is a modified trace.
\end{theorem}

Before we prove this, let us first understand $t^\prime$ explicitly. It is equal to the following composite:

\[\begin{tikzcd}
	{\displaystyle\int^{P\in\cP} \hspace{-20pt} \Hom_\C(P,P)} && {\displaystyle\int^{P,Q\in\cP} \hspace{-20pt} \Hom_\C(P,Q)\otimes\Hom_\C(Q,P)} \\
	&& {\displaystyle\int^{P,Q\in\cP} \hspace{-20pt} \Hom_\C(\unit,P^\lor\otimes Q)\otimes\Hom_\C(P^\lor\otimes Q,\unit)} \\
	&& {\displaystyle\int^{P,Q\in\cP} \hspace{-20pt} \Hom_\C(P^\lor\otimes Q,\unit)^*\otimes\Hom_\C(P^\lor\otimes Q,\unit)} \\
	&& {\displaystyle\int^{P,Q\in\cP} \hspace{-20pt} \Hom_\C(P^\lor\otimes Q,\unit)^*\otimes\Hom_\C(Q,P)} \\
	&& {\displaystyle\int^{P\in\cP} \hspace{-20pt} \Hom_\C(P^\lor\otimes P,\unit)^*} \\
	k && {\Hom_\C(\coend,\unit)^*}
	\arrow["\sim", from=1-1, to=1-3]
	\arrow["{t^\prime}"', from=1-1, to=6-1]
	\arrow["{(-)^\natural\otimes(-)^\sharp}", from=1-3, to=2-3]
	\arrow["{t^1_{P^\lor\otimes Q}\otimes \id}", from=2-3, to=3-3]
	\arrow["{\id\otimes (-)^\flat}", from=3-3, to=4-3]
	\arrow["\sim", from=4-3, to=5-3]
	\arrow["\sim", from=5-3, to=6-3]
	\arrow["{\text{ev}_{vect}(\varepsilon)}", from=6-3, to=6-1]
\end{tikzcd}\]

The maps comprising this composite are canonical. The first map is the cusp isomorphism of $\Znc_0$. It is followed by maps that are part of the morphisms of dualizability data described in the previous section. Lastly, the (solid torus) map $\Znc(\nu^\dagger\circ\epsilon)$ is applied. Note that the isomorphism $\Znc(T^2)\cong\Znc^\prime(T^2)$ is included in the composite. We could have defined the composite considering the actions of $\Znc^\prime(\epsilon)$ and $\Znc^\prime(\nu^\dagger)$ directly, as computed in Section \ref{3-mfld-inv}.


Let $f\in\Hom_\C(P,P)$. Since $(-)^\sharp$ cancels with $(-)^\flat$, and the co-Yoneda isomorphism cancels with its inverse, we compute that 

\begin{equation}
    t^\prime(f)=t_{P^\lor\otimes P}\left((\id_{P^\lor}\otimes f)\circ\text{coev}^R_P,\text{ev}_P\right).
\end{equation}

\tikzset{every node/.style={font=\tiny}}
\tikzset{morphism/.style={draw, fill=white}}
\tikzset{box/.style={draw, font=\small, node on layer=foreground, fill=white, inner sep=0pt, minimum height=0.51cm, minimum width=1.5cm}}

Graphically,

$$t^\prime(f)=t_{P^\lor\otimes P}\left(
\begin{tz}[xscale=0.4,yscale=0.7]
\node [halfbox] at (2,0.5) {$f$};
\draw (2,0) to (2,1.25) node [above] {$P$};
\draw (0,0) to (0,1.25) node [above] {$P^\lor$};
\draw [black strand] (0,0) to [out=down, in=down, looseness=1] (2,0);
\end{tz}
,
\begin{tz}[xscale=0.4,yscale=0.7]
\draw (2,0) node [below] {$P$} to (2,1.25);
\draw (0,0) node [below] {$P^\lor$} to (0,1.25);
\draw [black strand] (0,1.25) to [out=up, in=up, looseness=1] (2,1.25);
\end{tz}
\right)$$


We can now prove Theorem \ref{mod-tr-thm}:

\begin{proof}
    We check that $t^\prime$ satisfies the necessary conditions:
    
    \begin{itemize}
        \item The cyclicity condition is automatically satisfied, since $t^\prime$ is a map out of the coend $\displaystyle\int^{P\in\cP} \hspace{-20pt} \Hom_\C(P,P)$.
        
        \item Let $P\in\cP$, $X\in\C$ and $g_1\colon P\to X$, $g_2\colon X\to P$.
        $$t^\prime(g_2\circ g_1)=
        t_{P^\lor\otimes P}\left(\left(\id_{P^\lor}\otimes(g_2\circ g_1)\right)\circ\text{coev}^R_P,\text{ev}_P\right)$$

        which is non-degenerate due to the non-degeneracy of $t$.
        
        \item Since $\C$ is ribbon, it is sufficient to check the left partial trace property. In other words, we want to show that for $h\in \text{End}_\C( X\otimes P)$:

        $$t_{X\otimes P}^\prime(h)=t_P^\prime\left((\text{ev}_X\otimes \id_P)\circ(\id_{X^\lor}\otimes h)\circ(\text{coev}^R_X\otimes \id_P)\right).$$

        We once again remind that  $(X\otimes P)^\lor\cong P^\lor\otimes X^\lor$, with dualizability morphisms given by:

        $$\text{ev}_{X\otimes P}=\text{ev}_P\circ(\id_{P^\lor}\otimes \text{ev}_X\otimes\id_P),
        \quad
        \text{coev}_{X\otimes P}=(\id_{P}\otimes \text{coev}_X\otimes\id_{P^\lor})\circ\text{coev}_P$$

        Computing, we get:

        $t_{X\otimes P}^\prime(h)=t_{P^\lor\otimes X^\lor\otimes X\otimes P}\left((\id_{P^\lor}\otimes\id_{X^\lor}\otimes h)\circ(\text{coev}^R_{X\otimes P}),\text{ev}_{X\otimes P}\right)$

        Graphically,

        $$t_{X\otimes P}^\prime(h)=t_{P^\lor\otimes X^\lor\otimes X\otimes P}\left(
\begin{tz}[xscale=0.4,yscale=0.7]
\node [2halfbox] at (2.5,0.5) {$h$};
\draw (3,0) to (3,1.25) node [above] {$P$};
\draw (2,0) to (2,1.25) node [above] {$X$};
\draw (1,0) to (1,1.25) node [above] {$X^\lor$};
\draw (0,0) to (0,1.25) node [above] {$P^\lor$};
\draw [black strand] (0,0) to [out=down, in=down, looseness=1] (3,0);
\draw [black strand] (1,0) to [out=down, in=down, looseness=1] (2,0);
\end{tz}
,
\begin{tz}[xscale=0.4,yscale=0.7]
\draw (3,0) node [below] {$P$} to (3,1.25);
\draw (2,0) node [below] {$X$} to (2,1.25);
\draw (1,0) node [below] {$X^\lor$} to (1,1.25);
\draw (0,0) node [below] {$P^\lor$} to (0,1.25);
\draw [black strand] (0,1.25) to [out=up, in=up, looseness=1] (3,1.25);
\draw [black strand] (1,1.25) to [out=up, in=up, looseness=1] (2,1.25);
\end{tz}
\right)$$

        Using the morphism 
        $P^\lor\!\otimes\! X^\lor\!\!\otimes\! X\!\otimes \!P\xrightarrow{\id_{P^\lor}\otimes \text{ev}_X\otimes\id_P} P^\lor\!\!\otimes\! P$, dinaturality of $t$ gives us that the following square commutes:

        \tikzset{every node/.style={font=\small}}

\[\!\!\!\!\!\!\!\!\!\!\!\!\!\!\!\begin{tikzcd}
	{\Hom_\C(\unit,P^\lor\!\!\otimes\!\! X^\lor\!\!\otimes\! X\!\otimes\! P)\otimes\Hom_\C(P^\lor\!\!\otimes \!P,\unit)} & {\Hom_\C(\unit,P^\lor\!\!\otimes \!P)\otimes\Hom_\C(P^\lor\!\!\otimes \!P,\unit)} \\
	{\Hom_\C(\unit,P^\lor\!\!\otimes\!\! X^\lor\!\!\otimes\! X\!\otimes\! P)\otimes\Hom_\C(P^\lor\!\!\otimes\!\! X^\lor\!\!\otimes\! X\!\otimes\! P,\unit)} & k
	\arrow["{\text{ev}_X\otimes \id}", from=1-1, to=1-2]
	\arrow["{\id\otimes\text{ev}_X}"', from=1-1, to=2-1]
	\arrow["{t_{P^\lor\otimes P}}", from=1-2, to=2-2]
	\arrow["{t_{P^\lor\otimes X^\lor\otimes X\otimes P}}"', from=2-1, to=2-2]
\end{tikzcd}\]

    In other words, 
\tikzset{every node/.style={font=\tiny}}

    $$
    t_{P^\lor\otimes X^\lor\otimes X\otimes P}\left(
\begin{tz}[xscale=0.4,yscale=0.7]
\node [2halfbox] at (2.5,0.5) {$h$};
\draw (3,0) to (3,1.25) node [above] {$P$};
\draw (2,0) to (2,1.25) node [above] {$X$};
\draw (1,0) to (1,1.25) node [above] {$X^\lor$};
\draw (0,0) to (0,1.25) node [above] {$P^\lor$};
\draw [black strand] (0,0) to [out=down, in=down, looseness=1] (3,0);
\draw [black strand] (1,0) to [out=down, in=down, looseness=1] (2,0);
\end{tz}
,
\begin{tz}[xscale=0.4,yscale=0.7]
\draw (3,0) node [below] {$P$} to (3,1.25);
\draw (2,0) node [below] {$X$} to (2,1.25);
\draw (1,0) node [below] {$X^\lor$} to (1,1.25);
\draw (0,0) node [below] {$P^\lor$} to (0,1.25);
\draw [black strand] (0,1.25) to [out=up, in=up, looseness=1] (3,1.25);
\draw [black strand] (1,1.25) to [out=up, in=up, looseness=1] (2,1.25);
\end{tz}
\right)=
t_{P^\lor\otimes P}\left(
\begin{tz}[xscale=0.4,yscale=0.7]
\node [2halfbox] at (2.5,0.5) {$h$};
\draw (3,0) to (3,1.25) node [above] {$P$};
\draw (2,0) to (2,1);
\draw (1,0) to (1,1);
\draw (0,0) to (0,1.25) node [above] {$P^\lor$};
\draw [black strand] (0,0) to [out=down, in=down, looseness=1] (3,0);
\draw [black strand] (1,0) to [out=down, in=down, looseness=1] (2,0);
\draw [black strand] (1,1) to [out=up, in=up, looseness=1] (2,1);
\end{tz}
,
\begin{tz}[xscale=0.4,yscale=0.7]
\draw (3,0) node [below] {$P$} to (3,1.25);
\draw (0,0) node [below] {$P^\lor$} to (0,1.25);
\draw [black strand] (0,1.25) to [out=up, in=up, looseness=1] (3,1.25);
\end{tz}
\right),$$

which is the desired property.
    \end{itemize}
\end{proof}


\newpage

\printbibliography

\end{document}